\documentclass{amsart}

\usepackage[pdfborder={0 0 0}, draft=false, pagebackref, unicode=true]{hyperref}
\usepackage{esint,amssymb,mathtools,amsthm}

\def\multichoose#1#2{\ensuremath{\left(\kern-.3em\left(\genfrac{}{}{0pt}{}{#1}{#2}\right)\kern-.3em\right)}}

\newtheorem{thm}[equation]{Theorem} 
\newtheorem{lem}[equation]{Lemma}
\newtheorem{cor}[equation]{Corollary}

\theoremstyle{definition}
\newtheorem{defn}[equation]{Definition}
\newtheorem{rmk}[equation]{Remark}

\numberwithin{equation}{section}

\usepackage{cleveref}
\crefname{equation}{}{}
\crefname{section}{}{}
\crefname{lem}{}{}
\crefrangelabelformat{equation}{(#3#1#4--#5#2#6)}
\crefrangelabelformat{section}{#3#1#4--#5#2#6}

\newcommand\abs[2][empty]{\csname#1\endcsname \lvert{#2}\csname#1\endcsname\rvert}
\newcommand\doublebar[2][empty]{\csname#1\endcsname \lVert{#2}\csname#1\endcsname\rVert}

\newcommand\mat[1]{\boldsymbol{#1}}
\newcommand\arr[1]{\boldsymbol{\dot{#1}}}

\newcommand\dist{\mathop{\mathrm{dist}}\nolimits}
\newcommand\Div{\mathop{\mathrm{div}}\nolimits}
\newcommand\Tr{\mathop{\smash{\boldsymbol{\rlap{$\arr{\phantom{T}}$}\mathrm{Tr}}}\vphantom{T}}\nolimits}

\newcommand\Trace{\mathop{\mathrm{Tr}}\nolimits}
\newcommand\M{\mathop{\smash{\arr{\mathrm{M}}}\vphantom{M}}\nolimits}

\newcommand\vecM{\mathop{\smash{\vec{\mathrm{M}}}\vphantom{M}}\nolimits}
\newcommand\scalarM{\mathop{\mathrm{M}}\nolimits}

\newcommand\supp{\mathop{\mathrm{supp}}\nolimits}
\newcommand\diam{\mathop{\mathrm{diam}}\nolimits}

\newcommand\re{\mathop{\mathrm{Re}}\nolimits}

\newcommand\R{\mathbb{R}}

\newcommand\N{\mathbb{N}}
\newcommand\1{\mathbf{1}}
\newcommand\D{\mathcal{D}}
\newcommand\s{\mathcal{S}}

\newcommand\XX{\mathfrak{X}}
\newcommand\YY{\mathfrak{Y}}
\newcommand\DD{\mathfrak{D}}
\newcommand\NN{\mathfrak{N}}

\newcommand\pureH{\parallel}

\newcommand\dmn{{n+1}}
\newcommand\pdmn{{(n+1)}}
\newcommand\dmnMinusOne{n}

\begin{document}

\title[The $L^p$ Neumann problem]{The $L^p$ Neumann problem for higher order elliptic equations}

\author{Ariel Barton}
\address{Ariel Barton, Department of Mathematical Sciences,
			309 SCEN,
			University of Ar\-kan\-sas,
			Fayetteville, AR 72701}
\email{aeb019@uark.edu}

\begin{abstract}
We solve the Neumann problem in the half space~$\mathbb{R}^{n+1}_+$, for higher order elliptic differential equations with variable self-adjoint $t$-independent coefficients, and with boundary data in $L^p$, where $\max(1,\frac{2n}{n+2}-\varepsilon) < p < 2$. 

We also establish nontangential and area integral estimates on layer potentials with inputs in $L^p$ or $\dot W^{\pm1,p}$ for a similar range of~$p$, based on known bounds for $p\geq2$; in this case we may relax the requirement of self-adjointess.

\end{abstract}



\maketitle

\tableofcontents

\section{Introduction}

In this paper we study the Neumann boundary value problem and layer potentials for higher order elliptic differential operators of the form
\begin{equation}\label{eqn:divergence}
Lu = (-1)^m \sum_{\abs{\alpha}=\abs{\beta}=m} \partial^\alpha (A_{\alpha\beta} \partial^\beta u),\end{equation}
where $m$ is a positive integer, and 
with coefficients $\mat A$ that are $t$-independent in the sense that
\begin{equation}\label{eqn:t-independent}\mat A(x,t)=\mat A(x,s)=\mat A(x) \quad\text{for all $x\in\R^n$ and all $s$, $t\in\R$}.\end{equation}
Our coefficients may be merely bounded measurable in the $n$ horizontal variables. 
Second order operators with $t$-independent coefficients have been studied extensively; see, for example, \cite{
KenP93, AusT95,
KenR09,Rul07, 
AusAM10A, 
AlfAAHK11, Bar13, AusM14, HofKMP15A, HofKMP15B, HofMitMor15, BarM16A, MaeM16, AusS16, MaeM17
}. 
Higher order operators with $t$-independent coefficients have been studied by Hofmann and Mayboroda together with the author of the present paper in \cite{BarHM17,BarHM19A,BarHM19B,BarHM18,BarHM18p,Bar19p}.

Specifically, in \cite{BarHM18,BarHM18p}, we established the following result. Suppose that $L$ is an operator of the form~\eqref{eqn:divergence} associated to coefficients $\mat A$ that are $t$-independent, bounded, self-adjoint in the sense that $A_{\alpha\beta}=\overline{A_{\beta\alpha}}$ whenever $\abs\alpha=\abs\beta=m$, and satisfy the boundary G\r{a}rding inequality
\begin{equation}
\label{eqn:elliptic:slices}
\re\sum_{\abs\alpha=\abs\beta=m}\int_{\R^n}
\overline{\partial^\alpha\varphi(x,t)}\, A_{\alpha\beta}(x)\,\partial^\beta\varphi(x,t)\,dx \geq \lambda\doublebar{\nabla^m \varphi(\,\cdot\,,t)}_{L^2(\R^n)}^2
\end{equation}
for all $t\in \R$, all smooth test functions~$\varphi$ that are compactly supported in~$\R^\dmn$, and some $\lambda>0$ independent of $t$ and~$\varphi$. Then for every $\arr g\in L^2(\R^n)$ there is a solution~$w$, unique up to adding polynomials of degree~$m-1$, to the $L^2$ Neumann problem 
\begin{equation}
\label{eqn:neumann:regular:2}
\left\{\begin{gathered}\begin{aligned}
Lw&=0 \text{ in }\R^\dmn_+
,\\
\M_{\mat A}^+ w &\owns \arr g,
\end{aligned}\\
\doublebar{\mathcal{A}_2^+(t\nabla^m \partial_t w)}_{L^2(\R^n)} + \doublebar{\widetilde N_+(\nabla^{m}w)}_{L^2(\R^n)}
\leq C\doublebar{\arr g}_{L^2(\R^\dmnMinusOne)}
.\end{gathered}\right.\end{equation}

Here 
$\widetilde N_+$ is the modified nontangential maximal operator introduced in \cite{KenP93} and given (in $\R^\dmn_+$) by 
\begin{equation}
\label{dfn:NTM:modified:+}
\widetilde N_+ H(x) = \sup
\biggl\{\biggl(\fint_{B((y, s),{s}/2)} \abs{H(z,t)}^2\,dz\,dt\biggr)^{1/2}:
s>0,\>
\abs{x-y}< s
\biggr\}
.\end{equation}
$\mathcal{A}_2^+$ is the Lusin area integral given  by
\begin{equation}\label{dfn:lusin:+}\mathcal{A}_2^+ H(x) = \biggl(\int_0^\infty \int_{\abs{x-y}<t} \abs{H(y,t)}^2 \frac{dy\,dt}{t^\dmn}\biggr)^{1/2} 
.\end{equation} 
We adopt the convention that if a $t$ appears inside the argument of a tent space operator such as $\mathcal{A}_2^+$, then it denotes the $\pdmn$th coordinate function.

$\M_{\mat A}^+ w$ denotes the Neumann boundary values of $w$, and is the equivalence class of functions given by
\begin{equation}
\label{eqn:neumann:intro}
\arr g\in \M_{\mat A}^+ w \text{ if }
\sum_{\abs\gamma=m-1}\int_{\R^n} \partial^\gamma\varphi(x,0)\,g_\gamma(x)\,dx
= \sum_{\abs\alpha=\abs\beta=m} \int_{\R^\dmn_+} \partial^\alpha \varphi\,A_{\alpha\beta}\,\partial^\beta w\end{equation}
for all smooth test functions~$\varphi$ that are compactly supported in~$\R^\dmn$. An integration by parts argument shows that the right hand side depends only on the behavior of $\varphi$ near the boundary, and so $\M_{\mat A}^+ w$ is well defined as an operator on the space $\{\nabla^{m-1}\varphi\big\vert_{\partial\R^\dmn_+}:\varphi\in C^\infty_0(\R^\dmn)\}$. 

In the second order case $2m=2$, $\scalarM_{\mat A}^+ w$ consists of a single distribution; however, if $m\geq 2$, then by equality of mixed partials $\M_{\mat A}^+ w$ contains many arrays of distributions, and so is indeed an equivalence class. This is the formulation of Neumann boundary data used in \cite{Bar17,BarHM17,BarHM19B,BarHM18,BarHM18p,Bar19p}, and is closely related to the Neumann boundary values for the bilaplacian in \cite{CohG85,Ver05,She07B,MitM13B} and for general constant coefficient systems in \cite{MitM13A,Ver10,Ver14}. We refer the reader to \cite{BarM16B,BarHM17} for further discussion of higher order Neumann boundary data.

In the present paper we extend from results for $L^2$ boundary data to $L^p$ boundary data for appropriate $p<2$.
The first of the two main results of the present paper is the following theorem. (The second main result is Theorem~\ref{thm:potentials} below.)
\begin{thm}\label{thm:neumann:selfadjoint}
Suppose that $L$ is an elliptic operator of the 
form~\eqref{eqn:divergence} 
of order~$2m$ associated with coefficients $\mat A$ that are bounded, $t$-independent in the sense of formula~\eqref{eqn:t-independent}, satisfy the ellipticity condition~\eqref{eqn:elliptic:slices}, and are
self-adjoint in the sense that $A_{\alpha\beta}(x)=\overline{A_{\beta\alpha}(x)}$ for all $\abs\alpha=\abs\beta=m$ and all $x\in\R^n$.

Then there is a positive number $\varepsilon>0$, depending only on the dimension $\dmn$, the order $2m$ of the operator~$L$, the constant $\lambda$ in the bound~\eqref{eqn:elliptic:slices}, and~$\doublebar{\mat A}_{L^\infty(\R^n)}$, with the following significance. Suppose that $p$ satisfies 
\begin{equation}
\label{eqn:p:regular}
\max\biggl(1,\frac{2n}{n+2}-\varepsilon\biggr)<p<2
.\end{equation}
Then for every $\arr g\in L^p(\R^n)$, there is a solution $w$, unique up to adding polynomials of degree at most $m-1$, to the $L^p$-Neumann problem
\begin{equation}
\label{eqn:neumann:regular:p:selfadjoint}
\left\{\begin{gathered}\begin{aligned}
Lv&=0 \text{ in }\R^\dmn_+
,\\
\M_{\mat A}^+ v &\owns \arr g,
\end{aligned}\\
\doublebar{\mathcal{A}_2^+(t\nabla^m \partial_t w)}_{L^p(\R^n)} + \doublebar{\widetilde{N}_+(\nabla^{m}w)}_{L^p(\R^n)}
\leq C_p\doublebar{\arr g}_{L^p(\R^\dmnMinusOne)}
\end{gathered}\right.\end{equation}
where $C_p$ depends only on $p$, $n$, $m$, $\lambda$, and $\doublebar{\mat A}_{L^\infty(\R^n)}$.
\end{thm}

\subsection{The history of the Neumann problem}

We now discuss the history of the Neumann problem with boundary data in a Lebesgue space. The Neumann problem for the Laplacian with $L^p$ boundary data is traditionally the problem of finding a function $u$ such that
\begin{equation*}-\Delta u=0 \text{ in }\Omega,\quad \nu\cdot \nabla u=g\text{ on }\partial\Omega, \quad \doublebar{N_\Omega(\nabla u)}_{L^p(\partial\Omega)} \leq C \doublebar{g}_{L^p(\partial\Omega)}.\end{equation*}
Here $N_\Omega H(X)=\sup\{\abs{H(Y)}:\abs{X-Y}<2\dist(Y,\partial\Omega)\}$ is the standard nontangential maximal operator in~$\Omega$ and $\nu$ is the unit outward normal to~$\partial\Omega$. We observe that if $\Delta u=0$ in $\Omega$ and $u$ and $\partial\Omega$ are sufficiently smooth, then
\begin{equation*}\int_{\partial\Omega} \varphi\,\nu\cdot \nabla u\,d\sigma=\int_\Omega \nabla\varphi\cdot \nabla u\end{equation*}
and so the formulation of higher order Neumann boundary values~\eqref{eqn:neumann:intro} is in the spirit of the original harmonic Neumann problem.
The harmonic Neumann problem with $L^2$ boundary data was shown to be well posed in \cite{JerK81B} for all bounded Lipschitz domains~$\Omega$, and the Neumann problem with $L^p$ data for $p$ with $1<p<2+\varepsilon$ was shown to be well posed in \cite{DahK87}, where $\varepsilon>0$ depends on~$\Omega$. 

In \cite{KenP93}, the $L^p$ Neumann problem for more general second order equations
\begin{equation*}-\Div(\mat A\nabla u)=0 \text{ in }\Omega,\quad \nu\cdot\mat A\nabla u=g\text{ on }\partial\Omega, \quad \doublebar{\widetilde N_\Omega(\nabla u)}_{L^p(\partial\Omega)} \leq C \doublebar{g}_{L^p(\partial\Omega)}\end{equation*}
was shown to be well posed for $1<p< 2+\varepsilon$ in starlike Lipschitz domains with coefficients that are bounded, elliptic, real, symmetric, and independent of the radial coordinate. (This situation is very similar to the case of $t$-independent coefficients in the domain above a Lipschitz graph.) Here $\widetilde N_\Omega$ is a suitable modification of~$N_\Omega$; we remark that if $\Omega=\R^\dmn_+$ then $\widetilde N_\Omega=\widetilde N_+$ is given by formula~\eqref{dfn:NTM:modified:+}.

The case of real nonsymmetric $t$-independent coefficients was addressed in \cite{KenR09,Rul07}, in which the $L^p$ Neumann problem was solved in two dimensions for all $p$ with $1<p<1+\varepsilon$. (As shown in the appendix to \cite{KenR09}, there exist bounded real nonsymmetric $t$-independent coefficients for which the $L^2$ Neumann problem is ill posed.) Well posedness of the $L^2$ Neumann problem in the domain above a Lipschitz graph was shown to be stable under $t$-independent perturbation in \cite{AusAM10A} (and, under certain additional assumptions, in \cite{AlfAAHK11}), and some additional extrapolation type results were established in \cite{AusM14}.

The $L^p$ Neumann problem for a second order system of equations may be written as 
\begin{equation}
\label{eqn:neumann:2:system}
\left\{\begin{gathered}
(L\vec u)_j=\sum_{\alpha=1}^\dmn \sum_{\beta=1}^\dmn \sum_{k=1}^N \partial_{x_\alpha}(A_{\alpha\beta}^{jk}\partial_{x_\beta} u_k)=0 \text{ in }\Omega\text{ for }1\leq j\leq N,
\\
\vecM_{\mat A}^\Omega \vec u =\vec  g
,\quad \doublebar{N_\Omega(\nabla \vec u)}_{L^p(\partial\Omega)} \leq C_p\doublebar{\vec g}_{L^p(\partial\Omega)}
\end{gathered}\right.
\end{equation}
where $\vecM_{\mat A}^\Omega \vec u $ is given by
\begin{equation*}
\vecM_{\mat A}^\Omega \vec u=\vec g \quad \text{if} \quad
\sum_{j=1}^N\int_{\partial\Omega} \varphi_j\,g_{j}\,d\sigma
= \sum_{\alpha=1}^\dmn \sum_{\beta=1}^\dmn \sum_{j=1}^N \sum_{k=1}^N \int_\Omega \partial_{x_\alpha} \varphi_j\,A^{jk}_{\alpha\beta}\,\partial_{x_\beta} u_k
\end{equation*}
for all $\vec\varphi\in C^\infty_0(\R^\dmn)$.
As observed in \cite{She07B}, the traction boundary value problem for the Lam\'e system of elastostatics may be written in this form. The traction problem and the Neumann problem for the Stokes system, with boundary data in~$L^p(\partial\Omega)$, $2-\varepsilon<p<2+\varepsilon$, were shown to be well posed in \cite{DahKV88,FabKV88}; in \cite{She07B} Shen observed that their arguments apply to general second order  systems with real symmetric constant coefficients that satisfy an appropriate ellipticity condition. The traction boundary problem was shown to be well posed for $L^p$ boundary data, $1<p<2$, in \cite{DahK90}; their arguments relied on the fact that the Lam\'e system is defined in three dimensions, and applies to many more general three-dimensional (but not higher-dimensional) systems. In \cite{She07B}, Shen showed that if $\Omega\subset\R^\dmn$ is a Lipschitz domain with $\dmn\geq 4$, then for any second order elliptic system with real symmetric constant coefficients, the $L^p$ Neumann problem~\eqref{eqn:neumann:2:system} is well posed whenever $\frac{2n}{n+2}-\varepsilon <p<2$.

Turning to higher order equations, the $L^p$ Neumann problem for the biharmonic equation  is given by
\begin{equation*}(-\Delta)^2 u=0 \text{ in }\Omega,
\quad \vecM_\rho^\Omega u \owns \vec g, \quad \doublebar{N_\Omega(\nabla^2 u)}_{L^p(\partial\Omega)} \leq C_p \doublebar{\vec g}_{L^p(\partial\Omega)}
\end{equation*}
where
\begin{equation*}
\vecM_\rho^\Omega u\owns \vec g \quad\text{if}\quad
\int_{\Omega} \rho\Delta u\Delta \varphi +(1-\rho)\sum_{j,k=1}^\dmn \partial_{x_jx_k}u\,\partial_{x_jx_k}\varphi
= \int_{\partial\Omega} \vec g\cdot \nabla \varphi
\,d\sigma\end{equation*}
for all sufficiently smooth test functions~$\varphi$. The constant $\rho$ is called the Poisson ratio; we remark that an appropriate choice of coefficients $\mat A_\rho$ for the biharmonic equation yields that 
\begin{equation*}\M_{\mat A_\rho}^+u=\vecM_\rho^{\R^\dmn_+}u,\end{equation*}
where $\M_{\mat A}^+u$ is given by formula~\eqref{eqn:neumann:intro}. The biharmonic Neumann problem was shown to be well posed in bounded Lipschitz domains for for $p$ sufficiently close to~$2$ in \cite{Ver05} in dimension $\dmn\geq 2$, and for $p$ with $\frac{2n}{n+2}-\varepsilon<p<2$ in \cite{She07B} in dimension $\dmn\geq 4$. (The case of $C^1$ domains in $\R^2$ was considered earlier in \cite{CohG85}.)

Finally, the $L^2$ Neumann problem~\eqref{eqn:neumann:regular:2} was shown to be well posed in \cite{BarHM18,BarHM18p}. 

We observe that Shen's paper \cite{She07B} yields well posedness of the $L^p$ Neumann problem, for both the biharmonic equation and for constant coefficient second order systems, for the same range of $p$ as in our Theorem~\ref{thm:neumann:selfadjoint}. The present paper builds heavily on our preceding paper \cite{Bar19p}, and the techniques of \cite{Bar19p} owe much to the techniques of Shen. However, we remark that the arguments of \cite{Bar19p} are more closely related to those of Shen's earlier paper \cite{She06A} concerning the Dirichlet problem than to those of the later paper \cite{She07B} concerning the Neumann problem.

Our proof of Theorem~\ref{thm:neumann:selfadjoint} will involve well posedness of the subregular Neumann problem as established in \cite{Bar19p}. The subregular Neumann problem is the Neumann problem with boundary data in $\dot W^{-1,p}(\R^n)$. Here $\dot W^{-1,p}(\R^n)$ is the dual space to $\dot W^{1,p'}(\R^n)$,  the homogeneous Sobolev space in $\R^n$ with $\doublebar{\varphi}_{\dot W^{1,p'}(\R^n)}=\doublebar{\nabla_\pureH \varphi}_{L^{p'}(\R^n)}$, where $1/p+1/p'=1$ and where $\nabla_\pureH$ denotes the gradient in $\R^n$ (rather than $\R^\dmn$). We will discuss the main result of \cite{Bar19p} in more detail in Section~\ref{sec:neumann}. Here we only mention that subregular Neumann problems 
have received relatively little study; see \cite{Ver05} (the harmonic and biharmonic problems), \cite{AusM14,AusS16} (second order equations with $t$-independent coefficents), and \cite{BarHM18,BarHM18p,Bar19p} (higher order equations with $t$-independent coefficients).

The sharp range of $p$ for which a higher order $L^p$ Neumann problem is well posed is not known, even for special cases such as the biharmonic Neumann problem. However, results for related problems are somewhat suggestive. Specifically, the range of $p$ for which the biharmonic $\dot W^{1,p}$ Dirichlet problem 
\begin{equation*}(-\Delta)^2 u=0 \text{ in }\Omega,
\quad \nabla u=\vec f\text{ on }\partial\Omega,
\quad
\doublebar{N_\Omega(\nabla^2 u)}_{L^p(\partial\Omega)}\leq C\doublebar{\vec f}_{\dot W^{1,p}(\partial\Omega)}\end{equation*}
is well posed in all Lipschitz domains $\Omega\subset\R^\dmn$
is known to be $[6/5,2]$ in dimension $\dmn=4$,  is $[4/3,2]$ in dimension $\dmn=5$, $6$, or~$7$, and is known to be a subset of $[4/3,2]$ in dimension $\dmn\geq 8$.
See \cite{Ver90,She06A,She06B} for the well posedness results, \cite[Section~5]{DahKV86} and \cite[Theorem~10.7]{PipV92} for ill posedness results for the $L^{p'}$ Dirichlet problem, and \cite{KilS11A} for duality between the $L^{p'}$ and $\dot W^{1,p}$ Dirichlet problems for the bilaplacian.

This suggests that the $L^p$ Neumann problem~\eqref{eqn:neumann:regular:p:selfadjoint} is probably not well posed for the full range $1<p\leq 2$ in dimension $4$ and higher.

\subsection{Layer potentials}

We will prove Theorem~\ref{thm:neumann:selfadjoint} using the method of layer potentials. In the second order case $2m=2$, the double and single layer potentials are explicitly defined integral operators and are given by
\begin{align*}
\D^{\mat A}_\Omega f(X)&=\int_{\partial\Omega} \overline{\nu(Y)\cdot A^*(Y)\nabla E^{L^*}(Y,X)}\,f(Y)\,d\sigma(Y)
,\\
\s^L_\Omega g(X)&=\int_{\partial\Omega} E^L(X,Y)\,g(Y)\,d\sigma(Y)
\end{align*}
where $\nu$ is the unit outward normal vector to the domain~$\Omega\subset\R^\dmn$ and $E^L$ is the fundamental solution for the operator~$L$ in $\R^\dmn$. For reasonably well behaved domains $\Omega$ and inputs $f$ and~$g$, the outputs $\D^{\mat A}_\Omega f$ and $\s^L_\Omega g$ are locally Sobolev functions and satisfy $L(\D^{\mat A}_\Omega) =L(\s^L_\Omega g)=0$ away from~$\partial\Omega$. Certain other properties of layer potentials (in particular, the Green's formula and jump relations) are well known. It is possible to generalize layer potentials to the case of higher order operators. This may be done using integral kernels composed of various derivatives of higher order fundamental solutions (see \cite{Agm57,CohG83,CohG85,Ver05,She07B,MitM13A,MitM13B}) or by using the Lax-Milgram lemma to construct operators with appropriate properties (see \cite{Bar17,BarHM17} or Section~\ref{sec:dfn:potential} below).

If the operator $\arr f\to \M_{\mat A}^\Omega\D^{\mat A}_\Omega\arr f$ is invertible $\DD\to\NN$, for some function spaces $\DD$ and~$\NN$, where $\M_{\mat A}^\Omega$ is an appropriate Neumann boundary operator, then the function $u=\D^{\mat A}_\Omega((\M_{\mat A}^\Omega \D^{\mat A}_\Omega)^{-1}\arr g)$ is a solution to the Neumann problem 
\begin{equation*}Lu=0\text{ in }\Omega,\quad \M_{\mat A}^\Omega=\arr g\end{equation*}
with boundary data~$\arr g$. Furthermore, we may establish bounds on $u$ (such as the nontangential bound $\doublebar{\widetilde N_\Omega(\nabla^m u)}_{L^p(\partial\Omega)}\leq C_p\doublebar{\arr g}_\NN$) by establishing the corresponding bound $\doublebar{\widetilde N_\Omega(\nabla^m \D^{\mat A}_\Omega\arr f)}_{L^p(\partial\Omega)}\leq C_p\doublebar{\arr f}_\DD$ on the double layer potential.

Similarly, if $\arr g\to\Trace^\Omega\nabla^{m-1}\s^L_\Omega\arr g$ is invertible $\NN\to\DD$, then solutions to the Dirichlet problem
\begin{equation*}Lu=0\text{ in }\Omega,\quad \Trace^\Omega\nabla^{m-1}u=\arr f\end{equation*}
exist for all $\arr f\in\DD$.

This is the classic method of layer potentials. This method of constructing solutions to the Dirichlet or Neumann problem was used in  \cite{FabJR78,Ver84,DahK87,FabMM98,Zan00,May05} in the case of harmonic functions (that is, in the case $L=-\Delta$), in \cite{DahKV88,FabKV88,Fab88,Gao91,She07B} for second order constant coefficient systems, in \cite{AlfAAHK11,Bar13,HofMitMor15,HofKMP15B,BarM16A} for second order operators with variable $t$-inde\-pen\-dent coefficients, in \cite{Agm57,CohG83,CohG85,Ver05,She07B,MitM13A,MitM13B} for higher order operators with constant coefficients, and in \cite{BarHM18} for higher order operators with variable $t$-independent coefficients.

We will construct solutions to the problem~\eqref{eqn:neumann:regular:p:selfadjoint} by showing that $\M_{\mat A}^+\D^{\mat A}$ is invertible $\M_{\mat A}^+\D^{\mat A}:\dot W\!A^{1,p}_{m-1}(\R^n) \to (\dot W\!A^{0,p'}_{m-1}(\R^n))^*$, where $\dot W\!A^{j,p}_{m-1}(\R^n)$ is  the space of all arrays of functions in $\dot W^{j,p}(\R^n)$ (or $L^p(\R^n)$ if $j=0$) that can arise as the gradient $\nabla^{m-1}$ of order $m-1$ of a common function. If $m\geq 2$, then by equality of mixed partials $\dot W\!A^{j,p}_{m-1}(\R^n)$ is a proper subspace of $\dot W^{j,p}(\R^n)$. Then $(\dot W\!A^{0,p'}_{m-1}(\R^n))^*$ is a quotient space of~$L^p(\R^n)$ whose elements are equivalence classes of $L^p$ functions; in light of the definition~\eqref{eqn:neumann:intro} of Neumann boundary values, $\M_{\mat A}^+ \D^{\mat A}$ is naturally such an equivalence class.

Invertibility of the operator $\M_{\mat A}^+\D^{\mat A}:\dot W\!A^{1,p}_{m-1}(\R^n) \to (\dot W\!A^{0,p'}_{m-1}(\R^n))^*$ yields existence of solutions to the problem~\eqref{eqn:neumann:regular:p:selfadjoint} if in addition we have the estimates
\begin{equation*}\doublebar{\mathcal{A}_2^+(t\nabla^m \partial_t \D^{\mat A}\arr f)}_{L^p(\R^n)} + \doublebar{\widetilde{N}_+(\nabla^{m}\D^{\mat A}\arr f)}_{L^p(\R^n)}
\leq C_p\doublebar{\arr f}_{\dot W\!A^{1,p}_{m-1}(\R^\dmnMinusOne)}.\end{equation*}
Thus, we must establish these estimates for $p$ and $\mat A$ as in Theorem~\ref{thm:neumann:selfadjoint}. In fact, we will establish these estimates for $\mat A$ satisfying weaker conditions. (In particular, we will not need $\mat A$ to be self-adjoint to bound layer potentials.) Furthermore, we will establish estimates on the single layer potential and additional estimates on the double layer potential.

To discuss known results for higher order layer potentials, and to state the bounds on layer potentials to be established in this paper, we establish some terminology. We will consider coefficients $\mat A$ that satisfy the ellipticity condition
\begin{equation}
\label{eqn:elliptic}
\re\int_{\R^\dmn}
\sum_{\abs\alpha=\abs\beta=m}\overline{\partial^\alpha\varphi(x,t)}\, A_{\alpha\beta}(x)\,\partial^\beta\varphi(x,t)\,dx\,dt \geq \lambda\doublebar{\nabla^m \varphi}_{L^2(\R^\dmn)}^2
\end{equation}
for all $\varphi\in C^\infty_0(\R^\dmn)$ and some $\lambda>0$ independent of~$\varphi$. Observe that the condition~\eqref{eqn:elliptic} is weaker than the condition~\eqref{eqn:elliptic:slices} of Theorem~\ref{thm:neumann:selfadjoint}.

Meyers's reverse H\"older inequality for gradients of solutions is well known. In \cite{Cam80,AusQ00} it was generalized to operators of higher order. That is, if $L$ is an operator of order $2m$, $m\geq 1$, of the form~\eqref{eqn:divergence} and associated to bounded coefficients $\mat A$ that satisfy the ellipticity condition~\eqref{eqn:elliptic}, then there is a constant $\varepsilon>0$ such that if $2<p<2+\varepsilon$, then 
\begin{equation*}\biggl(\int_{B(X_0,r)} \abs{\nabla^m u}^p\biggr)^{1/p}
\leq \frac{c(0,L,p,2)}{r^{\pdmn(1/2-1/p)}} \biggl(\int_{B(X_0,2r)} \abs{\nabla^m u}^2\biggr)^{1/2}\end{equation*}
whenever $u\in \dot W^{m,2}(B(X_0,2r))$ and $Lu=0$ in $B(X_0,2r)$. 

In \cite[Section~9, Lemma~2]{FefS72} it was shown that if $L=-\Delta$, then the $L^2$ norm on the right hand side may be replaced by a $L^q$ norm for any $q<2$. The argument generalizes to arbitrary elliptic operators; see \cite[Theorem~24]{Bar16}. Furthermore, the Gagliardo-Nirenberg-Sobolev and Caccioppoli inequalities allow bounds on lower order derivatives to be established; see \cite[Section~4]{Bar16}.

Thus, we define 
$p_{j,L}^+$ to be the extended real number such that, whenever $p$ and $q$ satisfy $0<q<p<p_{j,L}^+$, there is a constant $c(j,L,p,q)<\infty$ such that 
\begin{equation}\label{eqn:Meyers}\biggl(\int_{B(X_0,r)} \abs{\nabla^{m-j}u}^p\biggr)^{1/p}
\leq \frac{c(j,L,p,q)}{r^{\pdmn(1/q-1/p)}} \biggl(\int_{B(X_0,2r)} \abs{\nabla^{m-j}u}^q\biggr)^{1/q}
\end{equation}
whenever $u\in \dot W^{m,2}(B(X_0,2r))$ and $Lu=0$ in $B(X_0,2r)$. 
We define $p_{j,L}^-$ by
\begin{equation}
\label{eqn:Meyers:p-}
\frac{1}{p_{j,L}^-}+\frac{1}{p_{j,L}^+}=1
.\end{equation}
By the results mentioned above, $p_{j,L}^+$ exists whenever $0\leq j\leq m$.
By \cite[Theorem~49]{AusQ00}, \cite[Section~4]{Bar16}, and \cite[Propositions~3.3 and~3.6]{Bar19p}, if $\mat A$ is bounded, $t$-inde\-pen\-dent in the sense of formula~\eqref{eqn:t-independent}, and elliptic in the sense of formula~\eqref{eqn:elliptic}, then there are numbers $\varepsilon>0$ and $\widetilde \varepsilon>0$, depending only on the order $2m$ of the operator~$L$, the ambient dimension $\dmn$, the number $\lambda$ in the ellipticity condition~\eqref{eqn:elliptic}, and the norm $\doublebar{\mat A}_{L^\infty(\R^n)}$ of the coefficients, such that the numbers $p_{j,L}^+$ satisfy
\begin{equation*}
\left\{\begin{aligned}
p_{0,L}^+&=\infty
, & p_{1,L}^+&=\infty 
&&\text{ if }\dmn=2,
\\
p_{0,L}^+&\geq 2+\varepsilon
, & p_{1,L}^+&=\infty 
&&\text{ if }\dmn=3,\\
p_{0,L}^+&\geq 2+\varepsilon
, & p_{1,L}^+&\geq\frac{2n}{n-2}+\varepsilon
&&\text{ if }\dmn\geq 4.
\end{aligned}\right.\end{equation*}
Therefore, there is a $\widetilde\varepsilon>0$ depending only on $n$ and $\varepsilon$ such that
\begin{equation}\label{eqn:Meyers:bound}
\left\{\begin{aligned}
p_{0,L}^-&=1
, & p_{1,L}^-&=1 
&&\text{ if }\dmn=2,
\\
p_{0,L}^-&\leq 2-\widetilde\varepsilon
, & p_{1,L}^-&=1 
&&\text{ if }\dmn=3,\\
p_{0,L}^-&\leq 2-\widetilde\varepsilon
, & p_{1,L}^-&\leq\frac{2n}{n+2}-\widetilde\varepsilon
&&\text{ if }\dmn\geq 4.
\end{aligned}\right.\end{equation}

\begin{rmk}\label{rmk:Meyers} 
If $p<2+\varepsilon$, or if $p<\infty$ and $\dmn=2$, then by again by \cite[Section~4]{Bar16} and \cite[Section~3]{Bar19p}, the numbers $c(0,L,p,q)$ in the bound~\eqref{eqn:Meyers} may be bounded by constants depending only on $p$, $q$ and the standard parameters $m$, $\dmnMinusOne$, $\lambda$, and $\doublebar{\mat A}_{L^\infty}$. The same is true of the numbers $c(1,L,p,q)$ if 
$\dmn\leq 3$ and $p<\infty$ or $\dmn\geq 4$ and $p<\frac{2n}{n-2}+\varepsilon$.
\end{rmk}

We may now discuss old and new bounds on layer potentials.
In \cite{BarHM17,BarHM19A,BarHM18p,Bar19p}, Hofmann, Mayboroda and the author of the present paper showed that if $L$ is an operator of the form~\eqref{eqn:divergence} associated to bounded elliptic $t$-independent coefficients, then there is a $\varepsilon>0$ such that
\begin{align}
\label{eqn:S:N:+}
\doublebar{\widetilde N_*(\nabla^m\s^L\arr g)}_{L^p(\R^n)}
&\leq C(0,L,p)\doublebar{\arr g}_{L^p(\R^n)}, & 2-\varepsilon&< p<p_{0,L}^+
,\allowdisplaybreaks
\\
\label{eqn:D:N:+}
\doublebar{\widetilde N_*(\nabla^m\D^{\mat A}\arr \varphi)}_{L^p(\R^n)}
&\leq C(0,L,p)\doublebar{\arr \varphi}_{\dot W\!A^{1,p}_{m-1}(\R^n)}, & 2-\varepsilon&< p<p_{0,L}^+
,\allowdisplaybreaks
\\
\label{eqn:S:lusin:+}
\doublebar{\mathcal{A}_2^*(t\nabla^m\partial_t\s^L\arr g)}_{L^p(\R^n)}
&\leq C(1,L,p)\doublebar{\arr g}_{L^p(\R^n)}, & 2-\varepsilon&< p<p_{1,L}^+
,\allowdisplaybreaks
\\
\label{eqn:D:lusin:+}
\doublebar{\mathcal{A}_2^*(t\nabla^m\partial_t\D^{\mat A}\arr \varphi)}_{L^p(\R^n)}
&\leq C(1,L,p)\doublebar{\arr \varphi}_{\dot W\!A_{m-1}^{1,p}(\R^n)}, & 2&\leq p<p_{1,L}^+
,\allowdisplaybreaks
\\
\label{eqn:S:lusin:rough:+}
\doublebar{\mathcal{A}_2^*(t\nabla^m\s^L_\nabla\arr h)}_{L^p(\R^n)}
&\leq C(1,L,p)\doublebar{\arr h}_{L^p(\R^n)}, & 2-\varepsilon&< p<p_{1,L}^+
,\allowdisplaybreaks
\\
\label{eqn:D:lusin:rough:+}
\doublebar{\mathcal{A}_2^*(t\nabla^m\D^{\mat A}\arr f)}_{L^p(\R^n)}
&\leq C(1,L,p)\doublebar{\arr f}_{\dot W\!A_{m-1}^{0,p}(\R^n)}, & 2&\leq p<p_{1,L}^+
,\allowdisplaybreaks
\\
\label{eqn:S:N:rough:+}
\doublebar{\widetilde N_*(\nabla^{m-1}\s^L_\nabla\arr h)}_{L^p(\R^n)}
&\leq C(1,L,p)\doublebar{\arr h}_{L^p(\R^n)}, & 2-\varepsilon&< p<p_{1,L}^+
,\allowdisplaybreaks
\\ 
\label{eqn:D:N:rough:+}
\doublebar{\widetilde N_*(\nabla^{m-1}\D^{\mat A}\arr f)}_{L^p(\R^n)}
&\leq C(1,L,p)\doublebar{\arr f}_{\dot W\!A_{m-1}^{0,p}(\R^n)}, & 2-\varepsilon&< p<p_{1,L}^+
\end{align}
where $p_{j,L}^+$ is as in the bound~\eqref{eqn:Meyers}, and $C(j,L,p)$ is a constant depending only on $m$, $n$, $\lambda$, $\doublebar{\mat A}_{L^\infty}$, $p$, and the number $c(j,L,p,2)$ in the bound~\eqref{eqn:Meyers}. These bounds played a crucial role in solving the $L^2$ Neumann problem~\eqref{eqn:neumann:regular:2} (and the subregular problem of \cite{Bar19p}).

Here
\begin{align}
\label{dfn:NTM:modified:*}
\widetilde N_* H(x) & = \sup
\biggl\{\biggl(\fint_{B((y, s),\abs{s}/2)} \abs{H(z,t)}^2\,dz\,dt\biggr)^{1/2}:
s\in\R,\>
\abs{x-y}< \abs{s}
\biggr\}
,\\
\label{dfn:lusin:*}
\mathcal{A}_2^* H(x) &= \biggl(\int_{-\infty}^\infty \int_{\abs{x-y}<\abs{t}} \abs{H(y,t)}^2 \frac{dy\,dt}{\abs{t}^\dmn}\biggr)^{1/2} \end{align}
are two-sided analogues of the nontangential and area integral operators of formulas~\eqref{dfn:NTM:modified:+} and~\eqref{dfn:lusin:+}.

The second of the two main results of the present paper is the following theorem, in which we expand the range of the parameter~$p$ in the bounds \cref{eqn:D:N:rough:+,eqn:S:N:rough:+,eqn:D:lusin:rough:+,eqn:S:lusin:rough:+,eqn:D:N:+,eqn:S:N:+,eqn:D:lusin:+,eqn:S:lusin:+} to include more values below~$2$.

\begin{thm}\label{thm:potentials} Suppose that $L$ is an operator of the form~\eqref{eqn:divergence} of order~$2m$ associated with bounded coefficients $\mat A$ that are $t$-independent in the sense of formula~\eqref{eqn:t-independent} and satisfy the ellipticity condition~\eqref{eqn:elliptic} for some $\lambda>0$.

Then the double and single layer potentials $\D^{\mat A}$, $\s^L$ and $\s^L_\nabla$, originally defined as in Section~\ref{sec:dfn:potential} below, extend by density to operators that satisfy the following bounds for all $p$ in the given ranges and all inputs $\arr f$, $\arr g$, $\arr h$, and $\arr\varphi$ in the indicated spaces.
\begin{align}
\label{eqn:S:N:intro}
\doublebar{\widetilde N_*(\nabla^m\s^L\arr g)}_{L^p(\R^n)}
&\leq C(1,L^*,p')\doublebar{\arr g}_{L^p(\R^n)}, & p_{1,L^*}^-&< p<2 
,\\ 
\label{eqn:D:N:intro}
\doublebar{\widetilde N_*(\nabla^m\D^{\mat A}\arr \varphi)}_{L^p(\R^n)}
&\leq C(1,L^*,p')\doublebar{\arr \varphi}_{\dot W\!A^{1,p}_{m-1}(\R^n)}, & p_{1,L^*}^-&<p<2 
,\allowdisplaybreaks
\\
\label{eqn:S:lusin:intro}
\doublebar{\mathcal{A}_2^*(t\nabla^m\partial_t\s^L\arr g)}_{L^p(\R^n)}
&\leq C(1,L^*,p')\doublebar{\arr g}_{L^p(\R^n)}, & p_{1,L^*}^-&< p<2 
,\allowdisplaybreaks
\\
\label{eqn:D:lusin:intro}
\doublebar{\mathcal{A}_2^*(t\nabla^m\partial_t\D^{\mat A}\arr \varphi)}_{L^p(\R^n)}
&\leq C(1,L^*,p')\doublebar{\arr \varphi}_{\dot W\!A_{m-1}^{1,p}(\R^n)}, & p_{1,L^*}^-&< p<2 
,\allowdisplaybreaks
\\
\label{eqn:S:lusin:rough:intro}
\doublebar{\mathcal{A}_2^*(t\nabla^m\s^L_\nabla\arr h)}_{L^p(\R^n)}
&\leq C(0,L^*,p')\doublebar{\arr h}_{L^p(\R^n)}, & p_{0,L^*}^-&< p<2 
,\allowdisplaybreaks
\\
\label{eqn:D:lusin:rough:intro}
\doublebar{\mathcal{A}_2^*(t\nabla^m\D^{\mat A}\arr f)}_{L^p(\R^n)}
&\leq C(0,L^*,p')\doublebar{\arr f}_{\dot W\!A_{m-1}^{0,p}(\R^n)}, & p_{0,L^*}^-&<p<2 
,\allowdisplaybreaks
\\
\label{eqn:S:N:rough:intro}
\doublebar{\widetilde N_*(\nabla^{m-1}\s^L_\nabla\arr h)}_{L^p(\R^n)}
&\leq C(0,L^*,p')\doublebar{\arr h}_{L^p(\R^n)}, & p_{0,L^*}^-&< p<2 
,
\\
\label{eqn:D:N:rough:intro}
\doublebar{\widetilde N_*(\nabla^{m-1}\D^{\mat A}\arr f)}_{L^p(\R^n)}
&\leq C(0,L^*,p')\doublebar{\arr f}_{\dot W\!A_{m-1}^{0,p}(\R^n)}, & p_{0,L^*}^-&< p<2 
.\end{align}
Here the numbers $p_{j,L}^-$ are as in formulas \cref{eqn:Meyers:p-,eqn:Meyers}, and in particular satisfy the bounds~\eqref{eqn:Meyers:bound}. The constants $C(j,L^*,p')$ depend only on the standard parameters $m$, $n$, $\lambda$, $\doublebar{\mat A}_{L^\infty(\R^n)}$, the number $p$, and the constants $c(j,L^*,p',2)$ in the bound~\eqref{eqn:Meyers}, where $1/p+1/p'=1$.
\end{thm}

The use of the numbers $p_{j,L}^+$ allows us to efficiently summarize several known special cases from the case $2m=2$. 

In particular, 
if $\mat A$ is constant then $p_{0,L}^+=p_{1,L}^+=\infty$. If $\dmn=2$ and $\mat A$ is $t$-independent, then we still have that $p_{0,L}^+=p_{1,L}^+=\infty$; see \cite[Th\'eor\`eme II.2]{AusT95} in the case $2m=2$ and \cite[Proposition~3.3]{Bar19p} (reproduced in the bound~\eqref{eqn:Meyers:bound} above) in the general case. Thus, in either of these two special cases, Theorem~\ref{thm:potentials} and the bounds \cref{eqn:D:N:rough:+,eqn:S:N:rough:+,eqn:D:lusin:rough:+,eqn:S:lusin:rough:+,eqn:D:N:+,eqn:S:N:+,eqn:D:lusin:+,eqn:S:lusin:+} imply that all eight bounds~\cref{eqn:D:N:rough:intro,eqn:S:N:rough:intro,eqn:D:lusin:rough:intro,eqn:S:lusin:rough:intro,eqn:D:N:intro,eqn:S:N:intro,eqn:D:lusin:intro,eqn:S:lusin:intro} (or \cref{eqn:D:N:rough:+,eqn:S:N:rough:+,eqn:D:lusin:rough:+,eqn:S:lusin:rough:+,eqn:D:N:+,eqn:S:N:+,eqn:D:lusin:+,eqn:S:lusin:+}) are valid for all $p$ with $1< p<\infty$. If $2m=2$, and if $\mat A$ is constant or $\dmn=2$, then all eight bounds are known (see \cite[Theorem~12.7]{AusS16}) for $1<p<\infty$. 

Furthermore, if the well known De Giorgi-Nash-Moser regularity conditions are valid, 
(which is true if $\mat A$ is real and $2m=2$, and which by \cite[Appendix~B]{AlfAAHK11} is true for complex $t$-independent coefficients in dimension $\dmn=3$), then $p_{1,L}^+=\infty$, and so the bounds \cref{eqn:S:lusin:intro,eqn:D:lusin:intro} are valid for $1<p<\infty$, the bounds \cref{eqn:S:N:intro,eqn:D:N:intro} are valid for $1<p<2+\varepsilon$, and the bounds \cref{eqn:S:lusin:rough:intro,eqn:D:lusin:rough:intro,eqn:S:N:rough:intro,eqn:D:N:rough:intro} are valid for $2-\varepsilon<p<\infty$. The $2+\varepsilon<p<\infty$ case of the bound~\eqref{eqn:D:lusin:intro} was established in \cite{Bar19p}; the remaining bounds on layer potentials were established earlier for second order $t$-independent operators satisfying the De Giorgi-Nash-Moser conditions in \cite{AusM14,HofKMP15A,HofKMP15B,HofMitMor15,HofMayMou15}.

Finally, in the general case (with $\dmn\geq 4$), \cite[Proposition~3.6]{Bar19p} (reproduced in the bound~\eqref{eqn:Meyers:bound} above) implies that the bounds \cref{eqn:S:lusin:intro,eqn:D:lusin:intro} are valid for $\frac{2n}{n+2}-\varepsilon<p<\frac{2n}{n-2}+\varepsilon$, the bounds \cref{eqn:S:N:intro,eqn:D:N:intro} are valid for $\frac{2n}{n+2}-\varepsilon<p<2+\varepsilon$, and the bounds \cref{eqn:S:lusin:rough:intro,eqn:D:lusin:rough:intro,eqn:S:N:rough:intro,eqn:D:N:rough:intro} are valid for $2-\varepsilon<p<\frac{2n}{n-2}+\varepsilon$. Again, the $2+\varepsilon<p<\frac{2n}{n-2}+\varepsilon$ cases of the bounds \cref{eqn:S:lusin:intro,eqn:D:lusin:intro} are due to \cite{Bar19p}; the remaining bounds on layer potentials for general second order operators with $t$-independent coefficients are due to~\cite[Theorem~12.7]{AusS16}.

\subsection{Outline}

The outline of this paper is as follows. In Section~\ref{sec:dfn} we will define our terminology. In Section~\ref{sec:preliminaries} we will state some known results of the theory that we will use several times throughout the paper, and (in Section~\ref{sec:tentspace}) will establish a number of results concerning tent space operators, that is, the operators $\widetilde N_+$, $\mathcal{A}_2^+$, $\widetilde N_*$, $\mathcal{A}_2^*$ given by formulas \eqref{dfn:NTM:modified:+}, \eqref{dfn:lusin:+}, \eqref{dfn:NTM:modified:*}, and~\eqref{dfn:lusin:*}, as well as the related Carleson operator $\widetilde{\mathfrak{C}}_1^\pm$, $\widetilde{\mathfrak{C}}_1^*$ given by formulas~\cref{eqn:Carleson:norm,eqn:Carleson:*}.

We will prove Theorem~\ref{thm:potentials} in Section~\ref{sec:p<2}. We will prove it by duality with the Newton potential, and so in Section~\ref{sec:newton} we will study the Newton potential. Specifically, we will establish duality formulas relating the Newton potential to the double and single layer potentials, then bound the Newton potential using the known bounds \cref{eqn:D:N:rough:+,eqn:S:N:rough:+,eqn:D:lusin:rough:+,eqn:S:lusin:rough:+,eqn:D:N:+,eqn:S:N:+,eqn:D:lusin:+,eqn:S:lusin:+} on the double and single layer potential, a decomposition argument in the spirit of \cite[Lemma~4.1]{HofMayMou15}, and the good-$\lambda$ results of \cite{Bar19p} modeled on those of \cite{She06A}.

In Section~\ref{sec:neumann} we will conclude the paper by proving Theorem~\ref{thm:neumann:selfadjoint} using the method of layer potentials. A crucial ingredient in the proof of uniqueness of solutions is the Green's formula; this formula is the subject of Section~\ref{sec:green}. 

\subsection*{Acknowledgements}

The author would like to thank Steve Hofmann and Svitlana Mayboroda for many useful conversations on topics related to this paper. The author would also like to thank
the Mathematical Sciences Research Institute for hosting a Program on Harmonic Analysis,
the Instituto de Ciencias Matem\'aticas for hosting a Research Term on ``Real Harmonic Analysis and Its Applications to Partial Differential Equations and Geometric Measure Theory'',
and
the IAS/Park City Mathematics Institute for hosting a Summer Session with a research topic of Harmonic Analysis,
at which many of the results and techniques of this paper were discussed.

\section{Definitions}\label{sec:dfn}

In this section, we will provide precise definitions of the notation and concepts used throughout this paper.

We will always work with operators~$L$ of order~$2m$ in the divergence form~\eqref{eqn:divergence} (interpreted in the weak sense of formula~\eqref{eqn:weak} below) acting on functions defined in open sets in~$\R^\dmn$, $\dmn\geq 2$.

As usual, we let $B(X,r)$ denote the ball in $\R^\dmn$ of radius $r$ and center $X$. We let $\R^\dmn_+$ and $\R^\dmn_-$ denote the upper and lower half spaces $\R^n\times (0,\infty)$ and $\R^n\times(-\infty,0)$; we will identify $\R^n$ with $\partial\R^\dmn_\pm$.
If $Q$ is a cube, we will let $\ell(Q)$ be its side length, and we let 
$cQ$ be the concentric cube of side length $c\ell(Q)$. If $E$ is a set of finite measure, we let
\begin{equation*}\fint_E f(x)\,dx=\frac{1}{\abs{E}}\int_E f(x)\,dx.\end{equation*}

If $E$ is a measurable set in Euclidean space and $H$ is a globally defined function, we will let $\1_E H=\chi_E H$, where $\chi_E$ is the characteristic function of~$E$. If $H$ is defined in all of~$E$, but is not globally defined, we will let $\1_E H$ be the extension of $H$ by zero, that is,
\begin{equation*}\1_E H(X)=\begin{cases} H(X), & X\in E, \\ 0, & \text{otherwise}.\end{cases}\end{equation*}
We will use $\1_\pm$ as a shorthand for $\1_{\R^\dmn_\pm}$.

\subsection{Multiindices and arrays of functions}\label{sec:array}

We will routinely work with multiindices in~$(\N_0)^\dmn$. (We will occasionally work with multiindices in $(\N_0)^\dmnMinusOne$.) Here $\N_0$ denotes the nonnegative integers. If $\zeta=(\zeta_1,\zeta_2,\dots,\zeta_\dmn)$ is a multiindex, then we define $\abs{\zeta}$ and $\partial^\zeta$ in the usual ways, as $\abs{\zeta}=\zeta_1+\zeta_2+\dots+\zeta_\dmn$ and $\partial^\zeta=\partial_{x_1}^{\zeta_1}\partial_{x_2}^{\zeta_2} \cdots\partial_{x_\dmn}^{\zeta_\dmn}$.

Recall that a vector $\vec H$ is a list of numbers (or functions) indexed by integers $j$ with $1\leq j\leq N$ for some $N\geq 1$. We similarly let an array $\arr H$ be a list of numbers or functions indexed by multiindices~$\zeta$ with $\abs\zeta=k$ for some $k\geq 1$. 
In particular, if $\varphi$ is a function with weak derivatives of order up to~$k$, then we view $\nabla^k\varphi$ as such an array.

The inner product of two such arrays of functions $\arr F$ and $\arr G$ defined in a measurable set $\Omega$ in Euclidean space is given by
\begin{equation*}\bigl\langle \arr F,\arr G\bigr\rangle_\Omega =
\sum_{\abs{\zeta}=k}
\int_{\Omega} \overline{F_{\zeta}(X)}\, G_{\zeta}(X)\,dX.\end{equation*}

\subsection{Function spaces and Dirichlet boundary values}

Let $\Omega$ be a measurable set in Euclidean space. We let $C^\infty_0(\Omega)$ be the space of all smooth functions supported in a compact subset of~$\Omega$. We let $L^p(\Omega)$ denote the usual Lebesgue space with respect to Lebesgue measure with norm given by
\begin{equation*}\doublebar{f}_{L^p(\Omega)}=\biggl(\int_\Omega \abs{f(x)}^p\,dx\biggr)^{1/p}.\end{equation*}

If $\Omega$ is a connected open set, then we let the homogeneous Sobolev space $\dot W^{k,p}(\Omega)$ be the space of equivalence classes of functions $u$ that are locally integrable in~$\Omega$ and have weak derivatives in $\Omega$ of order up to~$k$ in the distributional sense, and whose $k$th gradient $\nabla^k u$ lies in $L^p(\Omega)$. Two functions are equivalent if their difference is a polynomial of order at most~$k-1$.
We impose the norm
\begin{equation*}\doublebar{u}_{\dot W^{k,p}(\Omega)}=\doublebar{\nabla^k u}_{L^p(\Omega)}.\end{equation*}
Then $u$ is equal to a polynomial of order at most $k-1$ (and thus equivalent to zero) if and only if its $\dot W^{k,p}(\Omega)$-norm is zero.

If $1<p<\infty$, then $\dot W^{-1,p'}(\R^n)$ denotes the dual space to $\dot W^{1,p}(\R^n)$, where $1/p+1/p'=1$; this is a space of distributions on~$\R^n$.

The use of a dot to denote homogeneous Sobolev spaces (as opposed to the inhomogeneous spaces $W^{k,p}(\Omega)$ with $\|u\|_{W^{k,p}(\Omega)}^p=\sum_{j=0}^k\doublebar{\nabla^j u}_{L^p(\Omega)}^p$) is by now standard. The use of a dot to denote arrays of functions, as in Section~\ref{sec:array}, is also standard (see, for example, \cite{Agm57,
CohG83,CohG85,%
PipV95A,
She06B,%
MitM13A,MitM13B}). 
We apologize for any confusion arising from this overloading of notation, but the convention of these fields seems to require it.

We say that $u\in L^p_{loc}(\Omega)$ or $u\in\dot W^{k,p}_{loc}(\Omega)$ if $u\in L^p(U)$ or $u\in\dot W^{k,p}(U)$ for any bounded open set $U$ with $\overline U\subset\Omega$.

We will need a number of more specialized norms on functions.
In the introduction, we defined the nontangential maximal function $\widetilde N_+$, $\widetilde N_*$  and the Lusin area integral $\mathcal{A}_2^+$, $\mathcal{A}_2^*$. See formulas~\eqref{dfn:NTM:modified:+}, \eqref{dfn:NTM:modified:*} and~\eqref{dfn:lusin:+}, \eqref{dfn:lusin:*}. We will also need the corresponding operators in the lower half space; thus, we define
\begin{align}
\label{dfn:NTM:modified:pm}
\widetilde N_\pm H(x) &= \sup
\biggl\{\biggl(\fint_{B((y,\pm s),{s}/2)} \abs{H(z,t)}^2\,dz\,dt\biggr)^{1/2}:
s>0,\>
\abs{x-y}< s
\biggr\}
,
\\
\label{dfn:lusin:pm}
\mathcal{A}_2^\pm H(x) &= \biggl(\int_0^\infty \int_{\abs{x-y}<{t}} \abs{H(y,\pm t)}^2 \frac{dy\,dt}{{t}^\dmn}\biggr)^{1/2}
\end{align}
for all $x\in\R^n$. 

We will need one other tent space operator. 
Following \cite{CoiMS85,HofMayMou15}, the averaged Carleson operator is given by
\begin{equation}\label{eqn:Carleson:norm}
\widetilde{\mathfrak{C}}_1^\pm H(x)= \sup_{Q\owns x} \frac{1}{\abs{Q}}\int_Q \int_0^{\ell(Q)} \biggl(\fint_{B((y,\pm s),s/2)} \abs{ H(z,t)}^2\,dz\,dt\biggr)^{1/2}\,\frac{ds\,dy}{s}
\end{equation}
where the supremum is taken over cubes $Q$ in $\R^n$ containing~$x$. 
We will let the two-sided averaged Carleson operator be given by
\begin{equation}\label{eqn:Carleson:*}
\widetilde{\mathfrak{C}}_1^* H(x)= \max(\widetilde{\mathfrak{C}}_1^+ H(x), \widetilde{\mathfrak{C}}_1^- H(x)).
\end{equation}

We adopt the convention that if a $t$ appears inside the argument of one of the above operators, then it denotes the $\pdmn$th coordinate function.

Following \cite{BarHM19B}, we define the boundary values $\Trace^\pm u$ of a function $u$ defined in $\R^\dmn_\pm$ by
\begin{equation}
\label{eqn:trace}
\Trace^\pm  u
= f \quad\text{if}\quad
\lim_{t\to 0^\pm} \doublebar{u(\,\cdot\,,t)-f}_{L^1(K)}=0
\end{equation}
for all compact sets $K\subset\R^n$. We define
\begin{equation}
\label{eqn:Dirichlet}
\Tr_j^\pm u= \Trace^\pm \nabla^j u.\end{equation}
We remark that if $\nabla u$ is locally integrable up to the boundary, then $\Trace^\pm u$ exists, and furthermore $\Trace^\pm u$ coincides with the traditional trace in the sense of Sobolev spaces. Furthermore, if $\nabla u$ is locally integrable in a neighborhood of the boundary, then $\Trace^+u=\Trace^-u$; in this case we will refer to the boundary values (from either side) as $\Trace u$.

We are interested in functions with boundary data in Lebesgue or Sobolev spaces. However, observe that if $j\geq 1$, then the components of $\Tr_j^\pm u$ are derivatives of a common function and so must satisfy certain compatibility conditions. We thus define the following Whitney-Lebesgue, Whitney-Sobolev and Whitney-Besov spaces of arrays that satisfy these conditions.

\begin{defn} \label{dfn:Whitney}
Let
\begin{equation*}\mathfrak{D}=\{\Tr_{m-1}\varphi:\varphi\text{ smooth and compactly supported in $\R^\dmn$}\}.\end{equation*}

If $1\leq p<\infty$, then we let $\dot W\!A^{0,p}_{m-1}(\R^n)$ be the closure of the set $\mathfrak{D}$ in $L^p(\R^n)$.
We let $\dot W\!A^{1,p}_{m-1}(\R^n)$ be the closure of $\mathfrak{D}$ in $\dot W^{1,p}(\R^n)$. 
Finally, we let 
$\dot W\!A^{1/2,2}_{m-1}(\R^n)$ be the closure of $\mathfrak{D}$ in the Besov space $\dot B^{1/2,2}_{2}(\R^n)$; the norm in this space may be written as
\begin{equation}
\label{eqn:B:norm}
\doublebar{ f}_{\dot B^{1/2,2}_{2}(\R^n)} = \biggl(\int_{\R^n}\abs{\widehat {f}(\xi)}^2\abs{\xi}\,d\xi\biggr)^{1/2}\end{equation}
where $\widehat f$ denotes the Fourier transform of~$f$ in~$\R^n$.
\end{defn}

\begin{rmk}\label{rmk:W2:trace}
It is widely known that $\arr f\in \dot W\!A^{1/2,2}_{m-1}(\R^n)$ if and only if $\arr f=\Tr_{m-1}^+ F$ for some $F$ with $\nabla^m F\in L^2(\R^\dmn_+)$. This was essentially proven in \cite{Liz60,Jaw77}; see  \cite[Lemma~2.6]{BarHM19A} for further discussion.
\end{rmk}

\begin{rmk} There is an extensive theory of Besov spaces (see, for example, \cite{Tri83}). We will make use only of the Besov space $\dot B^{1/2,2}_2(\R^n)$ given by formula~\eqref{eqn:B:norm} and the space $\dot B^{-1/2,2}_{2}(\R^n)$. This space has norm 
\begin{equation}
\label{eqn:B:norm:-}
\doublebar{g}_{\dot B^{-1/2,2}_{2}(\R^n)} = \biggl(\int_{\R^n}\abs{\widehat {g}(\xi)}^2\frac{1}{\abs{\xi}}\,d\xi\biggr)^{1/2}.
\end{equation}
The two important properties of this space we will use are, first, that ${\dot B^{-1/2,2}_{2}(\R^n)}$ is the dual space to ${\dot B^{1/2,2}_{2}(\R^n)}$, and, second, that $f\in {\dot B^{1/2,2}_{2}(\R^n)}$ if and only if the gradient $\nabla f$ exists in the distributional sense and satisfies $\doublebar{\nabla f}_{\dot B^{-1/2,2}_{2}(\R^n)}\approx \doublebar{f}_{\dot B^{1/2,2}_{2}(\R^n)}$.
\end{rmk}

\subsection{Elliptic differential operators and Neumann boundary values}

Let $\mat A = \begin{pmatrix} A_{\alpha\beta} \end{pmatrix}$ be a matrix of measurable coefficients defined on $\R^\dmn$, indexed by multtiindices $\alpha$, $\beta$ with $\abs{\alpha}=\abs{\beta}=m$. If $\arr F$ is an array indexed by multiindices of length~$m$, then $\mat A\arr F$ is the array given by
\begin{equation*}(\mat A\arr F)_{\alpha} =
\sum_{\abs{\beta}=m}
A_{\alpha\beta} F_{\beta}.\end{equation*}

We let $L$ be the $2m$th-order divergence form operator associated with~$\mat A$. The weak formulation of such an operator is given by
\begin{equation}
\label{eqn:weak}
Lu=0 \text{ in }\Omega \text{ in the weak sense if } \langle \nabla^m \varphi, \mat A\nabla^m u\rangle_\Omega=0 \text{ for all } \varphi\in C^\infty_0(\Omega).
\end{equation}
Throughout we require our coefficients to be pointwise bounded and to satisfy the
G\r{a}rding inequality~\eqref{eqn:elliptic}, which by density we may restate as
\begin{align*}
\re {\bigl\langle\nabla^m \varphi,\mat A\nabla^m \varphi\bigr\rangle_{\R^\dmn}}
&\geq
\lambda\doublebar{\nabla^m\varphi}_{L^2(\R^\dmn)}^2
\quad\text{for all $\varphi\in\dot W^{m,2}(\R^\dmn)$}
\end{align*}
for some $\lambda>0$. 
The stronger G\r{a}rding inequality~\eqref{eqn:elliptic:slices} will play a minimal role in this paper; it is needed only because the proof of the primary results of \cite{BarHM18} required this stronger inequality, the paper \cite{Bar19p} used the results of \cite{BarHM18}, and our proof of Theorem~\ref{thm:neumann:selfadjoint} uses the results of \cite{Bar19p}.

We let $L^*$ be the elliptic operator associated with the adjoint matrix $\mat A^*$, where $(A^*)_{\alpha\beta}=\overline{A_{\beta\alpha}}$.

Recall from the introduction that the Neumann boundary values of a solution $w$ to $Lw=0$ in $\R^\dmn_+$ that satisfies estimates as in the problem~\eqref{eqn:neumann:regular:2} or~\eqref{eqn:neumann:regular:p:selfadjoint} are given by formula~\eqref{eqn:neumann:intro}.

We will also be concerned with solutions $u$ or $v$ to $Lu=0$ that satisfy $u\in\dot W^{m,2}(\R^\dmn_+)$ or $\mathcal{A}_2^+(t\nabla^m u)\in L^{p'}(\R^n)$ for $p'$ with $1<p'<\infty$.

If $u\in\dot W^{m,2}(\R^\dmn_+)$ then we may still use formula~\eqref{eqn:neumann:intro} to define $\M_{\mat A}^+ u$. Furthermore, by density, if $u\in \dot W^{m,2}(\R^\dmn_+)$ and $\M_{\mat A}^+ u$ is given by formula~\eqref{eqn:neumann:intro}, then
\begin{equation}
\label{eqn:neumann:W2}
\langle \M_{\mat A}^+ u,\Tr_{m-1}^+\varphi\rangle_{\R^n}
=
\langle \mat A\nabla^m u,\nabla^m\varphi\rangle_{\R^\dmn_+}
\quad\text{for all $\varphi\in \dot W^{m,2}(\R^\dmn_+)$.}
\end{equation}
Thus, if $u\in \dot W^{m,2}(\R^\dmn_+)$, then 
$\M_{\mat A}^+u$ is a bounded linear operator on $\dot W\!A^{1/2,2}_{m-1}(\R^n)$.

If $v$ satisfies $\mathcal{A}_2^+(t\nabla^m u)\in L^{p'}(\R^n)$, then $\nabla^m v$ may not be locally integrable up to the boundary and thus the integral on the right hand side of formula~\eqref{eqn:neumann:intro} may not converge. Thus, the definition of $\M_{\mat A}^+ v$ in this case is more delicate. We refer the reader to \cite[Section~2.3.2]{BarHM19B} for the precise formulation of the Neumann boundary values $\M_{\mat A}^+ v$ of a solution $v$ to $Lv=0$ with $\mathcal{A}_2^+(t\nabla^m v)\in L^{p'}(\R^n)$.

The numbers $C$ and $\varepsilon$ denote constants whose value may change from line to line, but which are always positive and depend only on the dimension~$\dmn$, the order $2m$ of any relevant operators, the bound $\doublebar{\mat A}_{L^\infty(\R^n)}$ on the coefficients, and the number $\lambda$ in the bound~\eqref{eqn:elliptic}. 
We say that $A\approx B$ if there are some positive constants $\varepsilon$ and~$C$ depending only on the above quantities such that $\varepsilon B\leq A\leq CB$.

The numbers $p_{j,L}^+$ are always as in the bound~\eqref{eqn:Meyers}.
The notation $C(j,L,p)$ denotes a constant that depends only on the standard parameters $n$, $m$, $\lambda$, $\doublebar{\mat A}_{L^\infty(\R^n)}$, the number~$p$, and the constant $c(j,L,p,2)$ in the bound~\eqref{eqn:Meyers}. 
(If $p$ is small enough, then $c(j,L,p,2)$ may be taken as depending only on~$p$ and the standard parameters, and so in this case we may simply write $C_p$ rather than $C(j,L,p)$. See Remark~\ref{rmk:Meyers}.)

\subsection{Potential operators}
\label{sec:dfn:potential}

In this section we will define the double and single layer potentials of Theorem~\ref{thm:potentials}. 

We will also define the Newton potential. We will use the Newton potential to define the double layer potential. Furthermore, we will prove Theorem~\ref{thm:potentials} by establishing various bounds on the Newton potential and using duality to pass to estimates on the double and single layer potentials.

For any $\arr H\in L^2(\R^\dmn)$, by the Lax-Milgram lemma there is a unique function $\Pi^L\arr H$ in $\dot W^{m,2}(\R^\dmn)$ that satisfies
\begin{equation}\label{eqn:newton}
\langle \nabla^m\varphi, \mat A\nabla^m \Pi^L\arr H\rangle_{\R^\dmn}=\langle \nabla^m\varphi, \arr H\rangle_{\R^\dmn}
\quad\text{for all $\varphi\in \dot W^{m,2}(\R^\dmn)$.}
\end{equation}
We will use the operator $\Pi^L$ operator frequently, and thus will refer it as the Newton potential. This represents a break from tradition, as the traditional Newton potential $\mathcal{N}^L$ is usually taken to satisfy $\langle \nabla^m\varphi,\mat A\nabla^m\mathcal{N}^L H\rangle_{\R^\dmn}=\langle\varphi,H\rangle_{\R^\dmn}$.

We record here that, by \cite[Lemma~43]{Bar16}, there is some $\varepsilon>0$ such that if $2-\varepsilon<r<2+\varepsilon$, then
\begin{equation}\label{eqn:newton:bound:r}
\doublebar{\nabla^m\Pi^{L}\arr H}_{L^r(\R^\dmn)} \leq C_r \doublebar{\arr H}_{L^r(\R^\dmn)}
\end{equation}
for all $\arr H\in L^r(\R^\dmn)\cap L^2(\R^\dmn)$.

We will be interested in the gradient $\nabla^{m-1}\Pi^{L}\arr H$ of order~$m-1$. However, $\Pi^{L}\arr H$ as defined by formula~\eqref{eqn:newton} is an element of $\dot W^{m,2}(\R^\dmn)$, and as such, it is the gradient $\nabla^{m}\Pi^{L}\arr H$ of order~$m$ that is well defined; $\nabla^{m-1}\Pi^{L}\arr H$ is defined only up to adding constants.

We may fix an additive normalization as follows. 
If $\dmn\geq 3$, then by the Gagliardo-Nirenberg-Sobolev inequality,  (see, for example, \cite[Section~5.6]{Eva10}) there is a unique additive normalization of $\nabla^{m-1}\Pi^{L}\arr H$ such that
\begin{equation}\label{eqn:newton:GNS}
\doublebar{\nabla^{m-1}\Pi^{L}\arr H}_{L^q(\R^\dmn)}
\leq C \doublebar{\nabla^{m}\Pi^{L}\arr H}_{L^2(\R^\dmn)}\end{equation}
where $\pdmn/q=\pdmn/2-1$ (and in particular where $q<\infty$).

If $\dmn=2$, let $r<2$ be as in the bound~\eqref{eqn:newton:bound:r}. If $\arr H \in L^2(\R^\dmn)$ is compactly supported, or more generally if $\arr H\in L^2(\R^\dmn)\cap L^r(\R^\dmn)$, then again by the Gagliardo-Nirenberg-Sobolev inequality, there is a unique additive normalization of $\nabla^{m-1}\Pi^{L}\arr H$ such that
\begin{equation}\label{eqn:newton:GNS:2}
\doublebar{\nabla^{m-1}\Pi^{L}\arr H}_{L^q(\R^2)}
\leq C_r \doublebar{\nabla^{m}\Pi^{L}\arr H}_{L^r(\R^2)}
\end{equation}
where $2/q=2/r-1$ (and so again $q<\infty$).

We will use this additive normalization throughout.

We now define the double layer potential.
Suppose that $\arr f\in \dot W\!A^{1/2,2}_{m-1}(\R^n)$. As mentioned in Remark~\ref{rmk:W2:trace}, $\arr f=\Tr_{m-1}^- F$ for some $F\in \dot W^{m,2}(\R^\dmn_+)$.
We define
\begin{equation}
\label{dfn:D:newton:-}
\D^{\mat A}\arr f =
-\Pi^L(\1_-\mat A\nabla^m F) + \1_- F
.\end{equation}
This operator is well-defined, that is, does not depend on the choice of~$F$. See \cite[Section~2.4]{BarHM17} or \cite[Section~4]{Bar17}. 
Using the bounds~\cref{eqn:D:N:+,eqn:D:N:rough:+} we may extend $\D^{\mat A}$ by density to an operator on all of $\dot W\!A^{k,p}_{m-1}(\R^n)$, for $k\in \{0,1\}$ and for an appropriate range of~$p$.

We now define the single layer potential.
Let $\arr g$ be a bounded linear operator on $\dot W\!A^{1/2,2}_{m-1}(\R^n)$. Then by Remark~\ref{rmk:W2:trace}, $F\to \langle \Tr_{m-1} F,\arr g\rangle_{\R^n}$ is a bounded linear operator on $\dot W^{m,2}(\R^\dmn)$. By the Lax-Milgram lemma, there is a unique function $\s^L\arr g\in \dot W^{m,2}(\R^\dmn)$ that satisfies
\begin{align}
\label{dfn:S}
\langle \nabla^m\varphi, \mat A\nabla^m\s^L\arr g\rangle_{\R^\dmn}&=\langle \Tr_{m-1}\varphi,\arr g\rangle_{\R^n}
&&\text{for all }\varphi\in \dot W^{m,2}(\R^\dmn)
.\end{align}
See \cite{Bar17}. We remark that formula~\eqref{dfn:S} is also meaningful and $\s^L\arr g$ is defined for $\arr g\in \dot B^{-1/2,2}_2(\R^n)$.
This definition coincides with the definition of $\s^L\arr g$ involving the Newton potential given in \cite{BarHM17,BarHM19A}. Using the bound~\eqref{eqn:S:N:+} we may extend $\s^L$ by density to an operator on all of $L^p(\R^n)$ for all $2-\varepsilon<p<p_{0,L}^+$.

\begin{rmk}
If $L$ is an operator of the form~\eqref{eqn:weak}, then $L$ may generally be associated to many choices of coefficients $\mat A$; for example, if $A_{\alpha\beta}=\widetilde A_{\alpha\beta}+M_{\alpha\beta}$, where $M$ is constant and $M_{\alpha\beta}=-M_{\beta\alpha}$, then the operators associated to $\mat A$ and $\mat{\widetilde A}$ are equal. The single layer potential $\s^L$ depends only on the operator~$L$, while the double layer potential $\D^{\mat A}$ depends on the particular choice of coefficients~$\mat A$.
\end{rmk}

In \cite{BarHM18p} the operator $\s^L_\nabla$ was defined in terms of integrals involving the fundamental solution. In the present paper, we will simply define $\s^L_\nabla$ as the operator satisfying \cite[formulas~(4.5--4.6)]{BarHM18p}. 
These formulas are as follows. If $\zeta$ is a multiindex, then
$\arr e_{\zeta}$ is the unit array associated to the multiindex $\zeta$; that is, 
\begin{equation}
\label{eqn:e}
(\arr e_\zeta)_\zeta=1,\quad (\arr e_\zeta)_\theta=0
\text{ whenever } \abs\theta=\abs\zeta\text{ and }\theta\neq \zeta.\end{equation}
Let $h \in \dot B^{1/2,2}_2(\R^n)\cap \dot B^{-1/2,2}_2(\R^n)$.
Suppose that $\alpha$ and $\gamma$ are multiindices with $\abs\alpha=m$ and $\abs\gamma=m-1$; in particular, we require that all entries of $\gamma$ be nonnegative.
Then
\begin{equation}\label{eqn:S:S:horizontal}
\nabla^{m}\s^L_\nabla 
(h\arr e_\alpha)(x,t) = -\nabla^{m}\s^L((\partial_{x_j}  h)\arr e_\gamma)(x,t)
\quad
\text{if } 1\leq j\leq n \text{ and } \alpha=\gamma+\vec e_j\end{equation}
and
\begin{equation}\label{eqn:S:S:vertical}
\nabla^{m-1}\s^L_\nabla 
(h\arr e_\alpha)(x,t) = -\nabla^{m-1}\partial_t\s^L (h\arr e_\gamma)(x,t)
\text{ if } \alpha=\gamma+\vec e_\dmn.
\end{equation}
We define $\s^L_\nabla\arr h$ for general $\arr h$ by linearity. 
As shown in \cite[Lemma~4.4]{BarHM18p}, $\s^L_\nabla$ is well defined in the sense that if $1\leq \alpha_\dmn \leq m-1$, then we may use either formula~\eqref{eqn:S:S:horizontal} or~\eqref{eqn:S:S:vertical} to define $\nabla^m\s^L_\nabla(h\arr e_\alpha)$, and furthermore that if $\alpha_\ell\geq 1$ and $\alpha_k\geq 1$ then the value of the right hand side of formula~\eqref{eqn:S:S:horizontal} is the same whether we choose $j=k$ or $j=\ell$. 

Furthermore, by \cite[Lemma~4.8]{BarHM18p}, if $\arr h\in L^2(\R^n)\subset \dot B^{1/2,2}_2(\R^n)\cap \dot B^{-1/2,2}_2(\R^n)$ then there is a (necessarily unique) additive normalization of $\nabla^{m-1}\s^L\arr h$ that satisfies $\lim_{t\to\pm\infty}\doublebar{\nabla^{m-1}\s^L_\nabla\arr h(\,\cdot\,,t)}_{L^2(\R^n)}=0$.
Using the bound~\eqref{eqn:S:N:rough:+}, we may extend $\s^L_\nabla$ by density to an operator on all of $L^p(\R^n)$ for $2-\varepsilon<p<p_{1,L}^+$.

\section{Preliminaries}
\label{sec:preliminaries}

In this section we will discuss a few known results and establish some general results that will be of use throughout the paper.

Specifically, in Section~\ref{sec:lower} we will discuss the change of variables $(x,t)\to (x,-t)$, and how it allows us to easily generalize from the upper half space to the lower half space. In Section~\ref{sec:regular} we will list some known results from the theory of solutions to elliptic equations $Lu=0$. Finally, in Section~\ref{sec:tentspace}, we will establish some general results involving tent spaces, that is, spaces of functions $H$ for which the tent space norms $\widetilde N_+H$, $\mathcal{A}_2^+H$ or $\widetilde{\mathfrak{C}}_1^+H$ lie in~$L^p(\R^n)$.

\subsection{The lower half space}
\label{sec:lower}

It is often notationally convenient to establish bounds only in the upper half space and to use change of variables arguments to generalize to the lower half space.

The change of variables $(x,t)\to (x,-t)$, for $x\in\R^n$ and $t\in\R$, interchanges the upper and lower half spaces. In \cite[Section~3.3]{BarHM18p}, it was shown that if $Lu=0$ in $\Omega$, then $L^-u^-=0$ in $\Omega^-$, where $u^-(x,t)=u(x,-t)$, $\Omega^-=\{(x,t):(x,-t)\in\Omega\}$, and $L^-$ is the operator of the form~\eqref{eqn:weak} associated to the coefficients $\mat A^-$ given by $A^-_{\alpha\beta}=(-1)^{\alpha_\dmn+\beta_\dmn} A_{\alpha\beta}$. Notice that if $\mat A$ is bounded, $t$-independent and satisfies the condition~\eqref{eqn:elliptic} (or \eqref{eqn:elliptic:slices}), then $\mat A^-$ satisfies the same conditions with $\doublebar{\mat A}_{L^\infty(\R^n)}=\doublebar{\mat A^-}_{L^\infty(\R^n)}$ and with the same value of~$\lambda$. 

We observe that by the same change of variables argument, if $j$ is an integer with $0\leq j\leq m$, and if $p_{j,L}^+$ and $c(j,L,p,q)$ are as in the bound~\eqref{eqn:Meyers}, then \begin{equation}p_{j,L}^+=p_{j,L^-}^+\quad\text{and}\quad c(j,L,p,q)=c(j,L^-,p,q)\text{ for all }0<q<p<p_{j,L}^+.\end{equation}

Furthermore, by \cite[Section~3.3]{BarHM18p},
\begin{align*}
\D^{\mat A}\arr f(x,-t)&=-\D^{\mat A^-}\arr f^-(x,t)
,&
\s^L\arr g(x,-t)&=\s^{L^-}\arr g^-(x,t)
,\\
\Pi^{L}\arr H(x,-t)&=\Pi^{L^-}\arr H^-(x,t)
,&
\s^L_\nabla\arr h(x,-t)&=\s^{L^-}_\nabla\arr h{}^-(x,t)
\end{align*}
where 
\begin{align*}
f_\gamma^-(x)&=(-1)^{\gamma_\dmn} f_\gamma(x),& g_\gamma^-(x)&=(-1)^{\gamma_\dmn} g_\gamma(x), \\ H_\alpha^-(x,t)&=(-1)^{\alpha_\dmn} H_\alpha(x,-t),& h_\beta^-(x)&=(-1)^{\beta_\dmn}h_\beta(x).
\end{align*}
It is straightforward to calculate that if $\arr f=\Tr_{m-1}^+\varphi$ in the sense of formula~\eqref{eqn:Dirichlet}, then $\arr f^-=\Tr_{m-1}^-\varphi^-$. 
Thus, $\arr f^-$ is in the distinguished subspace $\mathfrak{D}$ of Definition~\ref{dfn:Whitney} if and only if $\arr f$ is, and so the mapping $\arr f\to\arr f^-$ is an automorphism of $\dot W\!A^{s,p}_{m-1}(\R^n)$ for all the spaces $\dot W\!A^{s,p}_{m-1}(\R^n)$ defined by Definition~\ref{dfn:Whitney}.

We observe further that if $\M_{\mat A}^+w\owns \arr g$, then by the definition~\eqref{eqn:neumann:intro} of Neumann boundary values, if $\varphi\in C^\infty_0(\R^\dmn)$ then
\begin{align*}
\langle \Tr_{m-1}\varphi,\arr g^-\rangle_{\R^n}
&
=\langle \Tr_{m-1}(\varphi^-),\arr g\rangle_{\R^n} 
=\langle \nabla^m(\varphi^-),\mat A\nabla^m w\rangle_{\R^\dmn_+}
\\&=\langle \nabla^m \varphi,\mat A^-\nabla^m w^- \rangle_{\R^\dmn_-}
\end{align*}
and so
\begin{equation}\label{eqn:neumann:lower}
\text{if }\M_{\mat A}^+w\owns \arr g \text{ then }\M_{\mat A^-}^- w^-\owns \arr g^-.\end{equation}
An examination of the definition of Neumann boundary values in \cite[Section~2.3.2]{BarHM19B} reveals that formula~\eqref{eqn:neumann:lower} is valid if that definition of Neumann boundary values is used instead.

Thus, we may easily pass from bounds in the upper half space to bounds in the lower half space.

\subsection{Solutions to elliptic equations}
\label{sec:regular}

It is well known that solutions to the elliptic equation $Lu=0$ display many useful properties. In this section we will state two regularity results that will be used throughout the paper.

We begin with the higher order analogue of the Caccioppoli inequality. This lemma was proven in full generality in \cite{Bar16} and some important preliminary versions were established in \cite{Cam80,AusQ00}.
\begin{lem}[The Caccioppoli inequality]\label{lem:Caccioppoli}
Let $L$ be an operator of the form~\eqref{eqn:weak} of order~$2m$ associated to bounded coefficients~$\mat A$ that satisfy the ellipticity condition~\eqref{eqn:elliptic}.

Let $ u\in \dot W^{m,2}(B(X,2r))$ with $L u=0$ in $B(X,2r)$.
Then we have the bound
\begin{equation*}
\fint_{B(X,r)} \abs{\nabla^j  u(x,s)}^2\,dx\,ds
\leq \frac{C}{r^2}\fint_{B(X,2r)} \abs{\nabla^{j-1}  u(x,s)}^2\,dx\,ds
\end{equation*}
for any $j$ with $1\leq j\leq m$.
\end{lem}

If $\mat A$ is $t$-independent then solutions to $Lu=0$ have additional regularity. In particular, the following lemma was proven in the case $m = 1$ in \cite[Proposition 2.1]{AlfAAHK11} and generalized to the case $m \geq 2$ in \cite[Lemma 3.20]{BarHM19A}.
\begin{lem}\label{lem:slices}
Let $L$ be an operator of the form~\eqref{eqn:weak} of order~$2m$ associated to bounded $t$-independent coefficients~$\mat A$ that satisfy the ellipticity condition~\eqref{eqn:elliptic}.

Let $Q\subset\R^n$ be a cube of side length~$\ell(Q)$ and let $I\subset\R$ be an interval with $\abs{I}=\ell(Q)$.
If $u\in \dot W^{m,2}_{loc}(2Q\times 2I)$ and $Lu=0$ in $2Q\times 2I$, then
\begin{equation*}\int_Q \abs{\nabla^{m-j} \partial_t^k u(x,t)}^p\,dx \leq \frac{C(j,L,p)}{\ell(Q)}
\int_{2Q}\int_{2I} \abs{\nabla^{m-j} \partial_s^k u(x,s)}^p\,ds\,dx
\end{equation*}
for any $t\in I$, any integer $j$ with $0\leq j\leq m$, any $p$ with $0< p < p_{j,L}^+$, and any integer $k\geq 0$. 
\end{lem}

\subsection{Tent spaces}
\label{sec:tentspace}

Recall that Theorem~\ref{thm:potentials} concerns nontangential maximal and area integral norms of layer potentials. Thus, in order to prove Theorem~\ref{thm:potentials}, we will need a number of results concerning the area integral, the nontangential maximal operator, and the Carleson operator of formula~\eqref{eqn:Carleson:norm}.

We begin with the following lemma concerning the Lebesgue norm and the area integral.

\begin{lem}\label{lem:lebesgue:lusin} 
Let $\sigma>0$, $\kappa\in\R$, and $0<\theta\leq r\leq 2$. Let $\arr F\in L^2_{loc}(\R^\dmn_+)$ be such that $\mathcal{A}_2^+(t^\kappa\arr F)\in L^\theta(\R^n)$.

If $\theta\pdmn < r(n+\theta\kappa)$, then
\begin{equation*}\doublebar{\arr F}_{L^r(\R^n\times(\sigma,\infty))}
\leq \frac{C_{n,\theta,\kappa,r}}{\sigma^{\kappa+n/\theta-1/r-n/r}} \doublebar{\mathcal{A}_2^+(t^\kappa\arr F)}_{L^\theta(\R^n)}.\end{equation*}
If $\theta\pdmn > r(n+\theta\kappa)$, then
\begin{equation*}\doublebar{\arr F}_{L^r(\R^n\times(0,\sigma))}
\leq \frac{C_{n,\theta,\kappa,r}}{\sigma^{\kappa+n/\theta-1/r-n/r}} \doublebar{\mathcal{A}_2^+(t^\kappa\arr F)}_{L^\theta(\R^n)}.\end{equation*}
\end{lem}
\begin{proof}
Our argument is largely taken from \cite[Remark~5.3]{BarHM19B}, where the case $r=2$, $\kappa=1$ was considered. Let $j$ be an integer. Then
\begin{align*}\int_{\R^n}\int_{2^j\sqrt{n}}^{2^{j+1}\sqrt{n}}\abs{\arr F(x,t)}^r \,dt\,dx
&=\sum_{Q\in\mathcal{G}_j}
\int_Q \int_{\sqrt{n}\ell(Q)}^{2\sqrt{n}\ell(Q)}\abs{\arr F(x,t)}^r\,dt\,dx
\end{align*}
where $\mathcal{G}_j$ is a grid of pairwise-disjoint open cubes in $\R^n$ of side length $2^j$ whose union is almost all of $\R^n$. 
If $r\leq 2$, then by H\"older's inequality, 
\begin{equation*}\int_Q \int_{\ell(Q)\sqrt{n}}^{2\ell(Q)\sqrt{n}}\abs{\arr F(x,t)}^r\,dt\,dx
\leq 
\bigl(\abs{Q}\ell(Q)\sqrt{n}\bigr)^{1-r/2}
\biggl(\int_Q \int_{\ell(Q)\sqrt{n}}^{2\ell(Q)\sqrt{n}}\abs{\arr F(x,t)}^2\,dt\,dx\biggr)^{r/2}
.\end{equation*}
For every $x$, $y\in Q$ and every $t>\ell(Q)\sqrt{n}$ we have that
\begin{equation*}\abs{x-y}<\diam Q = \ell(Q)\sqrt{n}<t,\end{equation*}
and so for any $y\in Q$ we have that
\begin{equation*}
\int_Q \int_{\ell(Q)\sqrt{n}}^{2\ell(Q)\sqrt{n}}\abs{\arr F(x,t)}^2\,dt\,dx
\leq 
C_{n,\kappa}\int_0^\infty \int_{\abs{x-y}<t} \abs{\arr F(x,t)}^2\frac{\ell(Q)^{\dmn-2\kappa}}{t^{\dmn-2\kappa}}\,dx\,dt
.\end{equation*}
Thus,
\begin{align*}\int_{\R^n}\int_{2^j\sqrt{n}}^{2^{j+1}\sqrt{n}} \abs{\arr F(x,t)}^r \,dt\,dx
&\leq C_{n,\kappa}\sum_{Q\in\mathcal{G}_j}
\ell(Q)^{n+1-r\kappa}
\biggl(\fint_Q\mathcal{A}_2^+(t^\kappa\arr F)(y)^\theta\,dy\biggr)^{r/\theta}
.\end{align*}
If $\theta\leq r$ then
\begin{equation*}\sum_{Q\in \mathcal{G}_j}
\biggl(\int_Q\mathcal{A}_2^+(t\arr F)(y)^\theta\,dy\biggr)^{r/\theta}
\leq
\biggl(\int_{\R^n} \mathcal{A}_2^+(t\arr F)(y)^\theta \,dy \biggr)^{r/\theta}
\end{equation*}
and so
\begin{align*}
\int_{\R^n}\int_{2^j\sqrt{n}}^{2^{j+1}\sqrt{n}} \abs{\arr F(x,t)}^r \,dt\,dx
&\leq
C_{n,\kappa} 2^{j(n+1-r\kappa-nr/\theta)}
\biggl(\int_{\R^n}\mathcal{A}_2^+(t^\kappa\arr F)(y)^\theta\,dy\biggr)^{r/\theta}
.\end{align*}
By summing over $j$ with $2^{j+1}\sqrt{n}>\sigma$ or with $2^j\sqrt{n}<\sigma$, we may complete the proof.
\end{proof}

We now establish the following localization lemma involving the Carleson operator~\eqref{eqn:Carleson:norm}.
\begin{lem}\label{lem:C:local} 
Let $1<r\leq\infty$, let $Q\subset\R^n$ be a cube, and let $\arr H\in L^2_{loc}(\R^\dmn_+)$ be such that ${\widetilde{\mathfrak{C}}_1^+(t\arr H)}\in{L^{r}(16Q)}$. Then
\begin{equation}
\label{eqn:C:local}\doublebar{\widetilde{\mathfrak{C}}_1^+(\1_{10Q\times (0,\ell(Q))}t\arr H)}_{L^{r}(\R^n)}
\leq C_{n,r} \doublebar{\widetilde{\mathfrak{C}}_1^+(t\arr H)}_{L^{r}(16Q)}
.\end{equation}
In particular, if $\arr H\in L^2(\R^\dmn_+)$ is supported in a compact subset of $\R^\dmn_+$, then $\widetilde{\mathfrak{C}}_1^+(t\arr H)\in {L^{r}(\R^n)}$ for all $1<r\leq\infty$.
\end{lem}

\begin{proof} We begin with the bound~\eqref{eqn:C:local}.
If $x\in 16Q$, then ${\widetilde{\mathfrak{C}}_1^+(\1_{10Q\times(0,\ell(Q))}t\arr H)(x)}
\leq {\widetilde{\mathfrak{C}}_1^+(t\arr H)(x)}$. Thus, we need only consider $x\notin 16Q$.

Let $\arr\Phi(x,t)=t\arr H(x,t)$.
By formula~\eqref{eqn:Carleson:norm},
\begin{multline*}
\widetilde{\mathfrak{C}}_1^+ (\1_{10Q\times(0,\ell(Q))}t\arr H)(x)
\\=\sup_{R\owns x} \frac{1}{\abs{R}}\int_R \int_0^{\ell(R)} \biggl(\fint_{B((y,s),s/2)} \1_{10Q\times(0,\ell(Q))}\abs{\arr \Phi}^2\biggr)^{1/2}\,\frac{ds\,dy}{s}
\end{multline*}
where the supremum is over cubes $R\subset\R^n$ with $x\in R$. 
Observe that if 
\begin{equation*}B((y,s),s/2)\cap (10Q\times(0,\ell(Q)))\neq \emptyset,\end{equation*}
then $s < 2\ell(Q)$ and 
$\dist(y,10Q)<s/2<\ell(Q)$, so $y\in 12Q$.
Thus, 
\begin{multline*}
\int_{\R^n} \int_0^\infty \biggl(\fint_{B((y,s),s/2)} \1_{10Q\times(0,\ell(Q))}\abs{\arr\Phi}^2\biggr)^{1/2}\,\frac{ds\,dy}{s}
\\\leq
\int_{12Q} \int_0^{2\ell(Q)} \biggl(\fint_{B((y,s),s/2)} \abs{\arr \Phi}^2\biggr)^{1/2}\,\frac{ds\,dy}{s}
\leq 
\abs{12Q}
\widetilde{\mathfrak{C}}_1^+(t\arr H)(z)\end{multline*}
for all $z\in 12Q$.
Furthermore, if 
\begin{equation*}\int_R \int_0^{\ell(R)} \biggl(\fint_{B((y,s),s/2)} \1_{10Q\times(0,\ell(Q))}\abs{\arr\Phi}^2\biggr)^{1/2}\,\frac{ds\,dy}{s}\neq 0\end{equation*}
then $R\cap 12Q\neq\emptyset$. But observe that if $x\notin 16Q$ and $R\owns x$ with $R\cap 12Q\neq\emptyset$, then $\sqrt{n}\ell(R)=\diam(R)>\dist(x,12Q)$. Thus, if $x\notin16Q$ then
\begin{equation*}
\widetilde{\mathfrak{C}}_1^+ (\1_{10Q\times(0,\ell(Q))}t\arr H)(x)
\leq
\frac{\abs{12Q}}{n^{n/2}\dist(x,12Q)^n}
\biggl(\fint_{12Q} \widetilde{\mathfrak{C}}_1^+(t\arr H)(z)^{r}\,dz\biggr)^{1/r}
.\end{equation*}
A straightforward computation yields the bound~\eqref{eqn:C:local} for all $r>1$.

We now turn to the case of compactly supported~$\arr H$. If $\arr H$ is supported in a compact subset of $\R^\dmn_+$, then there is some $\varepsilon>0$ and some $N<\infty$ such that $\arr H(x,t)=0$ whenever $t<\varepsilon$ or $t>N$. We compute that
\begin{equation*}\biggl(\fint_{B((y,s),s/2)} \abs{t\arr H(z,t)}^2\,dz\,dt\biggr)^{1/2} \leq C_n s^{1/2-n/2}\doublebar{\arr H}_{L^2(\R^\dmn_+)}\end{equation*}
and is zero if $s<2\varepsilon/3$ or $s>2N$. Thus, by formula~\eqref{eqn:Carleson:norm},
\begin{align*}
\widetilde{\mathfrak{C}}_1^+ (t\arr H)(x)
&=\sup_{R\owns x} \fint_R \int_0^{\ell(R)} \biggl(\fint_{B((y,s),s/2)} \abs{\arr H(x,t)}^2\,t^2\,dx\,dt\biggr)^{1/2}\,\frac{ds\,dy}{s}
\\&\leq \int_{2\varepsilon/3}^{2N} C_n s^{-1/2-n/2}\doublebar{\arr H}_{L^2(\R^\dmn_+)}\,ds
.\end{align*}
The right hand side is finite and independent of~$x$. Thus, if $\arr H$ is supported in a compact subset of $\R^\dmn_+$, then $\widetilde{\mathfrak{C}}_1^+ (t\arr H)$ is bounded and so the right hand side of formula~\eqref{eqn:C:local} is finite for any fixed cube~$Q$. But if $\arr H$ is compactly supported, then it is supported in $10Q\times(0,\ell(Q))$ for some cube~$Q$; thus, $\1_{10Q\times(0,\ell(Q))}\arr H =\arr H$, and so by the bound~\eqref{eqn:C:local} we have that $\widetilde{\mathfrak{C}}_1^+ (t\arr H)\in L^{r}(\R^n)$, as desired.
\end{proof}

We now come to a method for bounding nontangential maximal functions by duality. This is the reason that the Carleson operator $\widetilde{\mathfrak{C}}_1^+$ is of interest in the present paper. It is well known (see \cite[Theorem~5.1]{AlvM87}) that if $1<p<\infty$ and $1/p+1/p'=1$, then the dual to the space of nontangentially bounded functions
\begin{equation*}\{U:N_+U\in L^p(\R^n)\},\quad \text{where } N_+U(x)=\sup\{\abs{U(y,t)}:\abs{x-y}<t\},\end{equation*}
is the space of Borel measures
\begin{equation*}\{\mu:
\mathfrak{C}_1^+(t\abs{\mu})\in L^{p'}(\R^n)\},
\quad\text{where}\quad
\mathfrak{C}_1^+(\mu)(x)=\sup_{Q\owns x} \frac{1}{\abs{Q}}\int_Q\int_0^{\ell(Q)}\frac{1}{t} d\abs{\mu}(y,t)
.\end{equation*}
We claim that a similar result is true for the spaces defined by the averaged norms $\widetilde N_+$ and $\widetilde{\mathfrak{C}}_1^+$ given by formulas~\eqref{dfn:NTM:modified:pm} and~\eqref{eqn:Carleson:norm}. More precisely, we will use the following two lemmas.

\begin{lem}\label{lem:N>C}
Let $p$ satisfy $1<p<\infty$ and let $p'$ satisfy $1/p+1/p'=1$.
Suppose that $\arr u$ and $\arr H$ are such that ${\widetilde N_+\arr u}\in{L^p(\R^n)}$ and ${\widetilde{\mathfrak{C}}_1^+(t\arr H)}\in{L^{p'}(\R^n)}$. Then
\begin{equation*}\abs{\langle \arr u,\arr H\rangle_{\R^\dmn_+}}
\leq C\doublebar{\widetilde N_+\arr u}_{L^p(\R^n)} \doublebar{\widetilde{\mathfrak{C}}_1^+(t\arr H)}_{L^{p'}(\R^n)}.\end{equation*}
\end{lem}

\begin{lem}\label{lem:N<C}
Suppose that $\arr u\in L^2_{loc}(\R^\dmn_+)$. Let $1<p<\infty$ and let $1/p+1/p'=1$. 
Then
\begin{equation*}
\doublebar{\widetilde N_+\arr u}_{L^p(\R^n)}
\leq C_p\sup_{\arr H\in L^2_c(\R^\dmn_+)\setminus\{\arr 0\}} \frac{\abs{\langle \arr u,\arr H\rangle_{\R^\dmn_+}}} {\doublebar{\widetilde{\mathfrak{C}}_1^+(t\arr H)}_{L^{p'}(\R^n)}}
\end{equation*}
provided the right hand side is finite.
Here $L^2_c(\R^\dmn_+)=\{\arr H\in L^2(\R^\dmn_+):\supp \arr H\subset K$ for some compact set $K\subset\R^\dmn_+\}$.
\end{lem}

\begin{proof}[Proof of Lemma~\ref{lem:N>C}]
Let $F$ be an integrable function. 
Then
\begin{align*}
\int_{\R^\dmn_+} F(x,t)\,dx\,dt
&= 
\int_{\R^\dmn_+} F(y+sz,s+sr)\,(1+r)\,dy\,ds 
\end{align*}
for any $z\in\R^n$ and any $r>-1$. Averaging over $(z,r)\in B(0,1/2)$, we have that
\begin{align*}
\int_{\R^\dmn_+} F(x,t)\,dx\,dt
&= 
\int_{\R^\dmn_+} \fint_{B(0,1/2)}F(y+sz,s+sr)\frac{s+sr}{s}\,dz\,dr\,dy\,ds 
\end{align*}
and a change of variables yields that
\begin{align}
\int_{\R^\dmn_+} F(x,t)\,dx\,dt
&= 
\int_{\R^\dmn_+} \fint_{B((y,s),s/2)}F(x,t)\frac{t}{s}\,dx\,dt\,dy\,ds 
.\end{align}

Let $K$ be a compact set in $\R^\dmn_+$. Observe that $\arr u$ and $\arr H$ are both in $L^2_{loc}(\R^\dmn_+)$; thus $F=\1_K\abs{ \arr u}\abs{\arr H}$ is integrable. Therefore,
\begin{equation*}
\int_K \abs{ \arr u}\abs{\arr H}
= \int_{\R^\dmn_+} \fint_{B((y,s),s/2)}
\abs{\1_K\arr u(x,t)}\abs{t\arr H(x,t)}\,dx\,dt\,\frac{dy\,ds}{s}.\end{equation*}
We define 
\begin{equation*}H(y,s)=\frac{1}{s}\biggl(\fint_{B((y,s),s/2)}\abs{t\arr H(x,t)}^2\,dx\,dt\biggr)^{1/2},
\quad U(y,s)=\biggl(\fint_{B((y,s),s/2)}\abs{\arr u}^2\biggr)^{1/2}\end{equation*} 
so that by the definitions~\eqref{dfn:NTM:modified:pm} and~\eqref{eqn:Carleson:norm} of $\widetilde N_+$ and~${\mathfrak{C}}_1^+$, 
\begin{equation*}N_+U=\widetilde N_+\arr u,\qquad \mathfrak{C}_1^+ (tH)=\widetilde{\mathfrak{C}}_1^+(t\arr H).\end{equation*}
By H\"older's inequality,
\begin{equation*}
\int_K \abs{ \arr u}\abs{\arr H}
\leq
\int_{\R^\dmn_+} UH.\end{equation*}
By the duality results discussed above (see \cite[formula~(2.6)]{CoiMS85}
), we have that
\begin{equation*}
\int_{\R^\dmn_+} U H
\leq
C \doublebar{ N_+ U}_{L^p(\R^n)}
\doublebar{{\mathfrak{C}}_1^+(t H)}_{L^{p'}(\R^n)}.
\end{equation*}
Thus,
\begin{equation*}
\int_K \abs{ \arr u}\abs{\arr H}
\leq
C \doublebar{\widetilde N_+ \arr u}_{L^p(\R^n)}
\doublebar{\widetilde{\mathfrak{C}}_1^+(t \arr H)}_{L^{p'}(\R^n)}
\end{equation*}
provided the right hand side is finite. Because $K$ was arbitrary, this inequality is still true if we integrate over $\R^\dmn_+$ instead of~$K$. Thus, $\langle \arr u,\arr H\rangle_{\R^\dmn_+}$ represents an absolutely convergent integral that satisfies
\begin{equation*}
\abs{\langle \arr u,\arr H\rangle_{\R^\dmn_+}}
\leq
C \doublebar{\widetilde N_+ \arr u}_{L^p(\R^n)}
\doublebar{\widetilde{\mathfrak{C}}_1^+(t \arr H)}_{L^{p'}(\R^n)}
\end{equation*}
as desired.
\end{proof}

The following lemma will be used in the proof of Lemma~\ref{lem:N<C}.

\begin{lem}\label{lem:C:L2} Let $\mu$ be a nonnegative measure on~$\R^n$. For each $(x,r)\in \R^\dmn_+$, let $\arr H_{(x,r)}$ be defined in $\R^\dmn_+$, supported in $B((x,r),r/2)$, and satisfy
\begin{equation*}\left(\fint_{B((x,r),r/2)} \abs{\arr H_{(x,r)}}^2\right)^{1/2}\leq 1.\end{equation*}
Define
\begin{equation*}\arr H(z,t)=\int_{\R^\dmn_+} \frac{1}{r^\dmn} \arr H_{(x,r)}(z,t)\,d\mu(x,r).\end{equation*}
Then
\begin{align*}
\widetilde{\mathfrak{C}}_1^+ (t\arr H)(\tilde x) 
&\leq
C\mathfrak{C}_1^+(t\mu)(\tilde x)
\end{align*}
for all $\tilde x\in\R^n$ such that the right hand side is finite.
\end{lem}
\begin{proof}
Let $W(y,s)=B((y,s),s/2)$, and let $V(y,s)=\{(x,r):W(y,s)\cap W(x,r)\neq\emptyset\}$.
Then
\begin{align*}\fint_{W(y,s)} \abs{\arr H(z,t)}^2\,dz\,dt
&=
\fint_{W(y,s)} \abs[bigg]{\int_{V(y,s)} \frac{1}{r^\dmn}\arr H_{(x,r)}(z,t)\,d\mu(x,r)}^2\,dz\,dt
.\end{align*}
By H\"older's inequality,
\begin{align*}\fint_{W(y,s)} \abs{\arr H(z,t)}^2\,dz\,dt
&\leq
\fint_{W(y,s)} \int_{V(y,s)} \frac{\mu(V(y,s))}{r^{2n+2}} \abs{\arr H_{(x,r)}(z,t)}^2\,d\mu(x,r)\,dz\,dt
.\end{align*}
Changing the order of integration, we see that
\begin{align*}\fint_{W(y,s)} \abs{\arr H(z,t)}^2\,dz\,dt
&\leq
\int_{V(y,s)} \frac{\mu(V(y,s))}{r^{2n+2}} \fint_{W(y,s)} \abs{\arr H_{(x,r)}(z,t)}^2\,dz\,dt\,d\mu(x,r)
.\end{align*}
A straightforward computation yields that $V(y,s)$ is the ellipsoid
\begin{align}
\label{eqn:V}
V(y,s)
&=
	\left\{(x,r):\frac43\abs{x-y}^2+\left(r-\frac{5}{3}s\right)^2 < \left(\frac{4}{3}s\right)^2\right\}
. \end{align} 
In particular, if $(x,r)\in V(y,s)$ then $\frac{1}{3}r< s< 3r$. Thus, $\abs{W(x,r)}\approx \abs{W(y,s)}$ and so
\begin{align*}\fint_{W(y,s)} \abs{\arr H(z,t)}^2\,dz\,dt
&\leq
C\frac{\mu(V(y,s))}{s^{2n+2}} \int_{V(y,s)} \fint_{W(x,r)} \abs{\arr H_{(x,r)}(z,t)}^2\,dz\,dt\,d\mu(x,r)
.\end{align*}
Recalling the $L^2$ norm of $\arr H_{(x,r)}$, we see that
\begin{align*}
\biggl(\fint_{W(y,s)} \abs{\arr H(z,t)}^2\,dz\,dt\biggr)^{1/2}
&\leq 
\frac{C}{s^\dmn}\mu(V(y,s))
=
\frac{C}{s^\dmn}\int_{V(y,s)} d\mu(x,r)
.\end{align*}
Then
\begin{align*}
\widetilde{\mathfrak{C}}_1^+ (t\arr H)(\tilde x) 
&=
\sup_{Q\owns \tilde x} \frac{1}{\abs{Q}}\int_Q \int_0^{\ell(Q)} \biggl(\fint_{W(y,s)} \abs{t \arr H(z,t)}^2\,dz\,dt\biggr)^{1/2}\,\frac{ds\,dy}{s}
\\&\leq
\sup_{Q\owns \tilde x} \frac{1}{\abs{Q}}\int_Q \int_0^{\ell(Q)} \frac{C}{s^\dmnMinusOne}\int_{V(y,s)} d\mu(x,r)\,\frac{ds\,dy}{s}
.\end{align*}
By formula~\eqref{eqn:V}, if $(y,s)\in Q\times(0,\ell(Q))$ then $V(y,s)\subset 4Q\times(0,3\ell(Q))$. Recall that $V(y,s)=\{(x,r):W(x,r)\cap W(y,s)\neq \emptyset\}$ and so $(x,r)\in V(y,s)$ if and only if $(y,s)\in V(x,r)$. Thus,
\begin{align*}
\widetilde{\mathfrak{C}}_1^+ (t\arr H)(\tilde x) 
&\leq
C\sup_{Q\owns \tilde x} \frac{1}{\abs{Q}}\int_{4Q} \int_0^{4\ell(Q)} \int_{V(x,r)}\,\frac{ds\,dy}{s^\dmn} d\mu(x,r)
.\end{align*}
But $\int_{V(x,r)}\,\frac{ds\,dy}{s^\dmn}$ is a constant. Renaming the variables $(x,r)$ to $(x,t)$, we see that
\begin{align*}
\widetilde{\mathfrak{C}}_1^+ (t\arr H)(\tilde x) 
&\leq
C\sup_{Q\owns \tilde x} \frac{1}{\abs{Q}}\int_{4Q} \int_0^{4\ell(Q)} \frac{t}{t}d\mu(x,t)
=4^nC\mathfrak{C}_1^+(t\mu)(\tilde x)
\end{align*}
as desired.
\end{proof}

\begin{proof}[Proof of Lemma~\ref{lem:N<C}]
Let $\arr u\in L^2_{loc}(\R^\dmn_+)$ be such that 
\begin{equation*}\sup_{\arr H\in L^2_c(\R^\dmn_+)\setminus\{0\}} \frac{\abs{\langle \arr u,\arr H\rangle_{\R^\dmn_+}}} {\doublebar{\widetilde{\mathfrak{C}}_1^+(t\arr H)}_{L^{p'}(\R^n)}}
<\infty.
\end{equation*}
Let $K_\varepsilon={\{(x,t):\varepsilon\leq t\leq 1/\varepsilon,\> \abs{x}\leq 1/\varepsilon\}}$. Define $\arr u_\varepsilon=\1_{K_\varepsilon}\arr u$. By the monotone convergence theorem,
\begin{equation*}\doublebar{\widetilde N_+\arr u}_{L^p(\R^n)} = \lim_{\varepsilon\to 0^+}\doublebar{\widetilde N_+\arr u_\varepsilon}_{L^p(\R^n)}= \sup_{\varepsilon>0}\doublebar{\widetilde N_+\arr u_\varepsilon}_{L^p(\R^n)}.\end{equation*}
Thus we need only bound $\doublebar{\widetilde N_+\arr u_\varepsilon}_{L^p(\R^n)}$, uniformly in~$\varepsilon>0$.
We observe that if $\varepsilon>0$, then $\arr u_\varepsilon\in L^2(\R^\dmn_+)$, and $\widetilde N_+\arr u_\varepsilon$ is bounded and compactly supported.

We now construct an $\arr H_\varepsilon$ that will allow us to bound $\widetilde N_+ \arr u_\varepsilon$. 
Let $W(x,r)=B((x,r),r/2)$, and let $U_\varepsilon(x,r)^2=\fint_{W(x,r)} \abs{\arr u_\varepsilon}^2$, so that $\widetilde N_+\arr u_\varepsilon=N_+U_\varepsilon$.
By \cite[Theorem~5.1 and formula~(2.12)]{AlvM87},
there is a (nonnegative) measure $\mu$ with $\doublebar{{\mathfrak{C}}_1^+ (t\mu)}_{L^{p'}(\R^n)}\leq C_p$ and with \begin{equation*}\doublebar{N_+U_\varepsilon}_{L^p(\R^n)} = \int_{\R^\dmn_+} U_\varepsilon(y,s)\,d\mu(y,s).\end{equation*}
Let 
\begin{equation*}\arr H{}_{(x,r)}^\varepsilon = \begin{cases}\frac{1}{U_\varepsilon(x,r)}\1_{W(x,r)} \arr u_\varepsilon,& U_\varepsilon(x,r)>0,\\ \arr 0, &U_\varepsilon(x,r)=0,\end{cases}\end{equation*}
so that
\begin{equation*}\biggl(\fint_{W(x,r)} \abs{\arr H{}_{(x,r)}^\varepsilon}^2\biggr)^{1/2}\leq 1
,\quad
U_\varepsilon(x,r)
= \frac{1}{\abs{W(x,r)}} \langle \arr u,\arr H{}_{(x,r)}^\varepsilon\rangle_{\R^\dmn_+}
.\end{equation*}
Observe that there is a constant $c_n$ such that $\abs{W(x,r)}=r^\dmn/c_n$ for all $x\in\R^n$ and all $r>0$. Then
\begin{align*}\doublebar{\widetilde N_+\arr u_\varepsilon}_{L^p(\R^n)}
&=
\int_{\R^\dmn_+} U_\varepsilon(x,r)\,d\mu(x,r)
\\&=
\int_{\R^\dmn_+} \frac{c_n}{r^\dmn} \langle \arr u,\arr H{}_{(x,r)}^\varepsilon\rangle_{\R^\dmn_+}\,d\mu(x,r)
.\end{align*}
Changing the order of integration, we see that
\begin{equation*}
\doublebar{\widetilde N_+\arr u_\varepsilon}_{L^p(\R^n)}
=\langle \arr u,\arr H_\varepsilon\rangle_{\R^\dmn_+}
\quad\text{where}\quad
\arr H_\varepsilon(z,t)=\int_{\R^\dmn_+} \frac{c_\dmnMinusOne}{r^\dmn}\arr H{}_{(x,r)}^\varepsilon(z,t)\,d\mu(x,r).\end{equation*}
We observe that $\arr H_\varepsilon$ is compactly supported.
By Lemma~\ref{lem:C:L2} and assumption on~$\mu$,
\begin{align*}
\doublebar{\widetilde{\mathfrak{C}}_1^+ (t\arr H)}_{L^{p'}(\R^n)}
&\leq
C\doublebar{\mathfrak{C}_1^+(t\mu)}_{L^{p'}(\R^n)}
\leq CC_p
\end{align*}
and so
\begin{align*}\doublebar{\widetilde N_+\arr u_\varepsilon}_{L^p(\R^n)}
&=\langle \arr u,\arr H_\varepsilon\rangle_{\R^\dmn_+}
\leq \frac{C_p}{\doublebar{\widetilde{\mathfrak{C}}_1^+(t\arr H_\varepsilon)}_{L^{p'}(\R^n)}} \langle \arr u,\arr H_\varepsilon\rangle_{\R^\dmn_+}
\\&\leq
C_p\sup_{\arr H\in L^2_c(\R^\dmn_+)\setminus\{0\}} 
\frac{\abs{\langle \arr u,\arr H\rangle_{\R^\dmn_+}}} {\doublebar{\widetilde{\mathfrak{C}}_1^+(t\arr H)}_{L^{p'}(\R^n)}} \end{align*}
as desired.
\end{proof}

We will use Lemma~\ref{lem:N<C} to prove the nontangential bounds~\cref{eqn:S:N:intro,eqn:D:N:intro}. In proving the bounds~\cref{eqn:S:N:rough:intro,eqn:D:N:rough:intro}, it will be convenient to be able to introduce an additional derivative in the inner product on the right hand side. Thus, we now prove the following lemma.

\begin{lem}\label{lem:N<C:vertical} Let $u\in \dot W^{m,2}_{loc}(\R^\dmn_+)$ satisfy $\widetilde N_+(\nabla^{m-1} u)\in L^2(\R^n)$. Let $p$ satisfy $1<p<2$ and let $1/p+1/p'=1$.
Then
\begin{equation*}
\doublebar{\widetilde N_+(\nabla^{m-1}u)}_{L^p(\R^n)}
\leq C_p\sup_{\arr \Psi}
\frac{\abs{\langle \arr \Psi,\nabla^{m}u\rangle_{\R^\dmn_+}}} {\doublebar{\widetilde{\mathfrak{C}}_1^+(t\,\partial_t\arr\Psi)} _ {L^{p'}(\R^n)}}
\end{equation*}
where the supremum is over all compactly supported $\arr\Psi\in L^2(\R^\dmn_+)$ that are not identically zero and have a weak vertical derivative also in $L^2(\R^\dmn_+)$.

\end{lem}

\begin{proof}
By Lemma~\ref{lem:N<C},
\begin{equation*}
\doublebar{\widetilde N_+(\nabla^{m-1}u)}_{L^p(\R^n)}
\leq C_p\sup_{\arr H\in L^2_c(\R^\dmn_+)\setminus\{\arr 0\}}
\frac{\abs{\langle \arr H,\nabla^{m-1}u\rangle_{\R^\dmn_+}}} {\doublebar{\widetilde{\mathfrak{C}}_1^+(t\arr H)}_{L^{p'}(\R^n)}}
.\end{equation*}
Choose some such $\arr H$.
Let $\arr \theta(x)=\int_0^\infty \arr H(x,t)\,dt$; because $\arr H$ is compactly supported we have that $\arr \theta(x)\in L^2(\R^n)$. Let $\arr G_T(x,t)=\arr \theta(x)\frac{2}{T}\chi_{(3T/4,5T/4)}(t)$, where $T>0$ is a real number and where $\chi_{(3T/4,5T/4)}$ denotes the characteristic function of the interval $(3T/4,5T/4)$.
Let $\arr H_T$ be such that $\arr H=\arr G_T+\arr H_T$.

Then $\int_0^\infty \arr H_T(x,t)\,dt=0$ for almost every $x\in\R^n$. 
Let 
\begin{equation*}(\Psi_T)_\alpha(x,t)=\int_0^t (H_T)_\gamma(x,s)\,ds\quad\text{where} \quad\alpha=\gamma+\vec e_\dmn,\end{equation*}
where $\vec e_\dmn$ is the unit vector in the $\dmn$th direction,
and let $(\Psi_T)_\alpha=0$ if $\alpha_\dmn=0$. Then $\arr\Psi_T\in L^2(\R^\dmn_+)$ and is compactly supported. Furthermore, $\partial_t (\arr\Psi_T)_\alpha(x,t)=(H_T)_\gamma(x,t)$, and so
$\langle \arr H_T,\nabla^{m-1}u\rangle_{\R^\dmn_+}=-\langle \arr\Psi_T,\nabla^m u\rangle_{\R^\dmn_+}$.

Thus,
\begin{equation*}\frac{\abs{\langle \arr H,\nabla^{m-1}u\rangle_{\R^\dmn_+}}} {\doublebar{\widetilde{\mathfrak{C}}_1^+(t\arr H)}_{L^{p'}(\R^n)}}
=
\frac{\abs{\langle \arr \Psi_T,\nabla^{m}u\rangle_{\R^\dmn_+}-\langle \arr G_T,\nabla^{m-1}u\rangle_{\R^\dmn_+}}} {\doublebar{\widetilde{\mathfrak{C}}_1^+ (t\,\partial_t\arr\Psi_T+t\arr G_T) }_{L^{p'}(\R^n)}}
\end{equation*}
for any $T>0$.

By the definition~\eqref{eqn:Carleson:norm} of $\widetilde{\mathfrak{C}}_1^+$, if $x\in\R^n$ then
\begin{equation*}\widetilde{\mathfrak{C}}_1^+(t\arr G_T)(x)
\leq CT^{-n/2}\doublebar{\arr\theta}_{L^2(\R^n)}.\end{equation*}
Suppose that $T$ is large enough that there is a cube $Q$ with $\ell(Q)=5T/4$ and with $\supp \arr\theta\subset 10Q$. By Lemma~\ref{lem:C:local} with $r=p'$, 
\begin{equation*}\doublebar{\widetilde{\mathfrak{C}}_1^+(t\arr G_T)}_{L^{p'}(\R^n)}
\leq C_{p'}\doublebar{\widetilde{\mathfrak{C}}_1^+(t\arr G_T)}_{L^{p'}(16Q)}
\leq C_{p'} T^{n/p'-n/2}\doublebar{\arr\theta}_{L^2(\R^n)}.\end{equation*}
If $p<2$ and $1/p+1/p'=1$, then $p'>2$ and so $\doublebar{\widetilde{\mathfrak{C}}_1^+(t\arr G_T)}_{L^{p'}(\R^n)}\to 0$ as $T\to \infty$. Because $(G_T)_\gamma+\partial_\dmn(\Psi_T)_\alpha=H_\gamma$, this implies that $\doublebar{\widetilde{\mathfrak{C}}_1^+(t\,\partial_t\arr\Psi_T)}_ {L^{p'}(\R^n)}\to \doublebar{\widetilde{\mathfrak{C}}_1^+(t\arr H)}_{L^{p'}(\R^n)}$ as $T\to \infty$; by assumption $\doublebar{\widetilde{\mathfrak{C}}_1^+(t\arr H)}_{L^{p'}(\R^n)}>0$, and so 
\begin{equation*}\frac{\abs{\langle \arr H,\nabla^{m-1}u\rangle_{\R^\dmn_+}}} {\doublebar{\widetilde{\mathfrak{C}}_1^+(t\arr H)}_{L^{p'}(\R^n)}}
=
\lim_{T\to\infty}
\frac{\abs{\langle \arr \Psi_T,\nabla^{m-1}u\rangle_{\R^\dmn_+}-\langle \arr G_T,\nabla^{m-1}u\rangle_{\R^\dmn_+}}} {\doublebar{\widetilde{\mathfrak{C}}_1^+ (t\,\partial_t\arr\Psi_T) }_{L^{p'}(\R^n)}}
.\end{equation*}

We claim that $\langle \arr G_T,\nabla^{m-1}u\rangle_{\R^\dmn_+}\to 0$ as $T\to\infty$ as well. We compute that
\begin{align*}
{\abs{\langle \arr G_T,\nabla^{m-1} u\rangle_{\R^\dmn_+}}}
&\leq
\fint_{3T/4}^{5T} \int_{\R^n} \abs{\arr\theta(x)} \abs{\nabla^{m-1}u(x,t)}\,dx\,dt
\\&\leq \doublebar{\arr\theta}_{L^2(\R^n)}\fint_{3T/4}^{5T}  \doublebar{\nabla^{m-1}u(\,\cdot\,,t)}_{L^2(\R^n)}\,dt
.\end{align*}
By H\"older's inequality,
\begin{equation*}\fint_{3T/4}^{5T}  \doublebar{\nabla^{m-1}u(\,\cdot\,,t)}_{L^2(\R^n)}\,dt
\leq
\biggl(\fint_{3T/4}^{5T}  \int_{\R^n} \abs{\nabla^{m-1}u(x,t)}^2\,dx\,dt
\biggr)^{1/2}
.\end{equation*}
Introducing a term $\fint_{\abs{x-y}<T/4}\,dy$ and changing the order of integration, we see that
\begin{align*}
\fint_{3T/4}^{5T}  \int_{\R^n} \abs{\nabla^{m-1}u(x,t)}^2\,dx\,dt
&\leq
\int_{\R^n} \fint_{3T/4}^{5T}  
\fint_{\abs{x-y}<T/4}\abs{\nabla^{m-1}u(x,t)}^2\,dx\,dt\,dy.\end{align*}
Observe that 
$\{(x,t):\abs{x-y}<T/4,\abs{t-T}<T/4\}\subset B((z,T),T/2)$, and that the two regions have comparable volume. Recalling the definition~\eqref{dfn:NTM:modified:+} of $\widetilde N_+$, we see that 
\begin{align*}
\fint_{3T/4}^{5T}  \int_{\R^n} \abs{\nabla^{m-1}u(x,t)}^2\,dx\,dt
&\leq
C_n\int_{\R^n} 
\biggl(\fint_{\abs{y-z}<T} \widetilde N_+(\nabla^{m-1} u)(z)\,dz\biggr)^2
\,dy.\end{align*}
Let $F_T(y)=\fint_{\abs{y-z}<T} \widetilde N_+(\nabla^{m-1} u)(z)\,dz$, so that
\begin{equation*}\abs{\langle \arr G_T,\nabla^{m-1} u\rangle_{\R^\dmn_+}}\leq C_n \doublebar{\arr\theta}_{L^2(\R^n)} \doublebar{F_T}_{L^2(\R^n)}.\end{equation*}
By H\"older's inequality, 
\begin{equation*}F_T(y)\leq C_n T^{-n/2} \doublebar{\widetilde N_+(\nabla^{m-1} u)}_{L^2(\R^n)},\end{equation*} and so $F_T(y)\to 0$ as $T\to \infty$ pointwise for each $y\in \R^n$. We also have that $F_T(y)\leq \mathcal{M}(\widetilde N_+(\nabla^{m-1} u))(y)$, where $\mathcal{M}$ is the Hardy-Littlewood maximal function. By boundedness of $\mathcal{M}$ on~$L^2(\R^n)$, $\mathcal{M}(\widetilde N_+(\nabla^{m-1} u))\in L^2(\R^n)$, and so by the dominated convergence theorem, $F_T\to 0$ in $L^2(\R^n)$ as $T\to\infty$.
Thus,
\begin{equation*}\lim_{T\to\infty} {\abs{\langle \arr G_T,\nabla^{m-1} u\rangle_{\R^\dmn_+}}}=0.\end{equation*}

Therefore,
\begin{align*}\frac{\abs{\langle \arr H,\nabla^{m-1}u\rangle_{\R^\dmn_+}}} {\doublebar{\widetilde{\mathfrak{C}}_1^+(t\arr H)}_{L^{p'}(\R^n)}}
&=
\lim_{T\to\infty}
\frac{\abs{\langle \arr \Psi_T,\nabla^{m-1}u\rangle_{\R^\dmn_+}-\langle \arr G_T,\nabla^{m-1}u\rangle_{\R^\dmn_+}}} {\doublebar{\widetilde{\mathfrak{C}}_1^+ (t\,\partial_t\arr\Psi_T) }_{L^{p'}(\R^n)}}
\\&=
\lim_{T\to\infty}
\frac{\abs{\langle \arr \Psi_T,\nabla^{m-1}u\rangle_{\R^\dmn_+}}} {\doublebar{\widetilde{\mathfrak{C}}_1^+ (t\,\partial_t\arr\Psi_T) }_{L^{p'}(\R^n)}}
\leq\sup_{\arr\Psi} \frac{\abs{\langle \arr \Psi,\nabla^{m-1}u\rangle_{\R^\dmn_+}}} {\doublebar{\widetilde{\mathfrak{C}}_1^+ (t\,\partial_t\arr\Psi) }_{L^{p'}(\R^n)}}
.\end{align*}
Recalling that Lemma~\ref{lem:N<C} implies that
\begin{equation*}
\doublebar{\widetilde N_+\arr u}_{L^p(\R^n)}
\leq C_p\sup_{\arr H}
\frac{\abs{\langle \arr H,\nabla^{m-1}u\rangle_{\R^\dmn_+}}} {\doublebar{\widetilde{\mathfrak{C}}_1^+(t\arr H)}_{L^{p'}(\R^n)}}
\end{equation*}
completes the proof.
\end{proof}

\section{The Newton potential}
\label{sec:newton}

We will establish the bounds on layer potentials of Theorem~\ref{thm:potentials} by duality with the Newton potential, as in \cite{HofMayMou15} and \cite[Section~9]{BarHM19A}. Thus, the present section is devoted to duality results for, and bounds on, the Newton potential.

Specifically, we will establish duality between the Newton potential and the double and single layer potentials in Section~\ref{sec:duality}. We will bound the Newton potential in Sections~\cref{sec:newton:boundary,sec:newton:N:regular,sec:newton:N:rough,sec:newton:lusin}. For ease of reference, the main bounds on the Newton potential established in the present paper are all listed in Corollary~\ref{cor:newton:final}. In Section~\ref{sec:p<2}, we will apply the duality results of Section~\ref{sec:duality} and the bounds of Sections \cref{sec:newton:N:regular,sec:newton:N:rough,sec:newton:lusin} to establish bounds on the double and single layer potentials; the bounds of Section~\ref{sec:newton:boundary} will be used in Sections~\cref{sec:newton:N:regular,sec:newton:N:rough,sec:newton:lusin}.

\subsection{Duality}
\label{sec:duality}

In this section we will prove the following lemma, that is, will establish appropriate duality relations between the Newton potential and the double and single layer potentials. In Section~\ref{sec:newton:boundary}, 
we will use these relations to establish bounds on the Newton potential. 
In Section~\ref{sec:p<2} we will reverse the argument and use these duality relations to establish bounds on the double and single layer potentials.

\begin{lem}\label{lem:newton:dual}
Let $L$ be an operator of the form~\eqref{eqn:weak} of order~$2m$ associated to bounded coefficients~$\mat A$ that satisfy the ellipticity condition~\eqref{eqn:elliptic}.

If $\arr\Psi\in L^2(\R^\dmn)$, and if $\arr g\in (\dot W\!A^{1/2,2}_{m-1}(\R^n))^*$, then we have the duality relation
\begin{align}
\label{eqn:S:Newton:dual}
\langle \Tr_{m-1} \Pi^{L^*}\arr \Psi,\arr g\rangle_{\R^n}
&= \langle \arr \Psi, \nabla^m\s^L\arr g\rangle_{\R^\dmn}
.\end{align}
If $\arr\Psi\in L^2(\R^\dmn_+)$, and if $\arr f\in \dot W\!A^{1/2,2}_{m-1}(\R^n)$, then we have the duality relation
\begin{align}
\label{eqn:D:Newton:dual}
\langle \M_{\mat A^*}^- \Pi^{L^*}(\1_+\arr\Psi),\arr f\rangle_{\R^n}
&=
	-\langle \arr \Psi, \nabla^m\D^{\mat A}\arr f\rangle_{\R^\dmn_+}
.\end{align}
If $\mat A$ is $t$-independent in the sense of formula~\eqref{eqn:t-independent}, if $\arr\Psi\in L^2(\R^\dmn)$ is zero in $\R^n\times(-\varepsilon,\varepsilon)$ for some $\varepsilon>0$, and if $\arr f\in \dot W\!A^{1/2,2}_{m-1}(\R^n)$, $\arr g\in (\dot W\!A^{1/2,2}_{m-1}(\R^n))^*$, and $\arr h\in L^2(\R^n)$, then
\begin{align}
\label{eqn:S:Newton:rough:dual}
\langle \Tr_m \Pi^{L^*}\arr\Psi,\arr h\rangle_{\R^n}
&= \langle \arr \Psi, \nabla^m\s^L_\nabla\arr h\rangle_{\R^\dmn}
,\\
\label{eqn:S:Newton:dual:vertical}
\langle \Tr_{m-1} \partial_\dmn\Pi^{L^*}\arr\Psi,\arr g\rangle_{\R^n}
&= -\langle \arr \Psi, \nabla^m\partial_\dmn\s^L\arr g\rangle_{\R^\dmn}
,\\
\label{eqn:D:Newton:dual:vertical}
\langle \M_{\mat A^*}^- \partial_\dmn\Pi^{L^*}(\1_+\arr\Psi),\arr f\rangle_{\R^n}
&=
	\langle \arr \Psi, \nabla^m\partial_\dmn\D^{\mat A}\arr f\rangle_{\R^\dmn_+}
.\end{align}
\end{lem}

\begin{proof}
By the definition~\eqref{eqn:newton} of the Newton potential, if $\arr\Psi\in L^2(\R^\dmn)$ then $\Pi^{L^*}\arr\Psi\in \dot W^{m,2}(\R^\dmn)$.
By the definition~\eqref{dfn:S} of the single layer potential,
\begin{equation*}\langle \Tr_{m-1}\Pi^{L^*}\arr\Psi,\arr g\rangle_{\partial\R^\dmn_+}
=
\langle \nabla^m\Pi^{L^*}\arr\Psi,\mat A\nabla^m\s^L\arr g\rangle_{\R^\dmn}
\end{equation*}
and by the definition~\eqref{eqn:newton} of the Newton potential, we have that the relation \eqref{eqn:S:Newton:dual} is valid.

If $\arr\Psi\in L^2(\R^\dmn_+)$ and $\arr f=\Tr_{m-1}^- F$ for some $F\in \dot W^{m,2}(\R^\dmn_-)$, then by definition of Neumann boundary values
\begin{align*}
\langle \M_{\mat A^*}^- \Pi^{L^*}(\1_+\arr\Psi),\arr f\rangle_{\R^n}
&=
	\langle \mat A^*\nabla^m  \Pi^{L^*}(\1_+\arr\Psi), \nabla^m F\rangle_{\R^\dmn_-}
\\&=
	\langle \nabla^m  \Pi^{L^*}(\1_+\arr\Psi), \1_-\mat A\nabla^m F\rangle_{\R^\dmn}
.\end{align*}
By \cite[Lemma~42]{Bar16}, we have that
\begin{equation}
\label{eqn:newton:adjoint}
\langle \nabla^m  \Pi^{L^*}\arr G, \arr H\rangle_{\R^\dmn}
=\langle \arr G, \nabla^m  \Pi^{L}\arr H\rangle_{\R^\dmn}
\end{equation}
for all $\arr G$, $\arr H\in L^2(\R^\dmn)$. Thus,
\begin{align*}
\langle \M_{\mat A^*}^- \Pi^{L^*}(\1_+\arr\Psi),\arr f\rangle_{\R^n}
&=
	\langle \arr\Psi, \nabla^m  \Pi^{L}(\1_-\mat A\nabla^m F)\rangle_{\R^\dmn_+}
.\end{align*}
By formula~\eqref{dfn:D:newton:-}, the relation \eqref{eqn:D:Newton:dual} is valid.

To prove the relations \cref{eqn:S:Newton:dual:vertical,eqn:D:Newton:dual:vertical}, we review some Sobolev space theory.
If $F\in L^2(\R^\dmn)$ and $h\neq 0$, let $F_h(x,t)=\frac{1}{h}(F(x,t+h)-F(x,t))$. Suppose that $\lim_{h\to 0}F_h$ exists in the the sense of $L^2$ functions, that is, that 
\begin{equation*}\lim_{h\to 0} \doublebar{F_h-G}_{L^2(\R^\dmn)}=0\end{equation*}
for some function $G\in L^2(\R^\dmn)$. Then by the weak definition of derivative, $\partial_\dmn F$ exists and equals~$G$. Conversely, if $F\in L^2(\R^\dmn)\cap \dot W^{1,2}(\R^\dmn)$, then an argument similar to the proof of the Lebesgue differentiation theorem shows that  $\lim_{h\to 0}\doublebar{F_h-\partial_\dmn F}_{L^2(\R^\dmn)}=0$.

By linearity and $t$-independence of~$\mat A$, $\Pi^{L^*}(\arr\Psi_h)=(\Pi^{L^*}\arr\Psi)_h$. If $\arr\Psi\in L^2(\R^\dmn)\cap\dot W^{1,2}(\R^\dmn)$, then taking limits as $h\to 0$ in $L^2(\R^\dmn)$ shows that
\begin{equation}\label{eqn:Newton:vertical}
\nabla^m\Pi^L(\partial_\dmn\arr \Psi)
=\partial_\dmn(\nabla^m\Pi^L\arr \Psi)
.\end{equation}
If in addition $\arr\Psi$ is zero in $\R^n\times(-\varepsilon,\varepsilon)$ for some $\varepsilon>0$, then formulas~\eqref{eqn:S:Newton:dual:vertical} and~\eqref{eqn:D:Newton:dual:vertical} then follow from formulas~\eqref{eqn:S:Newton:dual} and~\eqref{eqn:D:Newton:dual} by integrating by parts. 

To establish formulas~\eqref{eqn:S:Newton:dual:vertical} and~\eqref{eqn:D:Newton:dual:vertical} for arbitrary $\arr\Psi\in L^2(\R^n\times(\varepsilon,\infty))$, fix $\varepsilon>0$, $\arr f\in \dot W\!A^{1/2,2}_{m-1}(\R^n)$ and $\arr g\in (\dot W\!A^{1/2,2}_{m-1}(\R^n))^*$. By formulas~\eqref{dfn:D:newton:-} and~\eqref{dfn:S}, we have that $\D^{\mat A}\arr f\in \dot W^{m,2}(\R^\dmn_+)$ and $\s^L\arr g\in\dot W^{m,2}(\R^\dmn)$, and so by the Caccioppoli inequality $\nabla^m\partial_\dmn\D^{\mat A}\arr f\in L^2(\R^n\times(\varepsilon,\infty))$ and $\nabla^m\partial_\dmn\s^L\arr g\in L^2(\R^n\times(\varepsilon,\infty))$. Thus, the right hand sides of formulas~\eqref{eqn:S:Newton:dual:vertical} and~\eqref{eqn:D:Newton:dual:vertical} (regarded as functions of~$\arr\Psi$) represent bounded linear operators on $L^2(\R^n\times(\varepsilon,\infty))$. Similarly, by the Caccioppoli inequality $\partial_\dmn\Pi^{L^*}$ is a bounded linear operator from $L^2(\R^n\times(\varepsilon,\infty))$ to $\dot W^{m,2}(\R^\dmn_-)$, and so if $\arr g\in (\dot W\!A^{1/2,2}_{m-1}(\R^n))^*$ and $\arr f\in \dot W\!A^{1/2,2}_{m-1}(\R^n)$, then the left hand sides of formulas~\eqref{eqn:S:Newton:dual:vertical} and~\eqref{eqn:D:Newton:dual:vertical} represent bounded linear operators on $L^2(\R^n\times(\varepsilon,\infty))$. Thus, by density, formulas~\eqref{eqn:S:Newton:dual:vertical} and~\eqref{eqn:D:Newton:dual:vertical} are valid for all $\arr\Psi\in L^2(\R^n\times(\varepsilon,\infty))$. A similar argument (or the relations of Section~\ref{sec:lower}) establishes formula~\eqref{eqn:S:Newton:dual:vertical} for $\arr\Psi\in L^2(\R^n\times(-\infty,\varepsilon))$.

Formula~\eqref{eqn:S:Newton:rough:dual} was established in \cite[Section~9]{BarHM19A} under the additional assumption that $\arr\Psi$ is supported in $\R^\dmn_+$.
In the general case, by assumption on~$\supp\arr\Psi$, Lemma~\ref{lem:slices},  and the bound~\eqref{eqn:S:lusin:rough:+} (with $p=2$), we have that both sides of formula~\eqref{eqn:S:Newton:rough:dual} have norm at most 
\begin{equation*}\frac{C}{\sqrt\varepsilon}\doublebar{\arr\Psi}_{L^2(\R^\dmn)}\doublebar{\arr h}_{L^2(\R^n)}\end{equation*}
and in particular are meaningful if this quantity is finite. Thus, we need only establish formula~\eqref{eqn:S:Newton:rough:dual} for $\arr h$  in a dense subset of $L^2(\R^n)$. 
In particular, we need only consider $\arr h$ such that 
formulas~\eqref{eqn:S:S:vertical}, \eqref{eqn:S:S:horizontal}, \eqref{eqn:S:Newton:dual}, and~\eqref{eqn:S:Newton:dual:vertical} (with appropriate~$\arr g$) are valid, and formula~\eqref{eqn:S:Newton:rough:dual} is a straightforward consequence of the given formulas.
\end{proof}

\subsection{The boundary values of the Newton potential}
\label{sec:newton:boundary}

In this section, we will begin to establish bounds on the Newton potential by using Lemma~\ref{lem:newton:dual} and the known bounds \cref{eqn:D:N:rough:+,eqn:S:N:rough:+,eqn:D:lusin:rough:+,eqn:S:lusin:rough:+,eqn:D:N:+,eqn:S:N:+,eqn:D:lusin:+,eqn:S:lusin:+}. The argument is precisely dual to that of Section~\ref{sec:p<2}. Observe that it is the boundary values $\Tr_{m-1}\Pi^{L^*}\arr\Psi$, $\Tr_{m}^-\Pi^{L^*}\arr\Psi$ and $\M_{\mat A}^- \Pi^{L^*}(\1_+\arr\Psi)$ that appear in the bounds~\cref{eqn:S:Newton:dual,eqn:D:Newton:dual,eqn:S:Newton:rough:dual,eqn:S:Newton:dual:vertical}; thus, it is the boundary values of the Newton potential that will be bounded in the present section. We remark that we will not establish all of the bounds on the Newton potential that follow from formulas~\cref{eqn:S:Newton:dual,eqn:D:Newton:dual,eqn:S:Newton:rough:dual,eqn:S:Newton:dual:vertical} and the bounds \cref{eqn:D:N:rough:+,eqn:S:N:rough:+,eqn:D:lusin:rough:+,eqn:S:lusin:rough:+,eqn:D:N:+,eqn:S:N:+,eqn:D:lusin:+,eqn:S:lusin:+}, but only those that we will need in Sections~\cref{sec:newton:N:regular,sec:newton:N:rough,sec:newton:lusin}.

\begin{lem}\label{lem:Newton:boundary}
Let $L$ be an operator of the form~\eqref{eqn:weak} of order~$2m$ associated to bounded $t$-independent coefficients~$\mat A$ that satisfy the ellipticity condition~\eqref{eqn:elliptic}.

Then there is some $\varepsilon$ with the following significance. 
Suppose that $\arr\Psi\in L^2(\R^\dmn)$ is supported in a compact subset of $\R^\dmn_+\cup\R^\dmn_-$.
If $1/p+1/p'=1$ and $p$ lies in the indicated ranges, then 
\begin{align}
\label{eqn:newton:Trace} 
\doublebar{\Tr_{m} \Pi^{L^*}\arr\Psi}_{L^{p'}}
&\leq C(1,L,p)\doublebar{\mathcal{A}_2^*\arr \Psi}_{L^{p'}}
,&2-\varepsilon&<p<p_{1,L}^+
,\\
\label{eqn:newton:neumann}
\doublebar{\M_{\mat A^*}^- \Pi^{L^*}(\1_+\arr \Psi)}_{(\dot W\!A^{0,p}_{m-1})^*}
&\leq C(1,L,p)\doublebar{\mathcal{A}_2^+\arr \Psi}_{L^{p'}}
,&2&\leq p<p_{1,L}^+
.\end{align}
Suppose $\arr H\in L^2(\R^\dmn)$ is supported in a compact subset of $\R^\dmn_+\cup\R^\dmn_-$.
If we normalize $\nabla^{m-1}\Pi^{L^*}\arr H$ as in formulas \cref{eqn:newton:GNS,eqn:newton:GNS:2}, then
\begin{align}
\label{eqn:newton:trace:rough}
\doublebar{\Tr_{m-1}\Pi^{L^*}\arr H}_{L^{p'}} 
&\leq C(0,L,p)\doublebar{\widetilde{\mathfrak{C}}_1^*(t\arr H)}_{L^{p'}}
,&2-\varepsilon<p<p_{0,L}^+.\end{align}
\end{lem}

\begin{proof}
We will use Lemmas~\ref{lem:N>C} and~\ref{lem:N<C} to establish the bound~\eqref{eqn:newton:trace:rough}.
We will need a similar formula involving the area integral to establish the bounds \cref{eqn:newton:Trace,eqn:newton:neumann}. 
Let $T_2^p=\{\psi:\mathcal{A}_2^+\psi\in L^p(\R^n)\}$
with the natural norm. By \cite[p.~316]{CoiMS85}, if $1<p<\infty$ then  under the inner product
\begin{equation*}\langle f,g\rangle  = \int_{\R^\dmn_+} f(x,t)\,g(x,t)\,\frac{dx\,dt}{t}\end{equation*}
the dual space to $T_2^p$ is $T_2^{p'}$.
Thus,
\begin{equation}
\label{eqn:CoiMS85}
\frac{1}{C_p}
\doublebar{\mathcal{A}_2^+(t\arr u)}_{L^{p}(\R^n)}
\leq
\sup_{\arr \Psi}
\frac{\abs{\langle \arr \Psi, \arr u\rangle_{\R^\dmn_+}}
}{\doublebar{\mathcal{A}_2^+\arr \Psi}_{L^{p'}(\R^n)}}
\leq
C\doublebar{\mathcal{A}_2^+(t\arr u)}_{L^{p}(\R^n)}
\end{equation}
where the supremum is over all $\arr\Psi\in L^2_{loc}(\R^\dmn_+)$ such that ${\mathcal{A}_2^+\arr \Psi}\in{L^{p'}(\R^n)}$ and such that the denominator is positive. A similar formula is valid for $\mathcal{A}_2^-$ and~$\mathcal{A}_2^*$.
\begin{rmk}\label{rmk:CoiMS85:compact}
We may take the supremum only over all $\arr\Psi\in L^2_c(\R^\dmn_+)\setminus\{\arr 0\}$, where $L^2_c$ is as in Lemma~\ref{lem:N<C}. 
\end{rmk}

If $1<p<\infty$, then by formula~\eqref{eqn:S:Newton:rough:dual}, and by density of $L^p\cap L^2$ in $L^p$,
\begin{equation*}\doublebar{\Tr_{m} \Pi^{L^*}\arr\Psi}_{L^{p'}(\R^n)}
=\sup_{\arr h\in L^p\cap L^2\setminus\{\arr 0\}} \frac{\abs{\langle\arr\Psi,\nabla^m\s^L_\nabla\arr h\rangle_{\R^\dmn}}}{\doublebar{\arr h}_{L^p(\R^n)}}
.\end{equation*}
By the bound~\eqref{eqn:CoiMS85}, if $1<p<\infty$ then
\begin{equation*}\doublebar{\Tr_{m} \Pi^{L^*}\arr\Psi}_{L^{p'}(\R^n)}
\leq C\sup_{\arr h\in L^p\cap L^2\setminus\{0\}} \frac{\doublebar{\mathcal{A}_2^*\arr\Psi}_{L^{p'}(\R^n)}  \doublebar{\mathcal{A}_2^*(t\nabla^m\s^L_\nabla\arr h)}_{L^p(\R^n)} }{\doublebar{\arr h}_{L^p(\R^n)}}
\end{equation*}
and by the bound~\eqref{eqn:S:lusin:rough:+}, if $2-\varepsilon< p<p_{1,L}^+$ then the bound~\eqref{eqn:newton:Trace} is valid.

Similarly, by formula~\eqref{eqn:D:Newton:dual} and the bounds~\eqref{eqn:CoiMS85} and~\eqref{eqn:D:lusin:rough:+}, if $2\leq p<p_{1,L}^+$ and 
$\arr f\in \dot W\!A^{0,p}_{m-1}(\R^n)\cap \dot W\!A^{1/2,2}_{m-1}(\R^n)$, then
\begin{equation*}\abs{\langle\M_{\mat A^*}^- \Pi^{L^*}(\1_+\arr\Psi),\arr f\rangle_{\R^n}}
\leq C(1,L,p)\doublebar{\mathcal{A}_2^+\arr \Psi}_{L^{p'}(\R^n)}
\doublebar{\arr f}_{\dot W\!A^{0,p}_{m-1}(\R^n)}
.\end{equation*}
By density of $\dot W\!A^{0,p}_{m-1}(\R^n)\cap \dot W\!A^{1/2,2}_{m-1}(\R^n)$ in $\dot W\!A^{1/2,2}_{m-1}(\R^n)$, the bound~\eqref{eqn:newton:neumann} is valid.

We now turn to the bound~\eqref{eqn:newton:trace:rough}. Let $\arr \gamma\in \dot B^{-1/2,2}_2(\R^n)\cap L^p(\R^n)$ for some $p$ with $2-\varepsilon<p<p_{0,L}^+$. Then by formula~\eqref{eqn:S:Newton:dual}, Lemma~\ref{lem:N>C}, and the bound~\eqref{eqn:S:N:+}, we have that
\begin{equation*}\abs{\langle \Tr_{m-1} \Pi^{L^*}\arr H,\arr \gamma\rangle_{\R^n}
}
\leq
C(0,L,p)\doublebar{\arr\gamma}_{L^p(\R^n)}
\doublebar{\widetilde{\mathfrak{C}}_1^*(t\arr H)}_{L^{p'}(\R^n)}
.\end{equation*}
By density of $\dot B^{-1/2,2}_2(\R^n)\cap L^p(\R^n)$ in $L^p(\R^n)$, there is an $\arr f\in L^{p'}(\R^n)$ with $\doublebar{\arr f}_{L^{p'}(\R^n)}\leq C(0,L,p) \doublebar{\widetilde{\mathfrak{C}}_1^*(t\arr H)}_{L^{p'}(\R^n)}$ such that $\langle \arr f,\arr\gamma\rangle_{\R^n}={\langle \Tr_{m-1} \Pi^{L^*}\arr H,\arr \gamma\rangle_{\R^n}
}$ for all $\arr \gamma\in \dot B^{-1/2,2}_2(\R^n)\cap L^p(\R^n)$.
We need only establish that $\arr f=\Tr_{m-1}\Pi^{L^*}\arr H$.

We normalize $\Pi^{L^*}\arr H$ as in formulas~\cref{eqn:newton:GNS,eqn:newton:GNS:2}.
Then there is some $q<\infty$ such that $\nabla^{m-1}\Pi^{L^*}\arr H\in L^q(\R^\dmn)$. 
By Lemma~\ref{lem:slices} and because $\arr H$ is supported away from~$\partial\R^\dmn_\pm$, we have that $\Tr_{m-1}\Pi^{L^*}\arr H\in L^q(\R^n)$ (and in particular is locally integrable). 

If $\arr\varphi\in C^\infty_0(\R^n)$ and $\int_{\R^n}\arr\varphi=0$, then $\arr\varphi\in \dot B^{-1/2,2}_2(\R^n)\cap L^p(\R^n)$, and so
\begin{equation*}\langle \arr\varphi, \Tr_{m-1}\Pi^{L^*}\arr H-\arr f\rangle_{\R^n}=0.\end{equation*}
Thus, $\Tr_{m-1}\Pi^{L^*}\arr H-\arr f$ is a constant. 

We have seen that $\arr f\in L^{p'}(\R^n)$, $\Tr_{m-1}\Pi^{L^*}\arr H\in L^q(\R^n)$, for $p'$, $q<\infty$, and $\arr f-\Tr_{m-1}\Pi^{L^*}\arr H$ is constant; therefore, $\arr f=\Tr_{m-1}\Pi^{L^*}\arr H$, as desired.
\end{proof}

\subsection{Inputs satisfying area integral estimates}
\label{sec:newton:N:regular}

In this section we will continue to establish bounds on the Newton potential. 
The two main results of this section are Lemma~\ref{lem:newton:N:regular}, in which we establish the $L^2$ bound \begin{equation*}\doublebar{\widetilde N_*(\nabla^m \Pi^{L^*} \arr\Psi)}_{L^2(\R^n)}
\leq C \doublebar{\mathcal{A}_2^*\arr\Psi}_{L^2(\R^n)},\end{equation*} and Lemma~\ref{lem:newton:N:-}, in which we establish the $L^{p'}$ bound \begin{equation*}\doublebar{\widetilde N_*(\nabla^{m-j}\partial_t^j \Pi^{L^*} \arr\Psi)}_{L^{p'}(\R^n)}
\leq C(j,L^*,p') \doublebar{\mathcal{A}_2^*\arr\Psi}_{L^{p'}(\R^n)}, \quad p_{j,L^*}^-<p\leq 2.\end{equation*} 
The proof of Lemma~\ref{lem:newton:N:regular} will involve the bound~\eqref{eqn:newton:Trace} and some techniques from the proof of \cite[Lemma~4.1]{HofMayMou15}, while the proof of Lemma~\ref{lem:newton:N:-} will involve Lemma~\ref{lem:newton:N:regular} and some techniques from \cite{Bar19p}.

\begin{lem}
\label{lem:newton:N:regular}
Let $L$ be an operator of the form~\eqref{eqn:weak} of order~$2m$ associated to bounded $t$-independent coefficients~$\mat A$ that satisfy the ellipticity condition~\eqref{eqn:elliptic}.

Let $\arr\Psi\in L^2(\R^\dmn)$ be supported in a compact subset of $\R^\dmn_+\cup\R^\dmn_-$. Then
\begin{equation*}
\doublebar{\widetilde N_*(\nabla^m \Pi^{L^*} \arr\Psi)}_{L^2(\R^n)}
\leq C \doublebar{\mathcal{A}_2^*\arr\Psi}_{L^2(\R^n)}. 
\end{equation*}
\end{lem}

\begin{proof}
Let $z\in \R^n$ and let $(x_0,t_0)$ satisfy $\abs{z-x_0}<\abs{t_0}$. Let $B=B((x_0,t_0),\abs{t_0}/2)$. We wish to bound
\begin{equation*}\fint_{B} \abs{\nabla^m \Pi^{L^*}\arr\Psi}^2\end{equation*}
by a quantity depending only on $z$ and $\arr\Psi$, and not on $x_0$ or~$t_0$.

Let $\Delta(x,r)=\{y\in\R^n:\abs{x-y}<r\}$ denote a disk in $\R^n$ (not $\R^\dmn$). Let $E_k=\Delta(x_0,2^{k+2}\abs{t_0})\times(-2^{k+2}\abs{t_0},2^{k+2}\abs{t_0})$ be a cylinder in $\R^\dmn$.
We define
\begin{align*}
\arr\Psi_0&=\1_{E_0} \arr\Psi,
&
\arr\Psi_k&=\1_{E_k\setminus E_{k-1}}\1_+ \arr\Psi,
\quad k\geq 1,\end{align*}
and let
\begin{equation*}
w_k=\Pi^{L^*}\arr\Psi_k,\quad k\geq 0
.\end{equation*}
Then $\Pi^{L^*}\arr\Psi = w_0+\sum_{k=1}^\infty w_k$.

We begin by bounding $\nabla^m w_0$. By the $L^2$ boundedness of $\nabla^m\Pi^{L^*}$,
\begin{equation*}\fint_{B} \abs{\nabla^m w_0}^2
\leq \frac{C}{\abs{t_0}^\dmn}\doublebar{\arr\Psi_0}_{L^2(\R^\dmn)}^2
.\end{equation*}
By Lemma~\ref{lem:lebesgue:lusin} with $r=2$ and $\kappa=0$, if $\frac{2n}{n+1}<\theta\leq 2$, then
\begin{equation*}\doublebar{\arr \Psi_0}_{L^2(\R^\dmn)}
=\doublebar{\arr \Psi_0}_{L^2(\R^n\times(-4\abs{t_0},4\abs{t_0}))}
\leq \frac{C_{\theta}}{\abs{t_0}^{n/\theta-1/2-n/2}} \doublebar{\mathcal{A}_2^*\arr \Psi_0}_{L^\theta(\R^n)}.\end{equation*}
But $\mathcal{A}_2^*\arr\Psi_0=0$ outside of $\Delta(x_0,8\abs{t_0})\subset\Delta(z,9\abs{t_0})$, and $\mathcal{A}_2^*\arr\Psi_0(x)\leq \mathcal{A}_2^*\arr\Psi(x)$ for all $x\in \Delta(z,9\abs{t_0})$. Thus,
\begin{align*}
\fint_{B} \abs{\nabla^m w_0}^2
&\leq 
\frac{C_{\theta}}{\abs{t_0}^{2n/\theta}} 
\biggl(\int_{\Delta(z,9\abs{t_0})}
(\mathcal{A}_2^*\arr \Psi)^\theta\biggr)^{2/\theta}\!
.\end{align*}
Let $\mathcal{M}$ denotes the Hardy-Littlewood maximal operator (in $\R^n$) given by
\begin{equation*}\mathcal{M} f(z)=\sup_{r>0} \fint_{\abs{y-z}<r} \abs{f(y)}\,dy.\end{equation*}
We then have that
\begin{align}
\label{eqn:newton:N:regular:proof:1}
\fint_{B} \abs{\nabla^m w_0}^2
&\leq 
C_{\theta} \mathcal{M}((\mathcal{A}_2^*\arr \Psi)^\theta)(z)^{2/\theta}
\end{align}
whenever $\frac{2n}{n+1}<\theta\leq 2$.

We now turn to $\arr\Psi_k$, $k\geq 1$. Let $\widetilde w=\sum_{k=1}^\infty w_k$. Observe that $L^*\widetilde w=0$ in $E_0$. 

The following lemma may be proven using the same argument as {\cite[Lem\-ma~3.19]{BarHM18p}}, in which the case of cubes (rather than cylinders) was considered.
\begin{lem}
\label{lem:iterate} 
Let $L$ be an operator of the form~\eqref{eqn:weak} of order~$2m$ associated to bounded $t$-independent coefficients~$\mat A$ that satisfy the ellipticity condition~\eqref{eqn:elliptic}.

Let $x_0\in\R^n$ and let $r>0$, $c>0$, and $\sigma>1$. Let $E=\Delta(x_0,r)\times(-cr,cr)$ and let $\widetilde E=\Delta(x_0,\sigma r)\times(-c\sigma r,c\sigma r)$.
Suppose that $u\in \dot W^{m,2}(\widetilde E)$ and that $Lu=0$ in $\widetilde E$. Let $0\leq j\leq m$. Then there is a constant $C_{c,\sigma}$ depending only on $c$, $\sigma$ and the standard parameters (and in particular independent of $x_0$ and~$r$) such that
\begin{align*}\fint_E \abs{\nabla^j u(x,t)}^2\,dt\,dx
&\leq
C_{c,\sigma}\biggl(r\fint_{\widetilde E}\abs{\partial_t^{j+1} u(x,t)}\,dt\,dx\biggr)^2
\\&\qquad+C_{c,\sigma} \biggl(\fint_{\Delta(x_0,\sigma r)} \abs{\Tr_j u(x)}\,dx\biggr)^2
.\end{align*}
\end{lem}

Observe that $B=B((x_0,t_0),\abs{t_0}/2)\subset \Delta(x_0,\abs{t_0}/2)\times(-3\abs{t_0}/2,3\abs{t_0}/2)$.
Thus by Lemma~\ref{lem:iterate},
\begin{equation}
\label{eqn:newton:N:regular:proof:2}
\biggl(\fint_{B} \abs{\nabla^m \widetilde w}^2\biggr)^{1/2}
\leq
C\fint_{\Delta(x_0,\abs{t_0})}\int_{-2\abs{t_0}}^{2\abs{t_0}} \abs{\partial_\dmn^{m+1}\widetilde w} + C \fint_{\Delta(x_0,\abs{t_0})} \abs{\Tr_m \widetilde w}
.\end{equation}
Recall that $\theta$ is a number with $\frac{2n}{n+1}<\theta\leq 2$.
If $n\geq 1$ then $\theta>1$. Thus by H\"older's inequality,
\begin{align*}\fint_{\Delta(x_0,\abs{t_0})} \abs{\Tr_m \widetilde w}
&\leq
	C\mathcal{M}(\Tr_m w)(z)
	+
	\biggl(\fint_{\Delta(x_0,\abs{t_0})} \abs{\Tr_m w_0}^\theta\biggr)^{1/\theta}
.\end{align*}
By the bound~\eqref{eqn:Meyers:bound}, $p_{1,L}^-\leq \max(1,\frac{2n}{n+2})\leq \frac{2n}{n+1}$, and so 
if $\frac{2n}{n+1}<\theta\leq 2$ then the bound~\eqref{eqn:newton:Trace} is valid for $p'=\theta$. Furthermore, the constant $c(1,L,\theta',2)$ of the bound~\eqref{eqn:Meyers} may be bounded by a constant depending only on $\theta$ and the standard parameters. Thus,
\begin{equation*}\fint_{\Delta(x_0,\abs{t_0})} \abs{\Tr_m  w_0}^\theta
\leq
\frac{1}{\abs{\Delta(x_0,\abs{t_0})}} \int_{\R^n} \abs{\Tr_m  w_0}^\theta
\leq
C_\theta\frac{1}{\abs{\Delta(x_0,\abs{t_0})}} \int_{\R^n} (\mathcal{A}_2^*\arr\Psi_0)^\theta
.\end{equation*}
As before, $\mathcal{A}_2^*\arr\Psi_0\leq \mathcal{A}_2^*\arr\Psi$ and $\mathcal{A}_2^*\arr\Psi_0=0$ outside of $\Delta(z,9\abs{t_0})$, and so
\begin{equation}
\label{eqn:newton:N:regular:proof:3}
\fint_{\Delta(x_0,\abs{t_0})} \abs{\Tr_m \widetilde w}
\leq
	C\mathcal{M}(\Tr_m w)(z)
	+
	C_\theta\mathcal{M}((\mathcal{A}_2^*\arr\Psi)^\theta)(z)^{1/\theta}.
\end{equation}

We are left with the term involving $\partial_\dmn^{m+1}\widetilde w$. 

Choose some $k\geq 1$.
Let $(x,t)\in \Delta(x_0,2\abs{t_0})\times(-2\abs{t_0},2\abs{t_0})=E_{-1}\subseteq E_{k-2}$.
Observe that because $\mat A$ (and thus $\mat A^*$) is $t$-independent, we have that $L^*(\partial_\dmn^{m+1} w_k)=0$ in $E_{k-1}$ for each $k\geq 1$.
By \cite[formula~(29)]{Bar16}, 
if $2m>\dmn$, then 
\begin{equation*}\abs{\partial_t^{m+1}w_k(x,t)} \leq C \biggl(\fint_{E_{k-3/2}} \abs{\partial_s^{m+1}w_k(y,s)}^2\,dy\,ds\biggr)^{1/2}.\end{equation*}
Recall that $w_k=\Pi^{L^*}\arr\Psi_k$.
By the Caccioppoli inequality and boundedness of the Newton potential $L^2(\R^\dmn)\to \dot W^{m,2}(\R^\dmn)$,
\begin{equation*}\abs{\partial_t^{m+1}w_k(x,t)} \leq \frac{C}{(2^k\abs{t_0})^{1+\pdmn/2}} \doublebar{\arr\Psi_k}_{L^2(\R^\dmn)}.\end{equation*}
Observe that $\mathcal{A}_2^*\arr\Psi_k\leq \mathcal{A}_2^*\arr\Psi$ and that $\mathcal{A}_2^*\arr\Psi_k=0$ outside of $\Delta(x_0,2^{k+3}\abs{t_0})\subset \Delta(z,2^{k+4}\abs{t_0})$.
As before, by Lemma~\ref{lem:lebesgue:lusin} with $r=2$ and $\kappa=0$,
\begin{align*}\doublebar{\arr\Psi_k}_{L^2(\R^\dmn)}
&\leq \frac{C_\theta}{(2^k\abs{t_0})^{n/\theta-1/2-n/2}} \doublebar{\mathcal{A}_2^*\arr\Psi_k}_{L^\theta(\R^n)}
\\&\leq
C_\theta(2^k\abs{t_0})^{\pdmn/2} \mathcal{M}((\mathcal{A}_2^*\arr\Psi)^\theta)(z)^{1/\theta}.
\end{align*}
Thus,
\begin{align*}\fint_{\Delta(x_0,\abs{t_0})}\int_{-2\abs{t_0}}^{2\abs{t_0}} \abs{\partial_\dmn^{m+1}\widetilde w}
&\leq 
\sum_{k=1}^\infty
\frac{C_\theta}{2^k} \mathcal{M}((\mathcal{A}_2^*\arr\Psi)^\theta)(z)^{1/\theta}
.\end{align*}
Summing, and applying the bounds~\cref{eqn:newton:N:regular:proof:1,eqn:newton:N:regular:proof:2,eqn:newton:N:regular:proof:3}, we see that if $2m>\dmn$ and $\frac{2n}{n+1}<\theta\leq 2$, then
\begin{align*}
\biggl(\fint_{B} \abs{\nabla^m\Pi^{L^*}\arr\Psi}^2\biggr)^{1/2}
&\leq 
\biggl(\fint_{B} \abs{\nabla^mw_0}^2\biggr)^{1/2}
+\biggl(\fint_{B} \abs{\nabla^m\widetilde w}^2\biggr)^{1/2}
\\&\leq
C_\theta\mathcal{M}((\mathcal{A}_2^*\arr \Psi)^\theta)(z)^{1/\theta}
+ C \mathcal{M}(\Tr_{m} \Pi^{L^*}\arr\Psi)(z).\end{align*}
The right hand side depends only on~$z$, not on $x_0$ or $t_0$, and so
\begin{equation*}\widetilde N_*(\nabla^m \Pi^{L^*} \arr\Psi)(z)
\leq
C_\theta\mathcal{M}((\mathcal{A}_2^*\arr \Psi)^\theta)(z)^{1/\theta}
+ C \mathcal{M}(\Tr_{m} \Pi^{L^*}\arr\Psi)(z).\end{equation*}

By the bound~\eqref{eqn:newton:Trace}, we have that $\doublebar{\Tr_{m} \Pi^{L^*}\arr\Psi}_{L^2(\R^n)}\leq C\doublebar{\mathcal{A}_2^*\arr\Psi}_{L^2(\R^n)}$. Choose $\theta=(2n+1)/(n+1)<2$.
By boundedness of $\mathcal{M}$ on $L^2(\R^n)$ and on $L^{2/\theta}(\R^n)$, we have that
\begin{equation*}
\doublebar{\widetilde N_*(\nabla^m \Pi^{L^*} \Psi)}_{L^2(\R^n)}
\leq C \doublebar{\mathcal{A}_2^*\arr\Psi}_{L^2(\R^n)}. 
\end{equation*}
This completes the proof in the case $2m>\dmn$. 

Suppose now that $2m\leq \dmn$.
Let $\widetilde L = \Delta^M L \Delta^M$ for some large integer~$M$. As shown in the proof of \cite[Theorem~62]{Bar16}, there are constants $a_\xi$ such that
\begin{equation*}\Pi^L\arr \Psi = \Delta^M \Pi^{\widetilde L} \arr{\widetilde \Psi}
\quad\text{where}\quad
\arr {\widetilde\Psi} = \sum_{\abs\xi=2M}\sum_{\abs\beta=m} a_\xi \,\Psi_\beta \,\arr e_{\beta+\xi}
\end{equation*}
where $\arr e_{\beta+\xi}$ is given by formula~\eqref{eqn:e}.
Thus,
\begin{align*}\widetilde N_*(\nabla^m \Pi^{L^*} \arr\Psi)(z)
&=
	\widetilde N_+(\nabla^{m}\Delta^M \Pi^{\widetilde L^*} \arr{\widetilde \Psi})(z)
\end{align*}
and if we choose $M$ so that $2m+4M>\dmn$, then
\begin{equation*}\doublebar{\widetilde N_*(\nabla^{m+2M} \Pi^{\widetilde L^*} \arr{\widetilde \Psi})}_{L^2(\R^n)}
\leq C \doublebar{\mathcal{A}_2^* \arr{\widetilde\Psi}}_{L^2(\R^n)}
\leq C^2 \doublebar{\mathcal{A}_2^* \arr{\Psi}}_{L^2(\R^n)}\end{equation*}
as desired.
\end{proof}

We now extend to bounds for $\mathcal{A}_2^*\arr\Psi\in L^{p'}(\R^n)$, $p<2$.

\begin{lem}\label{lem:newton:N:-}
Let $L$ be an operator of the form~\eqref{eqn:weak} of order~$2m$ associated to bounded $t$-independent coefficients~$\mat A$ that satisfy the ellipticity condition~\eqref{eqn:elliptic}.

Let $j$ be an integer with $0\leq j\leq m$. Let $p$ satisfy $p_{j,L^*}^-<p<2$ and let $1/p+1/p'=1$.
Let $\arr\Psi\in L^2(\R^\dmn)$ be supported in a compact subset of $\R^\dmn_+\cup\R^\dmn_-$. Then
\begin{equation*}
\doublebar{\widetilde N_*(\nabla^{m-j} \partial_t^j \Pi^{L^*} \arr\Psi)}_{L^{p'}(\R^n)}
\leq C(j,L^*,p') \doublebar{\mathcal{A}_2^*\arr\Psi}_{L^{p'}(\R^n)}, \quad p_{j,L^*}^-<p\leq2.
\end{equation*}
\end{lem}

\begin{proof}
The $p=2$ case is Lemma~\ref{lem:newton:N:regular}.
Let
\begin{align*}
u&=\partial_t^j\Pi^{L^*} \arr\Psi, & u_Q &= \partial_t^j\Pi^{L^*}(\1_{10Q\times(-\ell(Q),\ell(Q))} \arr \Psi), &  \Phi_1&={\mathcal{A}_2^*\arr\Psi}
\end{align*}
where $Q$ is any cube in~$\R^n$.
Hereafter the proof closely parallels that of \cite[Theorem~4.12]{Bar19p}, and in fact will use many results of \cite{Bar19p}.
Choose some $p$ with $p_{j,L^*}^-<p<2$. By standard self-improvement properties of reverse H\"older estimates (see, for example, \cite[Chapter~V, Theorem~1.2]{Gia83}), there is a $p_2>p'$ such that the bound~\eqref{eqn:Meyers} is valid for solutions $u$ to $L^*u=0$ and for $p=p_2$. That is, there is a $p_2>p'$ such that $p_2<p_{j,L^*}^+$, with $p_2$ and $c(j,L^*,p_2,2)$ depending only on $p$ and $c(j,L^*,p',2)$.

We have that $u-u_Q\in\dot W^{m,2}_{loc}(10Q\times(-\ell(Q),\ell(Q)))$ and $L^*(u-u_Q)=0$ in $10Q\times (-\ell(Q),\ell(Q))$.
By \cite[Lemma~4.11]{Bar19p} with $v=u-u_Q$, we have that
\begin{align*}
\biggl(\fint_{8Q} \widetilde N^\ell_n(\nabla^{m-j}(u-u_Q))^{p_2}\biggr)^{1/{p_2}}
&\leq
	C(j,L^*,p_2)\biggl(\fint_{10Q}\widetilde N_n^{3\ell} (\nabla^{m-j}(u-u_Q))^2\biggr)^{1/2}
\end{align*}
where $\ell=\ell(Q)/4$ and where $\widetilde N_n^\ell$ is as given in \cite[Section~4]{Bar19p}. In particular, $\widetilde N_n^{3\ell} (\nabla^{m-j} u)\leq \widetilde N_*(\nabla^{m-j} u)$.
By Lemma~\ref{lem:newton:N:regular}, we have that
\begin{equation*}\doublebar{\widetilde N_*(\nabla^{m-j} u)}_{L^2(\R^n)}\leq C\doublebar{\Phi_1}_{L^2(\R^n)}<\infty.
\end{equation*}
Observe that $\mathcal{A}_2^*(\1_{10Q\times(-\ell(Q),\ell(Q))} \arr \Psi)(x)\leq \mathcal{A}_2^*\arr \Psi(x)=\Phi_1(x)$ and is zero if $x\notin 12Q$; thus, again by Lemma~\ref{lem:newton:N:regular}, we have that
\begin{equation*}\doublebar{\widetilde N_*(\nabla^{m-j} u_Q)}_{L^2(\R^n)}\leq C\doublebar{\Phi_1}_{L^2(16Q)}.\end{equation*}
These bounds imply that
\begin{align*}
\doublebar{\widetilde N^\ell_n(\nabla^{m-j} (u-u_Q))}_{L^{p_2}(8Q)}
&\leq
	\frac{C(j,L^*,p_2)}{\abs{Q}^{1/2-1/p_2}}
	\bigl(\doublebar{\Phi_1}_{L^2(16Q)}
	+
	\doublebar{\widetilde N_* (\nabla^{m-j} u)}_{L^2(10Q)} 
	\bigr).\end{align*}
The conditions of \cite[Lemma~4.3]{Bar19p} with $\arr u=\nabla^{m-j} u$ and $\arr u_Q=\nabla^{m-j} u_Q$  are thus satisfied, and so
\begin{equation*}\doublebar{\widetilde N_+(\nabla^{m-j} u)}_{L^{p'}(\R^n)}
\leq C(j,L^*,p') \doublebar{\Phi_1}_{L^{p'}(\R^n)},\end{equation*}
as desired.
\end{proof}

\subsection{Inputs satisfying Carleson estimates}
\label{sec:newton:N:rough}

In this section we will continue to establish bounds on the Newton potential. 
In Lemma~\ref{lem:newton:lusin:rough} we will establish the area integral bound
\begin{equation*}\doublebar{\mathcal{A}_2^*(t\nabla^m \Pi^{L^*}\arr H)}_{L^{p'}(\R^n)}
\leq C(0,L,p) \doublebar{\widetilde{\mathfrak{C}}_1^*(t\arr H)}_{L^{p'}(\R^n)},\quad 2\leq p <p_{0,L}^+,\end{equation*} 
and in Lemmas~\ref{lem:newton:N:rough:-} and~\ref{lem:newton:N:rough} we will establish the nontangential bound
\begin{align*}
\doublebar{\widetilde N_*(\nabla^{m-1}\Pi^{L^*}\arr H)}_{L^{p'}(\R^n)}
&\leq \widetilde C_p \doublebar{\widetilde{\mathfrak{C}}_1^*(t\arr H)}_{L^{p'}(\R^n)}
,\quad p_{1,L^*}^-<p<p_{0,L}^+
\end{align*}
for an appropriate constant~$\widetilde C_p$.

Lemma~\ref{lem:newton:lusin:rough} will be proven by a simple duality argument. The proof of Lem\-ma~\ref{lem:newton:N:rough:-} uses similar techniques to that of Lem\-ma~\ref{lem:newton:N:regular}. Most of the proof of Lemma~\ref{lem:newton:N:rough} will be omitted, as once some notation has been established it may be proven in the same fashion as \cite[Theorem~4.12]{Bar19p} or Lemma~\ref{lem:newton:N:-}.

\begin{lem}\label{lem:newton:lusin:rough}
Let $L$ be an operator of the form~\eqref{eqn:weak} of order~$2m$ associated to bounded $t$-independent coefficients~$\mat A$ that satisfy the ellipticity condition~\eqref{eqn:elliptic}.

Let $\arr H\in L^2(\R^\dmn)$ be supported in a compact subset of $\R^\dmn_+\cup\R^\dmn_-$.
Let $p_{0,L}^+$ be as in the bound~\eqref{eqn:Meyers}. If $2\leq p <p_{0,L}^+$, and if $1/p+1/p'=1$,
then
\begin{equation*}
\doublebar{\mathcal{A}_2^*(t\nabla^m \Pi^{L^*}\arr H)}_{L^{p'}(\R^n)}
\leq C(0,L,p) \doublebar{\widetilde{\mathfrak{C}}_1^*(t\arr H)}_{L^{p'}(\R^n)}
,\quad 2\leq p <p_{0,L}^+.\end{equation*}
\end{lem}

\begin{proof}
By the bound~\eqref{eqn:CoiMS85},
\begin{equation*}\doublebar{\mathcal{A}_2^*(t\nabla^m \Pi^{L^*}\arr H)}_{L^{p'}(\R^n)}
\approx\sup_{\arr \Psi} \frac{\abs{\langle \arr \Psi, \nabla^m \Pi^{L^*}\arr H\rangle_{\R^\dmn}}} {\doublebar{\mathcal{A}_2^*\arr\Psi}_{L^{p}(\R^n)}}. \end{equation*}
We may take the supremum over $\arr\Psi$ supported in a compact subset of $\R^\dmn_+\cup\R^\dmn_-$ such that the denominator is positive and finite. Thus, we may assume $\arr\Psi\in L^2(\R^\dmn)$. By \cite[Lemma~42]{Bar16} (reproduced as formula~\eqref{eqn:newton:adjoint} above),
\begin{equation*}\doublebar{\mathcal{A}_2^*(t\nabla^m \Pi^{L^*}\arr H)}_{L^{p'}(\R^n)}
\approx\sup_{\arr \Psi} \frac{\abs{\langle \nabla^m \Pi^{L}\arr \Psi, \arr H\rangle_{\R^\dmn}}} {\doublebar{\mathcal{A}_2^*\arr\Psi}_{L^{p}(\R^n)}} \end{equation*}
and by Lemma~\ref{lem:N>C},
\begin{equation*}\doublebar{\mathcal{A}_2^*(t\nabla^m \Pi^{L^*}\arr H)}_{L^{p'}(\R^n)}
\leq C_p
\sup_{\arr \Psi} \frac{\doublebar{\widetilde N_*(\nabla^m \Pi^{L}\arr \Psi)}_{L^{p}(\R^n)} \doublebar{\widetilde{\mathfrak{C}}_1^*(t\arr H)}_{L^{p'}(\R^n)}} {\doublebar{\mathcal{A}_2^*\arr\Psi}_{L^{p}(\R^n)}} .\end{equation*}
Using Lemma~\ref{lem:newton:N:-} with $j=0$ and with $p$, $p'$ and $L$, $L^*$ exchanged completes the proof.
\end{proof}

We now establish nontangential estimates.

\begin{lem}\label{lem:newton:N:rough:-}
Let $L$ be an operator of the form~\eqref{eqn:weak} of order~$2m$ associated to bounded $t$-independent coefficients~$\mat A$ that satisfy the ellipticity condition~\eqref{eqn:elliptic}.

Let $\arr H\in L^2(\R^\dmn)$ be supported in a compact subset of $\R^\dmn_+\cup\R^\dmn_-$.
Let $p_{0,L}^+$ be as in the bound~\eqref{eqn:Meyers}.
If $2\leq p <p_{0,L}^+$, and if $1/p+1/p'=1$,
then
\begin{equation*}
\doublebar{\widetilde N_*(\nabla^{m-1} \Pi^{L^*}\arr H)}_{L^{p'}(\R^n)}\leq C(0,L,p) \doublebar{\widetilde{\mathfrak{C}}_1^*(t\arr H)}_{L^{p'}(\R^n)},
\quad 2\leq p<p_{0,L}^+
.\end{equation*}
\end{lem}

\begin{proof} As in the proof of Lemma~\ref{lem:newton:N:regular}, let $\Delta(x,r)=\{y\in\R^n:\abs{x-y}<r\}$, let $z\in\R^n$, and let $B=B((x_0,t_0),\abs{t_0}/2)$ be a Whitney ball with $\abs{x_0-z}<\abs{t_0}$. Let
\begin{equation*}\arr H= \arr H_n+\arr H_f \quad\text{where}\quad \arr H_n = \1_{\Delta(x_0,4\abs{t_0})\times(-4\abs{t_0},4\abs{t_0})}\arr H.
\end{equation*}
Our goal is thus to show that
\begin{equation*}\biggl(\fint_B \abs{\nabla^{m-1}\Pi^{L^*}\arr H_n}^2\biggr)^{1/2}
+\biggl(\fint_B \abs{\nabla^{m-1}\Pi^{L^*}\arr H_f}^2\biggr)^{1/2}\end{equation*}
may be bounded by a quantity depending only on $z$ and $\arr H$, not $x_0$ and~$t_0$.

Fix some $p$ with $2\leq p<p_{0,L}^+$.
As in the proof of Lemma~\ref{lem:newton:N:-}, by standard self-improvement properties of reverse H\"older estimates (see, for example, \cite[Chapter~V, Theorem~1.2]{Gia83}), there is a $\theta>p$ such that the bound~\eqref{eqn:Meyers} is valid for solutions $u$ to $Lu=0$ and for $p=\theta$. That is, there is a $\theta$ such that $p<\theta<p_{0,L}^+$, with $\theta$ and $c(0,L,\theta,2)$ depending only on $p$ and $c(0,L,p,2)$.

If $\dmn\geq 3$, let $r=2$; if $\dmn=2$, let $r$ satisfy $\theta'<r<2$ and be close enough to $2$ that the bound~\eqref{eqn:newton:bound:r} is valid. Let $q$ be as in the bound~\eqref{eqn:newton:GNS} or~\eqref{eqn:newton:GNS:2}. Observe that $r>1$ and so $q>2$.

We begin with $\Pi^{L^*}\arr H_n$. 
By H\"older's inequality, 
\begin{equation*}\biggl(\fint_B \abs{\nabla^{m-1}\Pi^{L^*}\arr H_n}^2\biggr)^{1/2}
\leq
C\abs{t_0}^{-\pdmn/q}
\doublebar{\nabla^{m-1}\Pi^{L^*}\arr H_n}_{L^q(\R^n\times I(t_0))}\end{equation*}
where $I(t_0)=({t_0}/2,\infty)$ if $t_0>0$, and where $I(t_0)=(-\infty,t_0/2)=(-\infty,-\abs{t_0}/2)$ if $t_0<0$.

Recall that $\arr H\in L^2(\R^\dmn)$ and $r\leq 2$, and so $\arr H_n\in L^r(\R^\dmn)$. By the bound~\eqref{eqn:newton:bound:r}, we have that $\nabla^m\Pi^{L^*}\arr H_n\in L^r(\R^n\times I(t_0))$.
By the Gagliardo-Nirenberg-Sobolev inequality and standard extension theorems of Sobolev spaces on a half-space, we have that there is a constant $\arr c$ such that
\begin{equation*}
\doublebar{\nabla^{m-1}\Pi^{L^*}\arr H_n-\arr c}_{L^q(\R^n\times I(t_0))} \leq C_r \doublebar{\nabla^m\Pi^{L^*}\arr H_n}_{L^r(\R^n\times I(t_0))}.
\end{equation*}
By the bound~\eqref{eqn:newton:GNS} or~\eqref{eqn:newton:GNS:2}, $\doublebar{\nabla^{m-1}\Pi^{L^*}\arr H_n}_{L^q(\R^\dmn)}$ is finite, and so $\arr c=0$.

Recall that $1<\theta'<r\leq 2$ and so $\theta'(n+1)/(n+\theta')<r$. By Lemma~\ref{lem:lebesgue:lusin} with $\kappa=1$,
\begin{equation*}\doublebar{\nabla^m\Pi^{L^*}\arr H_n}_{L^r(\R^n\times I(t_0))}
\leq \frac{C_{\theta,r}}{\abs{t_0}^{1+n/\theta'-1/r-n/r}} \doublebar{\mathcal{A}_2^*(t\nabla^m\Pi^{L^*}\arr H_n)}_{L^{\theta'}(\R^n)}\end{equation*}
By Lemma~\ref{lem:newton:lusin:rough}, if $2<\theta<p_{0,L}^+$, then 
\begin{equation*}
\doublebar{\mathcal{A}_2^*(t\nabla^m\Pi^{L^*}\arr H_n)}_{L^{\theta'}(\R^n)}
\leq C(0,L,\theta) \doublebar{\widetilde{\mathfrak C}_1^*(t\arr H_n)}_{L^{\theta'}(\R^n)}.\end{equation*}
Thus,
\begin{equation*}\biggl(\fint_B \abs{\nabla^{m-1}\Pi^{L^*}(\arr H_n)}^2\biggr)^{1/2}
\leq
\frac{C(0,L,\theta)}{\abs{t_0}^{n/\theta'}}\doublebar{\widetilde{\mathfrak C}_1^*(t\arr H_n)}_{L^{\theta'}(\R^n)}
.\end{equation*}
By Lemma~\ref{lem:C:local} with $r=\theta'$, we have that 
\begin{equation*}\doublebar{\widetilde{\mathfrak C}_1^*(t\arr H_n)}_{L^{\theta'}(\R^n)}\leq 
C_\theta \abs{t_0}^{n/\theta'}\mathcal{M}((\widetilde{\mathfrak C}_1^*(t\arr H))^{\theta'})(z)^{1/{\theta'}}\end{equation*}
where $\mathcal{M}$ is the Hardy-Littlewood maximal function.
Thus, 
\begin{equation}
\label{eqn:Hn}\biggl(\fint_B \abs{\nabla^{m-1}\Pi^{L^*}(\arr H_n)}^2\biggr)^{1/2}
\leq C(0,L,\theta) \mathcal{M}((\widetilde{\mathfrak C}_1^*(t\arr H))^{\theta'})(z)^{1/\theta'}.\end{equation}

We now turn to $\Pi^{L^*}\arr H_f$. Recall that $\arr H_f=0$ in $\Delta(x_0,4\abs{t_0})\times(-4\abs{t_0},4\abs{t_0})$. By Lemma~\ref{lem:iterate},
\begin{align*}
\biggl(\fint_B \abs{\nabla^{m-1}\Pi^{L^*}\arr H_f}^2\biggr)^{1/2}
&\leq
C \fint_{\Delta(x_0,\abs{t_0})} \abs{\Tr_{m-1} \Pi^{L^*} \arr H_f}
\\&\qquad
+
C\int_{-2\abs{t_0}}^{2\abs{t_0}} \fint_{\Delta(x_0,\abs{t_0})} \abs{\partial_\dmn^m \Pi^{L^*} \arr H_f}
.\end{align*}
We begin by bounding the trace. 
We have that
\begin{align*}\fint_{\Delta(x_0,\abs{t_0})} \abs{\Tr_{m-1}\Pi^{L^*}\arr H_f}
&\leq
C\mathcal{M}(\Tr_{m-1}\Pi^{L^*}\arr H)(z)
+\fint_{\Delta(x_0,\abs{t_0})} \abs{\Tr_{m-1}\Pi^{L^*}\arr H_n}
.\end{align*}
By the bound~\eqref{eqn:newton:trace:rough}, we have that $\Tr_{m-1}\Pi^{L^*}\arr H\in L^{\theta'}(\R^n)$,
and that
\begin{equation*}
\doublebar{\Tr_{m-1}\Pi^{L^*}\arr H_n}_{L^{\theta'}(\R^n)}\leq C(0,L,\theta) \doublebar{\widetilde{\mathfrak{C}}_1^*(t\arr H_n)}_{L^{\theta'}(\R^n)}
.\end{equation*}
Thus, by H\"older's inequality and Lemma~\ref{lem:C:local},
\begin{align*}
\fint_{\Delta(x_0,\abs{t_0})} \abs{\Tr_{m-1}\Pi^{L^*}\arr H_n}
&\leq
C(0,L,\theta)t_0^{-n/\theta'}\doublebar{\widetilde{\mathfrak{C}}_1^* (t\arr H_n)}_{L^{\theta'}(\R^n)}
\\&\leq
C(0,L,\theta)\mathcal{M}((\widetilde{\mathfrak{C}}_1^*(t\arr H))^{\theta'})(z)^{1/\theta'}
.\end{align*}
Therefore,
\begin{align*}
\biggl(\fint_B \abs{\nabla^{m-1}\Pi^{L^*}\arr H_f}^2\biggr)^{1/2}
&\leq
C\mathcal{M}(\Tr_{m-1}\Pi^{L^*}\arr H)(z)
\\&\qquad
+C(0,L,\theta)\mathcal{M}((\widetilde{\mathfrak{C}}_1^*(t\arr H))^{\theta'})(z)^{1/\theta'}
\\&\qquad+
C\int_{-2\abs{t_0}}^{2\abs{t_0}} \fint_{\Delta(x_0,\abs{t_0})} \abs{\partial_\dmn^m \Pi^{L^*} \arr H_f}
.\end{align*}

We are left with the term involving $\partial_\dmn^m \Pi^{L^*}\arr H_f$. 
Let $w=\partial_\dmn^m\Pi^{L^*}\arr H_f$. By the bound~\eqref{eqn:Meyers}, if $0<\mu<\infty$ then
\begin{equation*}\fint_{-2\abs{t_0}}^{2\abs{t_0}} \fint_{\Delta(x_0,\abs{t_0})}\abs{w}
\leq
C_\mu
\biggl(\fint_{-3\abs{t_0}}^{3\abs{t_0}} \fint_{\Delta(x_0,2\abs{t_0})}\abs{w}^\mu\biggr)^{1/\mu}
.\end{equation*}
Choose $\mu=1/2$. By Lemma~\ref{lem:lebesgue:lusin} with $\theta=r=1/2$ and $\kappa=1$,
\begin{equation*}\biggl(\fint_{-3\abs{t_0}}^{3\abs{t_0}} \fint_{\Delta(x_0,2\abs{t_0})}\abs{w}^{1/2}\biggr)^2
\leq
Ct_0^{-2n-1} \doublebar{\mathcal{A}_2^*( t\1_E w)}_ {L^{1/2}(\R^n)}
\end{equation*}
where $E$ is the region of integration on the left hand side.
Observe that $\mathcal{A}_2^*(t\1_E w)= 0$ outside of $\Delta(x_0,5\abs{t_0})\subset\Delta(z,6\abs{t_0})$.
By H\"older's inequality, if $\theta'>1/2$ then
\begin{equation*}\int_{-2\abs{t_0}}^{2\abs{t_0}} \fint_{\Delta(x_0,\abs{t_0})}\abs{w}
\leq
C\biggl(\fint_{\Delta(z,6\abs{t_0})} \mathcal{A}_2^*(tw)(y)^{\theta'}\,dy\biggr)^{1/\theta'}
.\end{equation*}
Recalling the definitions of $w$ and $\arr H_f$, we see that if $\theta'\geq 1$, then
\begin{align*}
\biggl(\fint_{\Delta(z,6\abs{t_0})} \mathcal{A}_2^*(tw)^{\theta'}\biggr)^{1/\theta'}
&\leq
\biggl(\fint_{\Delta(z,6\abs{t_0})}  \mathcal{A}_2^*(t\nabla^m\Pi^{L^*} \arr H_n)^{\theta'}\biggr)^{1/\theta'}
\\&\quad+
\biggl(\fint_{\Delta(z,6\abs{t_0})} \mathcal{A}_2^*(t\nabla^m\Pi^{L^*}\arr H)^{\theta'} \biggr)^{1/\theta'}
.\end{align*}
By Lemma~\ref{lem:newton:lusin:rough}, if $2\leq \theta<p_{0,L}^+$, then 
\begin{align*}
\biggl(\fint_{\Delta(z,6\abs{t_0})} \mathcal{A}_2^*(tw)^{\theta'}\biggr)^{1/\theta'}
&\leq
\frac{C(0,L,\theta)}{t_0^{n/\theta'}}
\doublebar{\widetilde{\mathfrak{C}}_1^*(t\arr H_n)}_{L^{\theta'}(\R^n)}
\\&\quad
+
\mathcal{M}(\mathcal{A}_2^*(t\nabla^m\Pi^{L^*}\arr H)^{\theta'})(z)^{1/\theta'}
\end{align*}
and by Lemma~\ref{lem:C:local} with $r=\theta'$,
\begin{align*}
\frac{1}{t_0^{n/\theta'}}
\doublebar{\widetilde{\mathfrak{C}}_1^*(t\arr H_n)}_{L^{\theta'}(\R^n)}
&\leq
C\mathcal{M}((\widetilde{\mathfrak{C}}_1^*(t\arr H))^{\theta'})(z)^{1/\theta'}
.\end{align*}
Thus,
\begin{align*}
\biggl(\fint_B \abs{\nabla^{m-1}\Pi^{L^*}\arr H_f}^2\biggr)^{1/2}
&\leq
C\mathcal{M}(\Tr_{m-1}\Pi^{L^*}\arr H)(z)
\\&\qquad
+
C\mathcal{M}(\mathcal{A}_2^*(t\nabla^m\Pi^{L^*}\arr H)^{\theta'})(z)^{1/\theta'}
\\&\qquad
+C(0,L,\theta)\mathcal{M}((\widetilde{\mathfrak{C}}_1^*(t\arr H))^{\theta'})(z)^{1/\theta'}
.\end{align*}
Combining this estimate with the bound~\eqref{eqn:Hn} yields that
\begin{align*}
\widetilde N_*(\nabla^{m-1} \Pi^{L^*}(\1_+\arr H))(z)
&\leq
C\mathcal{M}(\Tr_{m-1}\Pi^{L^*}\arr H)(z)
\\&\qquad
+
C\mathcal{M}(\mathcal{A}_2^*(t\nabla^m\Pi^{L^*}\arr H)^{\theta'})(z)^{1/\theta'}
\\&\qquad
+C(0,L,\theta)\mathcal{M}((\widetilde{\mathfrak{C}}_1^*(t\arr H))^{\theta'})(z)^{1/\theta'}
.\end{align*}
Recall that $p<\theta<p_{0,L}^+$, so that $p'/\theta'>1$, and that $c(0,L,\theta,2)$ depends only on $p$ and $c(0,L,p,2)$.
By the bound~\eqref{eqn:newton:trace:rough}, Lemma~\ref{lem:newton:lusin:rough}, and the $L^{p'}$ and $L^{p'/\theta'}$-boundedness of~$\mathcal{M}$, we have that the lemma follows from the above bound.
\end{proof}

The techniques of \cite{Bar19p} let us extend the range of $p$ in our nontangential bound.

\begin{lem}\label{lem:newton:N:rough}
Let $L$ be an operator of the form~\eqref{eqn:weak} of order~$2m$ associated to bounded $t$-independent coefficients~$\mat A$ that satisfy the ellipticity condition~\eqref{eqn:elliptic}.

Let $\arr H\in L^2(\R^\dmn)$ be supported in a compact subset of $\R^\dmn_+\cup\R^\dmn_-$.
Let $p_{j,L}^+$ be as in the bound~\eqref{eqn:Meyers}.
If $p_{1,L^*}^-<p<2$, and if $1/p+1/p'=1$, then
\begin{align*}
\doublebar{\widetilde N_*(\nabla^{m-1}\Pi^{L^*}\arr H)}_{L^{p'}(\R^n)}
&\leq 
C(1,L^*,p')\doublebar{\widetilde{\mathfrak{C}}_1^*(t\arr H)}_{L^{p'}(\R^n)}
,\quad p_{1,L^*}^-<p<2
.\end{align*}
\end{lem}

\begin{proof}
Let
\begin{align*}
u&=\Pi^{L^*}(\1_+\arr H), &
u_Q &= \Pi^{L^*}(\1_{10Q\times(-\ell(Q),\ell(Q))}\arr H), &
\Phi_1&=\widetilde{\mathfrak{C}}_1^*(t\arr H)
, & j&=1
.\end{align*}
The proof is similar to the proof of \cite[Theorem~4.12]{Bar19p} or Lemma~\ref{lem:newton:N:-} and will be omitted.
\end{proof}

\subsection{Area integral estimates on the Newton potential}
\label{sec:newton:lusin}

In this section we will establish area integral estimates on the Neumann potential beyond Lemma~\ref{lem:newton:lusin:rough}. We recall that the Fatou-type estimates on Neumann boundary values established in \cite{BarHM19B} involve area integral estimates and not nontangential estimates; thus, in light of formulas~\eqref{eqn:D:Newton:dual} and~\eqref{eqn:D:Newton:dual:vertical}, area integral estimates are necessary to bound the double layer potential. We will also expand the range of $p$ in the nontangential bound of Lemma~\ref{lem:newton:N:-}. For ease of reference, all of our nontangential and area integral bounds on the Newton potential are listed in Corollary~\ref{cor:newton:final}.

\begin{lem} \label{lem:lusin:regular:2}
Let $L$ be an operator of the form~\eqref{eqn:weak} of order~$2m$ associated to bounded $t$-independent coefficients~$\mat A$ that satisfy the ellipticity condition~\eqref{eqn:elliptic}.

Let $\arr\Psi\in L^2(\R^\dmn_+)$ be compactly supported. Then 
\begin{align*}
\doublebar{\mathcal{A}_2^-(t\nabla^m\partial_t\Pi^{L^*}(\1_+\arr\Psi))}_{L^2(\R^n)}
&\leq
C\doublebar{\mathcal{A}_2^+\arr \Psi}_{L^2(\R^n)}
.\end{align*}

\end{lem}

\begin{proof}
By boundedness of the Newton potential (see Section~\ref{sec:dfn:potential}), $\Pi^{L^*}(\1_+\arr\Psi)\in \dot W^{m,2}(\R^\dmn)$. By the definition~\eqref{eqn:newton} of $\Pi^{L^*}$, ${L^*}(\Pi^{L^*}(\1_+\arr\Psi))=0$ in $\R^\dmn_-$. By \cite[Lemma~5.2]{Bar17} or \cite[formula~(2.26)]{BarHM17}, we have the Green's formula
\begin{equation*}
\1_-\nabla^m\Pi^{L^*}(\1_+\arr\Psi)
=\nabla^m\D^{\mat A^*} (\Tr_{m-1}^- \Pi^{L^*}(\1_+\arr\Psi))
+\nabla^m\s^{L^*}(\M_{\mat A^*}^- \Pi^{L^*}(\1_+\arr\Psi))
\end{equation*}
away from $\partial\R^\dmn_\pm$.
This formula may also be derived from formula \eqref{dfn:D:newton:-} for the double layer potential (with $F=\Pi^{L^*}(\1_+\arr\Psi)$), and from the definitions~\eqref{eqn:neumann:W2}, \eqref{eqn:newton} and \eqref{dfn:S} of $\M_{\mat A^*}^-$, $\Pi^{L^*}$, and~$\s^{L^*}$.

Thus
\begin{align*}
\doublebar{\mathcal{A}_2^-(t\nabla^m\partial_t\Pi^{L^*}(\1_+\arr\Psi))}_{L^2(\R^n)}
&\leq
\doublebar{\mathcal{A}_2^-(t\nabla^m\partial_t\D^{\mat A^*} (\Tr_{m-1}^- \Pi^{L^*}(\1_+\arr\Psi)))}_{L^2(\R^n)}
\\&\qquad+
\doublebar{\mathcal{A}_2^-(t\nabla^m\partial_t\s^{L^*}(\M_{\mat A^*}^- \Pi^{L^*}(\1_+\arr\Psi)))}_{L^2(\R^n)}
.\end{align*}
By the bound~\eqref{eqn:D:lusin:+} with $p=2$,
\begin{align*}\doublebar{\mathcal{A}_2^-(t\nabla^m\partial_t\D^{\mat A^*} (\Tr_{m-1}^- \Pi^{L^*}(\1_+\arr\Psi)))}_{L^2(\R^n)}
&\leq
C\doublebar{\Tr_{m-1}^- \Pi^{L^*}(\1_+\arr\Psi)}_{\dot W\!A^{1,2}_{m-1}(\R^n)}
,\end{align*}
and by definition of ${\dot W\!A^{1,2}_{m-1}(\R^n)}$ and the bound~\eqref{eqn:newton:Trace},
\begin{equation*}
\doublebar{\Tr_{m-1}^- \Pi^{L^*}(\1_+\arr\Psi)}_{\dot W\!A^{1,2}_{m-1}(\R^n)}
\leq \doublebar{\Tr_{m}^- \Pi^{L^*}(\1_+\arr\Psi)}_{L^2(\R^n)}
\leq C\doublebar{\mathcal{A}_2^+\arr\Psi}_{L^2(\R^n)},
\end{equation*}
as desired.

We will apply a similar argument to the second term. By the bound~\eqref{eqn:newton:neumann} with $p=2$,
\begin{equation*}
\abs{\langle\M_{\mat A^*}^- \Pi^{L^*}(\1_+\arr\Psi),\arr f\rangle_{\R^n}}
\leq C\doublebar{\mathcal{A}_2^+\arr \Psi}_{L^2(\R^n)}
\doublebar{\arr f}_{\dot W\!A^{0,2}_{m-1}(\R^n)}
.\end{equation*}
By boundedness of the Newton potential $L^2(\R^\dmn)\mapsto \dot W^{m,2}(\R^\dmn)$, and by the definition~\eqref{eqn:neumann:W2} of Neumann boundary data, we also have that
\begin{equation*}
\abs{\langle\M_{\mat A^*}^- \Pi^{L^*}(\1_+\arr\Psi),\arr f\rangle_{\R^n}}
\leq C\doublebar{\arr \Psi}_{L^2(\R^\dmn_+)}
\doublebar{\arr f}_{\dot W\!A^{1/2,2}_{m-1}(\R^n)}
.\end{equation*}
Thus, $\M_{\mat A^*}^- \Pi^{L^*}(\1_+\arr\Psi)$ extends to a bounded linear operator on $\dot W\!A^{0,2}_{m-1}(\R^n)+\dot W\!A^{1/2,2}_{m-1}(\R^n)$, and by the Hahn-Banach theorem extends to a bounded linear operator on $L^2(\R^n)+\dot B^{1/2,2}_2(\R^n)$. By standard duality arguments, there is a $\arr g\in L^2(\R^n)\cap \dot B^{-1/2,2}_2(\R^n)$ such that $\langle \arr g,\arr\varphi\rangle_{\R^n}=\langle \M_{\mat A^*}^- \Pi^{L^*}(\1_+\arr\Psi),\arr\varphi\rangle_{\R^n}$ for all $\arr\varphi\in \dot W\!A^{1/2,2}_{m-1}(\R^n)$. 
We may ensure that $\doublebar{\arr g}_{L^2(\R^n)} \leq C\doublebar{\mathcal{A}_2^+\arr \Psi}_{L^2(\R^n)}$ by carefully choosing the norm in $L^2(\R^n)+\dot B^{1/2,2}_2(\R^n)$.

By the definition~\eqref{dfn:S} of the single layer potential, we have that $\s^{L^*}\arr g=\s^{L^*}(\M_{\mat A^*}^- \Pi^L(\1_+\arr\Psi))$. Thus,
\begin{align*}
\doublebar{\mathcal{A}_2^-(t\nabla^m\partial_t\Pi^{L^*}(\1_+\arr\Psi))}_{L^2(\R^n)}
&\leq
C\doublebar{\mathcal{A}_2^+\arr\Psi}_{L^2(\R^n)}
+
\doublebar{\mathcal{A}_2^-(t\nabla^m\partial_t\s^{L^*}\arr g)}_{L^2(\R^n)}
\end{align*}
and the given bound on~$\doublebar{\arr g}_{L^2(\R^n)}$ and the bound~\eqref{eqn:S:lusin:+} completes the proof.
\end{proof}

We now establish area integral estimates for a wider range of~$p$.

\begin{lem}\label{lem:lusin:+} 
Let $L$ be an operator of the form~\eqref{eqn:weak} of order~$2m$ associated to bounded $t$-independent coefficients~$\mat A$ that satisfy the ellipticity condition~\eqref{eqn:elliptic}.

Let $\arr H$ and $\arr \Psi$ be elements of $L^2(\R^\dmn)$ that are supported in compact subsets of $\R^\dmn_+\cup\R^\dmn_-$ and $\R^\dmn_+$, respectively. Let $p_{j,L}^-$ be as in formula~\ref{eqn:Meyers:p-}.
If $1/p+1/p'=1$, and if $p_{1,L^*}^-< p< 2$, then we have the bounds
\begin{align*}
\doublebar{\mathcal{A}_2^*(t\nabla^{m}\Pi^{L^*}\arr H)}_{L^{p'}(\R^n)}
&\leq C(1,L^*,p') \doublebar{\widetilde{\mathfrak{C}}_1^*(t\arr H)}_{L^{p'}(\R^n)}
,& p_{1,L^*}^-&< p< 2
,\\
\doublebar{\mathcal{A}_2^-(t\nabla^{m}\partial_t\Pi^{L^*}(\1_+\arr \Psi))}_{L^{p'}(\R^n)}
&\leq C(1,L^*,p') \doublebar{\mathcal{A}_2^+\arr\Psi}_{L^{p'}(\R^n)}
,& p_{1,L^*}^-&< p<2
.\end{align*}
\end{lem}
\begin{proof}
We will use \cite[Lemma~6.2]{Bar19p}.

For ease of notation we will consider only $\mathcal{A}_2^-$ in both cases; a similar argument or Section~\ref{sec:lower} establishes the bound on $\mathcal{A}_2^+(t\nabla^{m}\Pi^{L^*}\arr H)$.
We make one of the following two choices of notation:
\begin{align*}
u&=\Pi^{L^*}\arr H, &
u_Q&=\Pi^{L^*}\arr H_Q,  &
\Phi_1&=\widetilde{\mathfrak{C}}_1^*(t\arr H), &
\text{or}
\\ 
u&=\partial_t\Pi^{L^*}(\1_+\arr\Psi),&
u_Q&=\partial_t\Pi^{L^*}\arr\Psi_Q, & 
\Phi_1&=\mathcal{A}_2^+\arr\Psi
\end{align*}
where
\begin{align*}
\arr H_Q&=\1_{10Q\times(-\ell(Q),\ell(Q))}\arr H,  
& 
\arr\Psi_Q&=(\1_{11Q\times(0,2\ell(Q))}\arr\Psi).
\end{align*}
Observe that $\mathcal{A}_2^+\arr\Psi_Q(x)\leq \mathcal{A}_2^+\arr\Psi(x)$ and $\mathcal{A}_2^+\arr\Psi_Q(x)=0$ whenever $x\notin 15Q$, while by Lemma~\ref{lem:C:local} we have that $\doublebar{\widetilde{\mathfrak{C}}_1^*(t\arr H_Q)}_{L^r(\R^n)}\leq C\doublebar{\widetilde{\mathfrak{C}}_1^*(t\arr H)}_{L^r(16Q)}$ for any $1<r<\infty$.

By definition of $\Pi^{L^*}$ and the Caccioppoli inequality, we have that
\begin{gather*}
u-u_Q\in \dot W^{m,2}(10Q\times(-\ell(Q),\ell(Q))),
\\
L^*(u-u_Q)=0\text{ in }10Q\times(-\ell(Q),\ell(Q))
.\end{gather*}
By Lemmas~\ref{lem:newton:lusin:rough} and~\ref{lem:lusin:regular:2}, we have that
\begin{align*}
\mathcal{A}_2^-(t\nabla^m u)&\in L^2(\R^n)
,\\
\doublebar{\mathcal{A}_2^-(t\nabla^m u_Q)}_{L^2(\R^n)}
&\leq C\doublebar{\Phi_1}_{L^2(16Q)}
.\end{align*}
By Lemmas~\ref{lem:newton:N:-} and~\ref{lem:newton:N:rough}, if $p_{1,L^*}^-<p\leq 2$ then
\begin{align*}
\doublebar{\widetilde N_*(\nabla^{m-1}u)}_{L^{p'}(\R^n)}
&\leq C(1,L^*,p') \doublebar{\Phi_1}_{L^{p'}(\R^n)}
,\\
\doublebar{\widetilde N_*(\nabla^{m-1} u_Q)}_{L^2(10Q)}
&\leq C\doublebar{\Phi_1}_{L^2(16Q)}
.\end{align*}
By \cite[Lemma~6.2]{Bar19p}, the conclusion is valid.
\end{proof}

For the sake of completeness, we establish a few final bounds on the Newton potential; for ease of reference we list all of our nontangential and area integral bounds on the Newton potential in the following corollary.

\begin{cor}
\label{cor:newton:final}
Let $L$ be an operator of the form~\eqref{eqn:weak} of order~$2m$ associated to bounded $t$-independent coefficients~$\mat A$ that satisfy the ellipticity condition~\eqref{eqn:elliptic}.

Let $\arr\Psi\in L^2(\R^\dmn)$ and $\arr H\in L^2(\R^\dmn)$ be supported in compact subsets of $\R^\dmn_+\cup\R^\dmn_-$. 
Let $j$ be an integer with $0\leq j\leq m$. Let $1/p+1/p'=1$.
If $p$ lies in the given ranges, then
\begin{align}
\label{eqn:newton:N:rough}
\doublebar{\widetilde N_*(\nabla^{m-1}\Pi^{L^*}\arr H)}_{L^{p'}(\R^n)}
&\leq \widetilde C_p \doublebar{\widetilde{\mathfrak{C}}_1^*(t\arr H)}_{L^{p'}(\R^n)}
,& p_{1,L^*}^-&<p<p_{0,L}^+
,\\
\label{eqn:newton:lusin:rough}
\doublebar{\mathcal{A}_2^*(t\nabla^{m}\Pi^{L^*}\arr H)}_{L^{p'}(\R^n)}
&\leq \widetilde C_p \doublebar{\widetilde{\mathfrak{C}}_1^*(t\arr H)}_{L^{p'}(\R^n)}
,& p_{1,L^*}^-&< p<p_{0,L}^+
,\\
\label{eqn:newton:N:regular}
\doublebar{\widetilde N_*(\nabla^{m-j} \partial_t^j \Pi^{L^*} \arr\Psi)}_{L^{p'}(\R^n)}
&\leq \widetilde C_p \doublebar{\mathcal{A}_2^*\arr\Psi}_{L^{p'}(\R^n)}, \quad &p_{j,L^*}^-&<p<p_{1,L}^+
,\\
\label{eqn:newton:lusin:regular}
\doublebar{\mathcal{A}_2^-(t\nabla^{m}\partial_t\Pi^{L^*}(\1_+\arr \Psi))}_{L^{p'}(\R^n)}
&\leq \widetilde C_p \doublebar{\mathcal{A}_2^+\arr\Psi}_{L^{p'}(\R^n)}
,& p_{1,L^*}^-&< p<p_{1,L}
\end{align}
where $\widetilde C_p$ depends only on the standard parameters, $p$, and the constants $c(k,L,p,2)$ (if $p>2$) or $c(k,L^*,p',2)$ (if $p<2$) in the bound~\eqref{eqn:Meyers}, for appropriate values of~$k$.
\end{cor}

\begin{proof}
The bounds~\eqref{eqn:newton:N:rough} and~\eqref{eqn:newton:lusin:rough} were established in Lemmas~\ref{lem:newton:N:rough:-} and~\ref{lem:newton:N:rough} and in Lemmas \ref{lem:newton:lusin:rough} and~\ref{lem:lusin:+}, respectively.

The $p\leq 2$ case of the bound \cref{eqn:newton:N:regular} was established in Lemmas~\ref{lem:newton:N:-}.
To establish the $p>2$ case, we may take $j=0$.
The bound then follows from Lemma~\ref{lem:N<C}, formula~\eqref{eqn:newton:adjoint}, and the bounds \eqref{eqn:CoiMS85} and~\eqref{eqn:newton:lusin:rough}, as in the proof of Lemma~\ref{lem:newton:lusin:rough}. 

The $p\leq 2$ case of the bound~\eqref{eqn:newton:lusin:regular} was established in Lemmas~\ref{lem:lusin:regular:2}--\ref{lem:lusin:+}. As in the proof of the bound~\eqref{eqn:newton:N:regular}, we will establish the $p>2$ case by duality.
By formulas~\cref{eqn:Newton:vertical,eqn:newton:adjoint}, if $\arr G$ and $\arr H$ are in $L^2(\R^\dmn)\cap \dot W^{1,2}(\R^\dmn)$, then 
\begin{equation}
\label{eqn:newton:adjoint:vertical}
\langle \partial_\dmn\nabla^m\Pi^{L^*}\arr G,\arr H\rangle_{\R^\dmn}
=-\langle \arr G,\partial_\dmn\nabla^m\Pi^{L}\arr H\rangle_{\R^\dmn}
.\end{equation}
If $\arr G$ is supported in $J$ and $\arr H$ is supported in~$K$, where $J$ and $K$ are disjoint compact sets, then by the Caccioppoli inequality, both the left hand side and right hand side are at most $C_{J,K}\doublebar{\arr G}_{L^2(J)}\doublebar{\arr H}_{L^2(K)}$; thus, by density, formula~\eqref{eqn:newton:adjoint:vertical} is valid whenever $\arr G\in L^2(\R^\dmn)$ and $\arr H\in L^2(\R^\dmn)$ have disjoint compact support.

We may now see that the $p>2$ case of the bound~\eqref{eqn:newton:lusin:regular} follows from the bound \eqref{eqn:CoiMS85}, formula~\eqref{eqn:newton:adjoint:vertical}, and the $p<2$ case of the bound~\eqref{eqn:newton:lusin:regular} (that is, Lemma~\ref{lem:lusin:+}). 
\end{proof}

\begin{rmk}
\label{rmk:twosides}
The nontangential bounds \cref{eqn:newton:N:rough,eqn:newton:N:regular} and the area integral estimate~\eqref{eqn:newton:lusin:rough} involve the two-sided operators $\widetilde N_*$ and~$\mathcal{A}_2^*$, while the bound \eqref{eqn:newton:lusin:regular} involves one-sided operators $\mathcal{A}_2^+$ and~$\mathcal{A}_2^-$.

This restriction cannot be removed. 
Let $F$ be a function that is smooth and supported in the Whitney ball $B((0,1),1/4)$.
Let $\arr\Psi=\mat A\nabla^m F$. 
It follows from the definition of $\Pi^L$ in Section~\ref{sec:dfn:potential} that $F=\Pi^L(\mat A\nabla^m F)=\Pi^L\arr\Psi$. Thus,
\begin{equation*}\doublebar{\mathcal{A}_2^+(t\nabla^m \partial_t \Pi^L\arr\Psi)}_{L^p(\R^n)} =\doublebar{\mathcal{A}_2^+(t\nabla^m \partial_t F)}_{L^p(\R^n)},
.\end{equation*}
By the ellipticity condition~\eqref{eqn:elliptic} and the definition~\eqref{dfn:lusin:+} of $\mathcal{A}_2^+$, if $0<p<\infty$ then
\begin{equation*}
\doublebar{\mathcal{A}_2^+(\nabla^m F)}_{L^p(\R^n)}
\approx \doublebar{\nabla^m F}_{L^2(B((0,1),1/4))}\approx
\doublebar{\mathcal{A}_2^+\arr\Psi}_{L^p(\R^n)}
\end{equation*}
where the constants of approximation depend on~$p$. Thus, $\doublebar{\mathcal{A}_2^+(\nabla^m F)}_{L^p(\R^n)}
\leq C_p
\doublebar{\mathcal{A}_2^+\arr\Psi}_{L^p(\R^n)}
.$

But for any fixed number $\widetilde C$, we may choose $F$ so that
\begin{equation*}\doublebar{\mathcal{A}_2^+(t\nabla^m \partial_t F)}_{L^p(\R^n)}
\not\leq \frac{\widetilde C}{C_p}\doublebar{\mathcal{A}_2^+(\nabla^m F)}_{L^p(\R^n)}\end{equation*}
and so
\begin{equation*}\doublebar{\mathcal{A}_2^+(t\nabla^m \partial_t \Pi^L\arr\Psi)}_{L^p(\R^n)}
\not\leq \widetilde C\doublebar{\mathcal{A}_2^+\arr\Psi}_{L^p(\R^n)}.\end{equation*}
Thus, no two-sided analogue to the bound~\eqref{eqn:newton:lusin:regular} is possible.
\end{rmk}

\section{The double and single layer potentials}
\label{sec:p<2}

In this section we will prove Theorem~\ref{thm:potentials}.

We will establish estimates on the double and single layer potentials using the duality results of Lemma~\ref{lem:newton:dual} and the bounds on the Newton potential of Corollary~\ref{cor:newton:final}. Recall that Lemma~\ref{lem:newton:dual} involves the Dirichlet and Neumann boundary values of the Newton potential along $\R^n=\partial\R^\dmn_\pm$, while Corollary~\ref{cor:newton:final} yields nontangential and area integral bounds, that is, bounds in the interior of~$\R^\dmn_\pm$. Thus, we will need Fatou type theorems to pass from  Corollary~\ref{cor:newton:final} to useful estimates on boundary values. 

We will list three Fatou type theorems from \cite{BarHM19B,Bar19p} in Section~\ref{sec:fatou:neumann}. These theorems suffice to prove the bounds~\cref{eqn:S:N:intro,eqn:S:lusin:intro,eqn:S:lusin:rough:intro,eqn:D:N:intro,eqn:D:lusin:intro,eqn:D:lusin:rough:intro}; the arguments will be given in Section~\ref{sec:D:S:regular}.
The bounds \cref{eqn:S:lusin:intro,eqn:S:lusin:rough:intro} allow us to eliminate a technical assumption in certain results of \cite{BarHM19B}; these simplified theorems will be stated in Section~\ref{sec:fatou:dirichlet}, after the bounds \cref{eqn:S:lusin:intro,eqn:S:lusin:rough:intro} have been established, and will be used in Section~\ref{sec:D:S:rough} to establish the bounds \cref{eqn:S:N:rough:intro,eqn:D:N:rough:intro}.

\subsection{Fatou type theorems}
\label{sec:fatou:neumann}

In this section we list some known results concerning boundary values of functions that satisfy nontangential or area integral estimates.

We begin with the following theorem concerning Dirichlet boundary values.
\begin{lem}[{\cite[Lemma~5.1]{Bar19p}}]
\label{lem:fatou:N}
Let $\arr u$ be defined and locally square integrable in $\R^\dmn_+$. Suppose that $\widetilde N_+\arr u\in L^p(\R^n)$ for some $p$ with $1<p\leq\infty$.
Suppose that $\Trace^+ \arr u$ exists in the sense of formula~\eqref{eqn:trace}; that is, there is an array of functions $\Trace^+\arr u$ such that
\begin{equation*}\lim_{t\to 0^+} \int_K \abs{\arr u(x,t)-\Trace^+\arr u(x)}\,dx=0\end{equation*}
for any compact set $K\subset\R^n$.
Then $\Trace^+ \arr u$ satisfies
\begin{equation*}\doublebar{\Trace^+\arr u}_{L^p(\R^n)}\leq \doublebar{\widetilde N_+\arr u}_{L^p(\R^n)}.\end{equation*}
\end{lem}

We will also need the following Fatou type theorems for Neumann boundary values. We remark that in \cite{BarHM19B}, these theorems are stated for solutions $v$, $w$ to $Lv=Lw=0$ in $\R^\dmn_+$ with $\mathcal{A}_2^+(t\nabla^m v)$, $\mathcal{A}_2^+(t\nabla^m\partial_t w)\in L^p(\R^n)$. We will usually apply these theorems to solutions $v$, $w$ to $L^*v=L^*w=0$ in $\R^\dmn_-$ with $\mathcal{A}_2^-(t\nabla^m v)$, $\mathcal{A}_2^-(t\nabla^m\partial_t w)\in L^{p'}(\R^n)$; we have modified the theorem statements accordingly.

\begin{thm}[{\cite[Theorem~6.1]{BarHM19B}}]
\label{thm:BHM17pB:rough:neumann}
Let $L$ be an operator of order~$2m$ of the form~\eqref{eqn:weak} associated to bounded $t$-independent coefficients~$\mat A$ that satisfy the ellipticity condition~\eqref{eqn:elliptic}.

Let $1<p<\infty$ and let $1/p+1/p'=1$.
Let $v$ satisfy $\mathcal{A}_2^-(t\nabla^m v)\in L^{p'}(\R^n)$ and $L^*v=0$ in $\R^\dmn_-$. If $p<2$, suppose further that $v\in \dot W^{m,2}(\R^n\times(-\infty,-\sigma))$ for all $\sigma>0$, albeit possibly with norm that approaches $\infty$ as $\sigma\to 0^+$.

Then for all $\varphi\in C^\infty_0(\R^\dmn)$ we have that
\begin{equation*}\abs{\langle \Tr_{m-1}\varphi,\M_{\mat A^*}^- v\rangle_{\R^n}}
\leq C_{p} \doublebar{\Tr_{m-1}\varphi}_{\dot W\!A^{1,p}_{m-1}(\R^n)} \doublebar{\mathcal{A}_2^-(t\nabla^m v)}_{L^{p'}(\R^n)}\end{equation*}
where $\M_{\mat A^*}^-v$ is as in \cite[Section~2.3.2]{BarHM19B}. In particular, if $v\in\dot W^{m,2}(\R^\dmn_-)$, then by  \cite[Lemma~2.4]{BarHM19B} $\M_{\mat A^*}^-v$ is as in formula~\eqref{eqn:neumann:W2}.
\end{thm}

The theorem as stated in \cite{BarHM19B} requires that $v\in \dot W^{m,2}(\R^n\times(-\infty,-\sigma))$ for all~$p$; however, if $p\geq 2$ then this condition follows from Lemma~\ref{lem:lebesgue:lusin} or its predecessor \cite[Remark~5.3]{BarHM19B}.

\begin{thm}[{\cite[Theorem~6.2]{BarHM19B}}]
\label{thm:BHM17pB:regular:neumann}
Let $L$ be an operator of order $2m$ of the form~\eqref{eqn:weak} associated to bounded $t$-independent coefficients~$\mat A$ that satisfy the ellipticity condition~\eqref{eqn:elliptic}.

Let $1<p<\infty$ and let $1/p+1/p'=1$.
Let $w$ satisfy $\mathcal{A}_2^-(t\nabla^m \partial_t w)\in L^{p'}(\R^n)$, $\widetilde N_-(\nabla^m w)\in L^{p'}(\R^n)$, and 
$L^*w=0$ in~$\R^\dmn_-$.
If $p<2$, we impose the additional condition that   $\partial_\dmn w\in \dot W^{m,2}(\R^n\times(-\infty,-\sigma))$ for all $\sigma>0$. 

Then
for all $\varphi\in C^\infty_0(\R^\dmn)$ we have that
\begin{multline*}\abs{\langle \Tr_{m-1}\varphi,\M_{\mat A^*}^- w\rangle_{\R^n}}
\\\leq C_{p} \doublebar{\Tr_{m-1}\varphi}_{\dot W\!A^{0,p}_{m-1}(\R^n)} \bigl(\doublebar{\mathcal{A}_2^-(t\nabla^m \partial_t w)}_{L^{p'}(\R^n)} + \doublebar{\widetilde N_-(\nabla^m  w)}_{L^{p'}(\R^n)}\bigr)
\end{multline*}
where $\M_{\mat A^*}^-w$ is as in formula~\eqref{eqn:neumann:intro}.
\end{thm}

\subsection{The bounds 
(\ref{eqn:S:N:intro}--\ref{eqn:D:lusin:rough:intro})
}
\label{sec:D:S:regular}

In this section we will prove most of Theorem~\ref{thm:potentials}; specifically, we will establish the estimates~\cref{eqn:S:N:intro,eqn:S:lusin:intro,eqn:S:lusin:rough:intro,eqn:D:N:intro,eqn:D:lusin:intro,eqn:D:lusin:rough:intro}.
Throughout this section we will let $L$ and $\mat A$ be as in Theorem~\ref{thm:potentials}; that is, $L$ will be an operator of the form~\eqref{eqn:weak} of order~$2m$ associated to bounded $t$-independent coefficients $\mat A$ that satisfy the ellipticity condition~\eqref{eqn:elliptic}.

\subsubsection*{The estimate~\eqref{eqn:S:N:intro}}
By Lemma~\ref{lem:N<C},  if $1<p<\infty$ then
\begin{align*}
\doublebar{\widetilde N_*(\nabla^m \s^L\arr g)}_{L^p(\R^n)}
&\leq 
	C_p\sup_{\arr H} \frac{\abs{\langle \arr H,\nabla^m \s^L\arr g\rangle_{\R^\dmn}}} {\doublebar{\widetilde{\mathfrak{C}}_1^*(t\arr H)}_{L^{p'}(\R^n)}}
\end{align*}
where the supremum is over all $\arr H\in L^2(\R^\dmn)$ supported in a compact subset of $\R^\dmn_+\cup\R^\dmn_-$ such that the denominator is positive. 
By formula~\eqref{eqn:S:Newton:dual}, if $\arr g\in \dot B^{-1/2,2}_2(\R^n)$ then
\begin{align*}
\doublebar{\widetilde N_*(\nabla^m \s^L\arr g)}_{L^p(\R^n)}
&\leq 
	C_p\sup_{\arr H} \frac{\abs{\langle \Tr_{m-1} \Pi^{L^*}\arr H,\arr g\rangle_{\R^n} }} {\doublebar{\widetilde{\mathfrak{C}}_1^*(t\arr H)}_{L^{p'}(\R^n)}}
.\end{align*}
Because $\Pi^{L^*}\arr H\in \dot W^{m,2}(\R^\dmn)$, we have that $\Tr_{m-1}\Pi^{L^*}\arr H$ exists in the sense of Sobolev spaces, and thus in the sense of formulas \cref{eqn:trace,eqn:Dirichlet}. 
By Lemma~\ref{lem:fatou:N},
\begin{equation*}\doublebar{\Tr_{m-1} \Pi^{L^*}\arr H}_{L^{p'}(\R^n)}
\leq \doublebar{\widetilde N_*(\nabla^{m-1} \Pi^{L^*}\arr H)}_{L^{p'}(\R^n)}\end{equation*}
and so by Lemmas~\ref{lem:newton:N:rough:-} and~\ref{lem:newton:N:rough}, if $p_{1,L^*}^-<p<p_{0,L}^+$ and $\arr g\in \dot B^{-1/2,2}_2(\R^n)\cap L^p(\R^n)$ then
\begin{align*}
\doublebar{\widetilde N_*(\nabla^m \s^L\arr g)}_{L^p(\R^n)}
&\leq 
	\widetilde C_p \doublebar{\arr g}_{L^p(\R^n)}
\end{align*}
where $\widetilde C_p$ is as in Corollary~\ref{cor:newton:final}.
By density, the bound~\eqref{eqn:S:N:intro} is valid.

\subsubsection*{The estimate~\eqref{eqn:D:N:intro}}
By Lemma~\ref{lem:N<C} and formula~\eqref{eqn:D:Newton:dual}, if $\arr\varphi\in \dot W\!A^{1/2,2}_{m-1}(\R^n)$ then
\begin{align*}
\doublebar{\widetilde N_+(\nabla^m \D^{\mat A}\arr \varphi)}_{L^p(\R^n)}
&\leq C_p
	\sup_{\arr H} \frac{\abs{\langle \arr H,\nabla^m \D^{\mat A}\arr \varphi\rangle_{\R^\dmn_+}}} {\doublebar{\widetilde{\mathfrak{C}}_1^+(t\arr H)}_{L^{p'}(\R^n)}}
\\&= C_p
	\sup_{\arr H} \frac{\abs{\langle \M_{\mat A^*}^- \Pi^{L^*}(\1_+\arr H) , \arr \varphi\rangle_{\R^n} }} {\doublebar{\widetilde{\mathfrak{C}}_1^+(t\arr H)}_{L^{p'}(\R^n)}}
.\end{align*}
By Theorem~\ref{thm:BHM17pB:rough:neumann}, 
if $\arr\varphi=\Tr_{m-1}\Phi$ for some $\Phi\in C^\infty_0(\R^\dmn)$, then
\begin{align*}
\doublebar{\widetilde N_+(\nabla^m \D^{\mat A}\arr \varphi)}_{L^p(\R^n)}
&\leq
	C_p\sup_{\arr H} \frac{\doublebar{\arr\varphi}_{\dot W\!A^{1,p}_{m-1}(\R^n)} \doublebar{\mathcal{A}_2^-(t\nabla^m \Pi^{L^*}(\1_+\arr H))}_{L^{p'}(\R^n)}} {\doublebar{\widetilde{\mathfrak{C}}_1^+(t\arr H)}_{L^{p'}(\R^n)}}
.\end{align*}
By Lemmas~\ref{lem:newton:lusin:rough} and~\ref{lem:lusin:+}, if $p_{1,L^*}^-<p<p_{0,L}^+$ then
\begin{align*}
\doublebar{\widetilde N_+(\nabla^m \D^{\mat A}\arr \varphi)}_{L^p(\R^n)}
&\leq
	\widetilde C_p\doublebar{\arr\varphi}_{\dot W\!A^{1,p}_{m-1}(\R^n)}
.\end{align*}
We establish a bound on ${\widetilde N_-(\nabla^m \D^{\mat A}\arr \varphi)}$ using Section~\ref{sec:lower} and extend to all $\arr\varphi\in \dot W\!A^{1,p}_{m-1}(\R^n)$ by density.
This completes the proof of the bound~\eqref{eqn:D:N:intro}.

\subsubsection*{The estimate~\eqref{eqn:S:lusin:intro}}
By the bound~\eqref{eqn:CoiMS85} and formula~\eqref{eqn:S:Newton:dual:vertical}, if $1<p<\infty$ and $\arr g\in \dot B^{-1/2,2}_2(\R^n)$, then
\begin{align*}
\doublebar{\mathcal{A}_2^*(t\nabla^m\partial_t\s^L\arr g)}_{L^p(\R^n)}
&\leq C_p
	\sup_{\arr \Psi}
	\frac{\abs{\langle \arr \Psi, \nabla^m\partial_\dmn\s^L\arr g\rangle_{\R^\dmn}}
	}{\doublebar{\mathcal{A}_2^*\arr \Psi}_{L^{p'}(\R^n)}}
\\&= C_p
	\sup_{\arr \Psi}
	\frac{\abs{\langle \Tr_{m-1} \partial_\dmn\Pi^{L^*}\arr\Psi,\arr g\rangle_{\R^n}}
	}{\doublebar{\mathcal{A}_2^*\arr \Psi}_{L^{p'}(\R^n)}}
.\end{align*}
We may take $\arr\Psi$ to be supported away from $\partial\R^\dmn_\pm$. By Lemma~\ref{lem:slices}, $\Tr_m\Pi^{L^*}\arr\Psi$ exists in the sense of formula~\eqref{eqn:Dirichlet}, and so by Lemma~\ref{lem:fatou:N}, 
\begin{align*}
\doublebar{\mathcal{A}_2^*(t\nabla^m\partial_t\s^L\arr g)}_{L^p(\R^n)}
&\leq
	C_p\sup_{\arr \Psi}
	\frac{\doublebar{\widetilde N_-( \nabla^{m-1}\partial_\dmn\Pi^{L^*}(\1_+\arr\Psi))}_{L^{p'}(\R^n)}
	\doublebar{\arr g}_{L^p(\R^n)}
	}{\doublebar{\mathcal{A}_2^+\arr \Psi}_{L^{p'}(\R^n)}}
.\end{align*}
By the bound~\eqref{eqn:newton:N:regular} with $j=1$, if $p^-_{1,L^*}<p<p_{1,L}^+$ and if $\arr g\in \dot B^{-1/2,2}_2(\R^n)\cap L^p(\R^n)$, then
\begin{align}
\label{eqn:S:lusin:-}
\doublebar{\mathcal{A}_2^*(t\nabla^m\partial_t\s^L\arr g)}_{L^p(\R^n)}
&\leq
	\widetilde C_p\doublebar{\arr g}_{L^p(\R^n)}
	, & p^-_{1,L^*}<p<p_{1,L}^+
.\end{align}
By density, the bound~\eqref{eqn:S:lusin:intro} is valid.

\subsubsection*{The estimate~\eqref{eqn:D:lusin:intro}}
By the bound~\eqref{eqn:CoiMS85} and formula~\eqref{eqn:D:Newton:dual:vertical}, if $\arr\varphi\in \dot W\!A^{1/2,2}_{m-1}(\R^n)$, then
\begin{align*}
\doublebar{\mathcal{A}_2^+(t\nabla^m\partial_t\D^{\mat A}\arr \varphi)}_{L^p(\R^n)}
&\leq C_p
	\sup_{\arr \Psi}
	\frac{\abs{\langle \arr \Psi, \nabla^m\partial_t\D^{\mat A}\arr \varphi\rangle_{\R^\dmn_+}}
	}{\doublebar{\mathcal{A}_2^+\arr \Psi}_{L^{p'}(\R^n)}}
\\&= C_p
	\sup_{\arr \Psi}
	\frac{\abs{\langle \M_{\mat A^*}^- \partial_\dmn\Pi^{L^*}(\1_+\arr\Psi),\arr \varphi\rangle_{\R^n} }
	}{\doublebar{\mathcal{A}_2^+\arr \Psi}_{L^{p'}(\R^n)}}
.\end{align*}
By 
Theorem~\ref{thm:BHM17pB:rough:neumann} and the bound~\eqref{eqn:newton:lusin:regular}, if $p_{1,L^*}^-<p<p_{1,L}^+$ and $\arr\varphi=\Tr_{m-1}^+\Phi$ for some $\Phi\in C^\infty_0(\R^\dmn)$, then
\begin{align*}
\doublebar{\mathcal{A}_2^+(t\nabla^m\partial_t\D^{\mat A}\arr \varphi)}_{L^p(\R^n)}
&\leq C_p
	\sup_{\arr \Psi}
	\frac{\doublebar{\mathcal{A}_2^-(t\nabla^m \partial_t\Pi^{L^*}(\1_+\arr\Psi))}_{L^{p'}(\R^n)} \doublebar{\arr\varphi}_{\dot W\!A^{1,p}_{m-1}(\R^n)}
	}{\doublebar{\mathcal{A}_2^+\arr \Psi}_{L^{p'}(\R^n)}}
\\&\leq C_p\doublebar{\arr\varphi}_{\dot W\!A^{1,p}_{m-1}(\R^n)}
.\end{align*}
As before, we may use density arguments and Section~\ref{sec:lower} to complete the proof of the bound~\eqref{eqn:D:lusin:intro}.

\subsubsection*{The estimate~\eqref{eqn:S:lusin:rough:intro}}
By the bound~\eqref{eqn:CoiMS85}, formula~\eqref{eqn:S:Newton:rough:dual}, and  Lemma~\ref{lem:fatou:N}, if $1<p<\infty$ and $\arr h\in L^2(\R^n)\cap L^p(\R^n)$, then
\begin{align*}
\doublebar{\mathcal{A}_2^*(t\nabla^m\s^L_\nabla\arr h)}_{L^p(\R^n)}
&\leq C_p
	\sup_{\arr \Psi}
	\frac{\abs{\langle \arr \Psi, \nabla^m\s^L_\nabla\arr h\rangle_{\R^\dmn}}
	}{\doublebar{\mathcal{A}_2^*\arr \Psi}_{L^{p'}(\R^n)}}
= C_p
	\sup_{\arr \Psi}
	\frac{\langle \Tr_m \Pi^{L^*}\arr\Psi,\arr h\rangle_{\R^n}
	}{\doublebar{\mathcal{A}_2^*\arr \Psi}_{L^{p'}(\R^n)}}
\\&\leq C_p
	\sup_{\arr \Psi}
	\frac{
	\doublebar{\widetilde N_*(\nabla^m\Pi^{L^*}\arr\Psi)}_{L^{p'}(\R^n)} \doublebar{\arr h}_{L^p(\R^n)}
	}{\doublebar{\mathcal{A}_2^*\arr \Psi}_{L^{p'}(\R^n)}}
.\end{align*}
By density and the bound~\eqref{eqn:newton:N:regular} with $j=0$, if $p_{0,L^*}^-<p<p_{1,L}^+$ then
\begin{align}
\label{eqn:S:lusin:rough:-}
\doublebar{\mathcal{A}_2^*(t\nabla^m\s^L_\nabla\arr h)}_{L^p(\R^n)}
&\leq \widetilde C_p \doublebar{\arr h}_{L^p(\R^n)}
,&p_{0,L^*}^-<p<p_{1,L}^+
.\end{align}
Thus, the bound~\eqref{eqn:S:lusin:rough:intro} is valid.

\subsubsection*{The estimate~\eqref{eqn:D:lusin:rough:intro}}
By the bound~\eqref{eqn:CoiMS85} and formula~\eqref{eqn:D:Newton:dual}, if $1<p<\infty$ and $\arr f\in \dot W\!A^{1/2,2}_{m-1}(\R^n)$ then
\begin{align*}
\doublebar{\mathcal{A}_2^+(t\nabla^m\D^{\mat A}\arr f)}_{L^p(\R^n)}
&\approx
	\sup_{\arr \Psi}
	\frac{\abs{\langle \arr \Psi, \nabla^m\D^{\mat A}\arr f\rangle_{\R^\dmn_+}}
	}{\doublebar{\mathcal{A}_2^+\arr \Psi}_{L^{p'}(\R^n)}}
=
	\sup_{\arr \Psi}
	\frac{\abs{\langle \M_{\mat A^*}^- \Pi^{L^*}(\1_+\arr\Psi),\arr f\rangle_{\R^n} }
	}{\doublebar{\mathcal{A}_2^+\arr \Psi}_{L^{p'}(\R^n)}}
.\end{align*}
By Theorem~\ref{thm:BHM17pB:regular:neumann}, if $\arr f=\Tr_{m-1}F$ for some $F\in C^\infty_0(\R^\dmn)$ then
\begin{multline*}{\abs{\langle \M_{\mat A^*}^- \Pi^{L^*}(\1_+\arr\Psi),\arr f\rangle_{\R^n} }
	}
\\\leq C_p {
	\doublebar{\widetilde N_-(\nabla^m\Pi^{L^*}(\1_+\arr\Psi)) + \mathcal{A}_2^-(t\nabla^m\partial_t\Pi^{L^*}(\1_+\arr\Psi))}_{L^{p'}(\R^n)}
	\doublebar{\arr f}_{\dot W\!A^{0,p}_{m-1}(\R^n)}}\end{multline*}
provided the right hand side is finite.
Thus, by the bounds~\eqref{eqn:newton:N:regular} and~\eqref{eqn:newton:lusin:regular}, if $p_{0,L^*}^-<p<p_{1,L}^+$ then
\begin{equation*}
\doublebar{\mathcal{A}_2^+(t\nabla^m\D^{\mat A}\arr f)}_{L^p(\R^n)}
\leq C\doublebar{\arr f}_{\dot W\!A^{0,p}_{m-1}(\R^n)}
.\end{equation*}
By density and by Section~\ref{sec:lower}, we have that the bound~\eqref{eqn:D:lusin:rough:intro} is valid.

\subsection{Further Fatou type theorems}
\label{sec:fatou:dirichlet}

In order to establish the bounds~\cref{eqn:S:N:rough:intro,eqn:D:N:rough:intro}, we will need further Fatou type theorems.

The Fatou theorems \cite[Theorem~5.1 and~5.2]{BarHM19B} contain a technical assumption involving the single layer potential. As observed in \cite[Remark~5.3]{BarHM19B}, this technical assumption is true if $p\geq 2$; given that the bounds~\cref{eqn:S:lusin:intro,eqn:S:lusin:rough:intro} have been established (see the bounds~\cref{eqn:S:lusin:-,eqn:S:lusin:rough:-} above), we have that this technical assumption is true for a broader range of~$p$. Thus, we will now restate the parts of {\cite[Theorems~5.1, 5.2 and 6.2]{BarHM19B}} necessary for the proofs of the bounds~\cref{eqn:S:N:rough:intro,eqn:D:N:rough:intro}. As in Section~\ref{sec:fatou:neumann}, we have interchanged the roles of $L$ and $L^*$, $p$ and $p'$, and $\R^\dmn_+$ and $\R^\dmn_-$ relative to their roles in \cite{BarHM19B}.

In \cite{BarHM19B}, $p_j^+$ is defined as $p_j^+=\min(p_{j,L}^+,p_{j,L^*}^+)$; however, a careful examination of the proofs in \cite{BarHM19B} yields that the results are valid for $p_{j,L}^\pm$ and $p_{j,L^*}^\pm$  as indicated below.

\begin{thm}[{\cite[Theorem~5.1]{BarHM19B}}]
\label{thm:BHM17pB:rough:dirichlet}
Let $L$ be an operator of order~$2m$ of the form~\eqref{eqn:weak} associated to bounded $t$-independent coefficients~$\mat A$ that satisfy the ellipticity condition~\eqref{eqn:elliptic}.

Let $p_{1,L^*}^-<p<\infty$ and let $1/p+1/p'=1$.
Let $v$ satisfy $\mathcal{A}_2^-(t\nabla^m v)\in L^{p'}(\R^n)$ and $L^*v=0$ in $\R^\dmn_-$. If $p<2$, suppose further that $v\in \dot W^{m,2}(\R^n\times(-\infty,-\sigma))$ for all $\sigma>0$, albeit possibly with norm that approaches $\infty$ as $\sigma\to 0^+$.

Then $\Tr_{m-1}^- v$ exists in the sense of formula~\eqref{eqn:trace}, and there is some constant array $\arr c$ such that
\begin{equation*}\doublebar{\Tr_{m-1}^- v-\arr c}_{L^{p'}(\R^n)}\leq C(1,L^*,p')\doublebar{\mathcal{A}_2^-(t\nabla^m v)}_{L^{p'}(\R^n)}.\end{equation*}
\end{thm}

\begin{thm}[{\cite[Theorems~5.2 and~6.2]{BarHM19B}}]
\label{thm:BHM17pB:regular:dirichlet}
Let $L$ be an operator of order $2m$ of the form~\eqref{eqn:weak} associated to bounded $t$-independent coefficients~$\mat A$ that satisfy the ellipticity condition~\eqref{eqn:elliptic}.

Let $p_{0,L^*}^-<p<\infty$ and let $1/p+1/p'=1$.
Let $w\in \dot W^{m,2}_{loc}(\R^\dmn_-)$ satisfy 
$L^*w=0$ in $\R^\dmn_-$ and $\mathcal{A}_2^-(t\nabla^m \partial_t w)\in L^{p'}(\R^n)$.
If $p<2$, we impose the additional condition that   $\partial_\dmn w\in \dot W^{m,2}(\R^n\times(-\infty,\sigma))$ for all $\sigma>0$. 

If there is some $t<0$ such that $\nabla^m w(\,\cdot\,,t)\in L^{p'}(\R^n)$, then $\Tr_m^- w$ exists in the sense of formula~\eqref{eqn:trace} and satisfies 
\begin{equation*}
\doublebar{\Tr_{m}^- w}_{L^{p'}(\R^n)}
\leq C(0,L^*,p')\doublebar{\mathcal{A}_2^-(t\nabla^m \partial_t w)}_{L^{p'}(\R^n)}
.\end{equation*}
We also have the uniform bound
\begin{equation*}
\sup_{t>0} \doublebar{\nabla^m w(\,\cdot\,,t)}_{L^{p'}(\R^n)}
\leq C(0,L^*,p')\doublebar{\mathcal{A}_2^-(t\nabla^m \partial_t w)}_{L^{p'}(\R^n)}
\end{equation*}
and the limits
\begin{equation*}
\lim_{t\to\infty} \doublebar{\nabla^m w(\,\cdot\,,t)}_{L^{p'}(\R^n)}
=
\lim_{t\to 0^+} \doublebar{\nabla^m w(\,\cdot\,,t)-\Tr_{m}^- w}_{L^{p'}(\R^n)} 
= 0
.\end{equation*}
Finally, we have that $\M_{\mat A^*}^- w$ exists in the sense of formula~\eqref{eqn:neumann:intro}, and that
\begin{equation*}
\abs{\langle \Tr_{m-1}\varphi,\M_{\mat A^*}^- w\rangle_{\R^n}}
\leq C(0,L^*,p') \doublebar{\Tr_{m-1}\varphi}_{\dot W\!A^{0,p}_{m-1}(\R^n)} \doublebar{\mathcal{A}_2^-(t\nabla^m \partial_t w)}_{L^{p'}(\R^n)}
\end{equation*}
for every $\varphi\in C^\infty_0(\R^\dmn)$.
\end{thm}

\subsection{The bounds (\ref{eqn:S:N:rough:intro}) and~(\ref{eqn:D:N:rough:intro})}
\label{sec:D:S:rough}

In this section we will complete the proof of Theorem~\ref{thm:potentials} by establishing the bounds \cref{eqn:S:N:rough:intro,eqn:D:N:rough:intro}. As in Section~\ref{sec:D:S:regular}, throughout this section we will let $L$ and $\mat A$ be as in Theorem~\ref{thm:potentials}.

We begin with the bound~\eqref{eqn:S:N:rough:intro}.
Let $\arr h\in L^2(\R^n)\cap L^p(\R^n)$ for some $p$ with $p_{0,L^*}^-<p<2$.
By the bound~\eqref{eqn:S:N:rough:+} with $p=2$, we may apply  Lemma~\ref{lem:N<C:vertical} with $u=\s^L_\nabla\arr h$; by Lemma~\ref{lem:N<C:vertical} and formula~\eqref{eqn:S:Newton:rough:dual},
\begin{align*}
\doublebar{\widetilde N_+(\nabla^{m-1}\s^L_\nabla\arr h)}_{L^p(\R^n)}
&\leq C_p\sup_{\arr \Psi}
\frac{\abs{\langle \arr \Psi,\nabla^{m}\s^L_\nabla\arr h\rangle_{\R^\dmn_+}}} {\doublebar{\widetilde{\mathfrak{C}}_1^+(t\,\partial_t\arr\Psi)} _ {L^{p'}(\R^n)}}
\\&= C_p\sup_{\arr \Psi}
\frac{\abs{\langle \Tr_m^- \Pi^{L^*}(\1_+\arr\Psi),\arr h\rangle_{\R^n}}} {\doublebar{\widetilde{\mathfrak{C}}_1^+(t\,\partial_t\arr\Psi)} _ {L^{p'}(\R^n)}}
\end{align*}
where the supremum is over all $\arr\Psi\in L^2(\R^\dmn_+)$ that are supported in a compact subset of $\R^\dmn_+$ and have a weak vertical derivative in $L^2(\R^\dmn_+)$.

By definition of the Newton potential and the Caccioppoli inequality, we have that $\partial_\dmn\Pi^{L^*}(\1_+\arr\Psi)\in \dot W^{m,2}(\R^\dmn_-)$. By Lemma~\ref{lem:slices} and the bound~\eqref{eqn:newton:N:regular}, we have that $\nabla^m \Pi^{L^*}(\1_+\arr\Psi)(\,\cdot\,,t)\in L^{p'}(\R^n)$ for any (hence some) $t<0$. Thus, we may apply Theorem~\ref{thm:BHM17pB:regular:dirichlet} with $w=\Pi^{L^*}(\1_+\arr\Psi)$ and see that
\begin{align*}
{\abs{\langle \Tr_m^- \Pi^{L^*}(\1_+\arr\Psi),\arr h\rangle_{\R^n}}}&\leq 
	C(0,L^*,p')
{\doublebar{\mathcal{A}_2^-(t\nabla^m \partial_t \Pi^{L^*}(\1_+\arr\Psi))}_{L^{p'}(\R^n)} \doublebar{\arr h}_{L^p(\R^n)}} 
.\end{align*}
By formula~\eqref{eqn:Newton:vertical} and  the bound~\eqref{eqn:newton:lusin:rough},
\begin{align*}
\doublebar{\mathcal{A}_2^-(t\nabla^m \partial_t \Pi^{L^*}(\1_+\arr\Psi))}_{L^{p'}(\R^n)}
&=
	\doublebar{\mathcal{A}_2^-(t\nabla^m \Pi^{L^*}(\1_+\partial_t \arr\Psi))}_{L^{p'}(\R^n)}
\\&\leq C(1,L^*,p'){\doublebar{\widetilde{\mathfrak{C}}_1^+(t\,\partial_t\arr\Psi)} _ {L^{p'}(\R^n)}}
.\end{align*}
Thus, if $p_{0,L^*}^-<p<2$, then
\begin{align*}
\doublebar{\widetilde N_+(\nabla^{m-1}\s^L_\nabla\arr h)}_{L^p(\R^n)}
&\leq C(0,L^*,p')\doublebar{\arr h}_{L^p(\R^n)} 
.\end{align*}
By density and Section~\ref{sec:lower}, the bound~\eqref{eqn:S:N:rough:intro} is valid.

Similarly, let $\arr f=\Tr_{m-1} F$ for some $F\in C^\infty_0(\R^\dmn)$. By the bound~\eqref{eqn:D:N:rough:+} with $p=2$,  Lemma~\ref{lem:N<C:vertical}, and formula~\eqref{eqn:D:Newton:dual},  if $1<p<2$ then
\begin{align*}
\doublebar{\widetilde N_+(\nabla^{m-1}\D^{\mat A}\arr f)}_{L^p(\R^n)}
&\leq 
	C_p\sup_{\arr \Psi}
	\frac{\abs{\langle \arr \Psi,\nabla^{m}\D^{\mat A}\arr f\rangle_{\R^\dmn_+}}} {\doublebar{\widetilde{\mathfrak{C}}_1^+(t\,\partial_t\arr\Psi)} _ {L^{p'}(\R^n)}}
\\&=
	C_p\sup_{\arr \Psi}
	\frac{\abs{\langle \M_{\mat A^*}^-\Pi^{L^*}(\1_+\arr \Psi),\arr f\rangle_{\R^n}}} {\doublebar{\widetilde{\mathfrak{C}}_1^+(t\,\partial_t\arr\Psi)} _ {L^{p'}(\R^n)}}
.\end{align*}
By Theorem~\ref{thm:BHM17pB:regular:dirichlet}, formula~\eqref{eqn:Newton:vertical} and the bound~\eqref{eqn:newton:lusin:rough}, if $p_{0,L^*}^-<p<2$ then
\begin{align*}
\abs{\langle \M_{\mat A^*}^-\Pi^{L^*}(\1_+\arr \Psi),\arr f\rangle_{\R^n}}
&\leq 
	C(0,L^*,p')
	\doublebar{\mathcal{A}_2^-(t\nabla^m\partial_t \Pi^{L^*}(\1_+\arr\Psi))}_{L^{p'}(\R^n)} \doublebar{\arr f}_{L^p(\R^n)}
\\&\leq 
	C(0,L^*,p')
	\doublebar{\widetilde{\mathfrak{C}}_1^+(t\,\partial_t\arr\Psi)} _{L^{p'}(\R^n)} \doublebar{\arr f}_{L^p(\R^n)}
.\end{align*}
By density and Section~\ref{sec:lower}, the bound~\eqref{eqn:D:N:rough:intro} is valid. This completes the proof of Theorem~\ref{thm:potentials}.

\section{The Green's formula}
\label{sec:green}

A useful tool in the theory of higher order equations, and one of the reasons layer potentials are of interest, is the Green's formula
\begin{equation}
\label{eqn:green}
\1_+ \nabla^m u=-\nabla^m \D^{\mat A}(\Tr_{m-1}^+  u) + \nabla^m \s^L(\M_{\mat A}^+  u) 
.\end{equation}
This formula is valid for all $u\in \dot W^{m,2}(\R^\dmn_+)$ that satisfy $Lu=0$ in $\R^\dmn_+$. See \cite[Lemma~5.2]{Bar17} or \cite[formula~(2.26)]{BarHM17}. It is also valid if $Lu=0$ in $\R^\dmn_+$, $\mathcal{A}_2^+(t\nabla^m \partial_t u)\in L^2(\R^n)$ and $\nabla^m u(\,\cdot\,,t)\in L^2(\R^n)$ for some $t>0$; see \cite[Theorem~4.3]{BarHM18}. This Green's formula was used in \cite{BarHM18} to establish uniqueness of solutions to the $L^2$ Neumann problem~\eqref{eqn:neumann:regular:2}; the corresponding formula in the lower half space was used to prove Lemma~\ref{lem:lusin:regular:2} above.

In this section, we will show that the Green's formula is still valid if $Lu=0$ in $\R^\dmn_+$, $\mathcal{A}_2^+(t\nabla^m \partial_t u)\in L^p(\R^n)$ and $\sup_{t>0}\doublebar{\nabla^m u(\,\cdot\,,t)}_{L^p(\R^n)}<\infty$ for some $p$ with $p_{1,L^*}^-<p\leq 2$. The Green's formula for such solutions will be used in Section~\ref{sec:neumann} to establish uniqueness of solutions to the Neumann problem~\eqref{eqn:neumann:regular:p:selfadjoint}.

We will begin with some useful auxiliary lemmas. Specifically, recall from Theorem~\ref{thm:BHM17pB:regular:dirichlet} that $\nabla^m w(\,\cdot\,,t)\to \Tr_m w$ as $t\to 0^+$ and $\nabla^m w(\,\cdot\,,t)\to 0$ as $t\to \infty$. We wish to prove a similar result for Neumann boundary values. Our argument will follow the proof of \cite[Lemma~4.2]{BarHM18}.

\begin{lem}\label{lem:lusin:translate}
Let $L$ be an operator of the form~\eqref{eqn:weak} of order~$2m$ associated to bounded coefficients~$\mat A$ that satisfy the ellipticity condition~\eqref{eqn:elliptic}.

Let $p$ satisfy $0<p\leq 2$. Let $j$ be an integer with $0\leq j\leq m$. Let $u\in \dot W^{m,2}_{loc}(\R^\dmn_+)$ be such that $Lu=0$ in $\R^\dmn_+$ and $\mathcal{A}_2^+(t\nabla^j u)\in L^p(\R^n)$. Define $u_\varepsilon(x,t)=u(x,t+\varepsilon)$. If $\varepsilon>0$, then
\begin{equation*}
\doublebar{\mathcal{A}_2^+ (t\nabla^j u_\varepsilon)}_{L^p(\R^n)}
\leq C \doublebar{\mathcal{A}_2^+ (t\nabla^j u)}_{L^p(\R^n)},
\end{equation*}
and
\begin{equation*}
\lim_{\varepsilon\to 0^+} \doublebar{\mathcal{A}_2^+ (t\nabla^j (u-u_\varepsilon))}_{L^p(\R^n)} = \lim_{T\to \infty} \doublebar{\mathcal{A}_2^+ (t\nabla^j u_T)}_{L^p(\R^n)} = 0.
\end{equation*}
\end{lem}

\begin{proof}
We define 
\begin{align*}
\mathcal{A}_f^{\ell} H(x) &= \biggl(\int_{\ell}^\infty \int_{\abs{x-y}<t} \abs{H(y, t)}^2 \frac{dy\,dt}{t^\dmn}\biggr)^{1/2}
,\\
\mathcal{A}_n^{\ell} H(x) &= \biggl(\int_0^{\ell} \int_{\abs{x-y}<t} \abs{H(y, t)}^2 \frac{dy\,dt}{t^\dmn}\biggr)^{1/2}
,\end{align*}
so that
$
\mathcal{A}_2^+ H(x)^2 = \mathcal{A}_f^{\ell}H(x)^2 + \mathcal{A}_n^{\ell}H(x)^2
$.

Let $c>1$ be a constant to be chosen later.
We begin by analyzing $\mathcal{A}_n^{\varepsilon/c} (t\nabla^j u_\varepsilon)$.
Let $\mathcal{G}$ be a grid of pairwise-disjoint open cubes in $\R^n$ of side length $\varepsilon/c$ whose union is almost all of~$\R^n$. Then
\begin{align*}\doublebar{\mathcal{A}_n^{\varepsilon/c}(t\nabla^j u_\varepsilon)}_{L^p(\R^n)}^p
&=
\sum_{Q\in\mathcal{G}}
\int_Q \mathcal{A}_n^{\varepsilon/c}(t\nabla^j u_\varepsilon)(x)^p\,dx
.\end{align*}
By H\"older's inequality,
\begin{align*}\doublebar{\mathcal{A}_n^{\varepsilon/c}(t\nabla^j u_\varepsilon)}_{L^p(\R^n)}^p
&\leq
\sum_{Q\in\mathcal{G}}
\abs{Q}^{1-p/2}
\biggl(\int_Q \mathcal{A}_n^{\varepsilon/c}(t\nabla^j u_\varepsilon)(x)^2\,dx\biggr)^{p/2}
.\end{align*}
By definition of $u_\varepsilon$ and of $\mathcal{A}_n^\ell$,
\begin{align*}\doublebar{\mathcal{A}_n^{\varepsilon/c}(t\nabla^j u_\varepsilon)}_{L^p(\R^n)}^p
&\leq
\sum_{Q\in\mathcal{G}}
\abs{Q}^{1-p/2}
\biggl(\int_Q \int_0^{\varepsilon/c} \!\int_{\abs{x-y}<t} \abs{\nabla^j u(y,t+\varepsilon)}^2\frac{dy\,dt}{t^{n-1}}\,dx\biggr)^{p/2}
.\end{align*}
Changing the order of integration and evaluating the integral~$dx$, we have that
\begin{align*}\doublebar{\mathcal{A}_n^{\varepsilon/c}(t\nabla^j u_\varepsilon)}_{L^p(\R^n)}^p
&\leq
\alpha_n^{p/2}
\sum_{Q\in\mathcal{G}}
\abs{Q}^{1-p/2}
\biggl(\int_0^{\varepsilon/c} \int_{3Q} t\abs{\nabla^j u(y,t+\varepsilon)}^2\,dy\,dt\biggr)^{p/2}
\end{align*}
where $\alpha_n$ is the area of the unit disk in~$\R^n$.

Making a change of variables, we see that
\begin{align*}\doublebar{\mathcal{A}_n^{\varepsilon/c}(t\nabla^j u_\varepsilon)}_{L^p(\R^n)}^p
&\leq
\alpha_n^{p/2}
\sum_{Q\in\mathcal{G}}
\abs{Q}^{1-p/2}
\biggl(\int_{\varepsilon}^{\varepsilon+\varepsilon/c} \int_{3Q} (t-\varepsilon)\abs{\nabla^j u(y,t)}^2\,dy\,dt\biggr)^{p/2}
.\end{align*}
Let $c=2\sqrt{n}=\sqrt{4n}$. If $x\in Q$, $y\in 3Q$, and $t\in (\varepsilon,\varepsilon+\varepsilon/c)$, then $\abs{x-y}<2\sqrt{n}\,\ell(Q)=\varepsilon<t$. Thus, if $x\in Q$ then
\begin{multline*} \biggl(\int_{\varepsilon}^{\varepsilon+\varepsilon/\sqrt{4n}} \int_{3Q} (t-\varepsilon)\abs{\nabla^j u(y,t)}^2\,dy\,dt\biggr)^{p/2}
\\\leq
\biggl((\varepsilon+\varepsilon/\sqrt{4n})^{n} \int_{\varepsilon}^{\varepsilon+\varepsilon/\sqrt{4n}} \int_{\abs{x-y}<t} \abs{t\nabla^j u(y,t)}^2\,\frac{dy\,dt}{t^{n+1}}\biggr)^{p/2}
.\end{multline*}
The right hand side is at most 
\begin{equation*}(C_n \abs{Q})^{p/2} \min\bigl(\mathcal{A}_n^{\varepsilon+\varepsilon/\sqrt{4n}}(t\nabla^j u)(x), \mathcal{A}_f^{\varepsilon}(t\nabla^j u)(x)\bigr)^p
\end{equation*}
For ease of notation we replace ${\varepsilon+\varepsilon/\sqrt{4n}}$ with $2\varepsilon$.
Thus,
\begin{align*}\doublebar{\mathcal{A}_n^{\varepsilon/\sqrt{4n}}(t\nabla^j u_\varepsilon)}_{L^p(\R^n)}^p
&\leq
C_n^{p/2}
\sum_{Q\in\mathcal{G}} 
\int_Q\min\bigl(\mathcal{A}_n^{2\varepsilon} (t\nabla^j u)(x), \mathcal{A}_f^{\varepsilon}(t\nabla^j u)(x)\bigr)^p\,dx
.\end{align*}
Summing over $Q$, we have that 
\begin{equation}
\label{eqn:green:proof:n}
\doublebar{\mathcal{A}_n^{\varepsilon/\sqrt{4n}}(t\nabla^j u_\varepsilon)}_{L^p(\R^n)}
\leq
C_n
\min\bigl( \doublebar{\mathcal{A}_n^{2\varepsilon}(t\nabla^j u)}_{L^p(\R^n)}, 
\doublebar{\mathcal{A}_f^{\varepsilon}(t\nabla^j u)}_{L^p(\R^n)} \bigr)
.\end{equation}

We now turn to $\mathcal{A}_f^{\varepsilon/\sqrt{4n}}$.
By definition of $u_\varepsilon$,
\begin{align*}
\mathcal{A}_f^{\varepsilon/\sqrt{4n}}(t\nabla^j(u-u_\varepsilon))(x)
&= 
\biggl(\int_{\varepsilon/\sqrt{4n}}^\infty \int_{\abs{x-y}<t} \abs{\nabla^j(u(y,t)-u(y,t+\varepsilon))}^2 \frac{dy\,dt}{t^{n-1}}\biggr)^{1/2}
\\&= 
\biggl(\int_{\varepsilon/\sqrt{4n}}^\infty \int_{\abs{x-y}<t} \abs[bigg]{\int_t^{t+\varepsilon} \nabla^j\partial_s u(y,s)\,ds}^2 \frac{dy\,dt}{t^{n-1}}\biggr)^{1/2}
.\end{align*}
Applying H\"older's inequality and changing the order of integration, we have that
\begin{align*}
\mathcal{A}_f^{\varepsilon/\sqrt{4n}}(t\nabla^j(u-u_\varepsilon))(x)
&\leq
\biggl(\int_{\varepsilon/\sqrt{4n}}^\infty \int_{\abs{x-y}<t} 
\varepsilon \int_t^{t+\varepsilon} \abs{\nabla^j\partial_s u(y,s)}^2\,ds \frac{dy\,dt}{t^{n-1}}\biggr)^{1/2}
\\&\leq
C_n\biggl(\varepsilon^2
\int_{\varepsilon/\sqrt{4n}}^\infty 
\int_{\abs{x-y}<s} 
\abs{\nabla^j\partial_s u(y,s)}^2 \frac{dy\,ds}{s^{n-1}}\biggr)^{1/2}
.\end{align*}
By the Caccioppoli inequality,
\begin{align*}
\mathcal{A}_f^{\varepsilon/\sqrt{4n}}(t\nabla^j(u-u_\varepsilon))(x)
&\leq
C\biggl(
\int_{\varepsilon/\sqrt{16n}}^\infty \frac{\varepsilon^2}{s^2}
\int_{\abs{x-y}<2s} 
\abs{\nabla^j u(y,s)}^2 \frac{dy\,ds}{s^{n-1}}\biggr)^{1/2}
.\end{align*}

Now, define 
\begin{equation*}\mathcal{A}_2^{r} H(x) =  \biggl(\int_0^\infty \int_{\abs{x-y}<rt} \abs{H(y, t)}^2 \frac{dy\,dt}{t^\dmn}\biggr)^{1/2}\end{equation*}
for any $r>0$, so that $\mathcal{A}_2^+ H=\mathcal{A}_2^1 H$. It is well known (see \cite[Proposition~4]{CoiMS85} or \cite[Theorem~3.4]{CalT75}) that if $0<p<\infty$, then $\doublebar{\mathcal{A}_2^2 H}_{L^p(\R^n)} \leq C_p \doublebar{\mathcal{A}_2^1 H}_{L^p(\R^n)}$. Thus,
\begin{equation}
\label{eqn:green:proof:f}
\doublebar{\mathcal{A}_f^{\varepsilon/\sqrt{4n}}(t\nabla^j(u-u_\varepsilon))}_{L^p(\R^n)}
\leq C_p
\doublebar{\mathcal{A}_f^{\varepsilon/\sqrt{16n}}(\varepsilon\nabla^j u)}_{L^p(\R^n)}.\end{equation}

The bound $\doublebar{\mathcal{A}_2^+(t\nabla^j u_\varepsilon)}_{L^p(\R^n)}\leq C_p \doublebar{\mathcal{A}_2^+(t\nabla^j u)}_{L^p(\R^n)}$ follows from the bounds \cref{eqn:green:proof:n,eqn:green:proof:f}.
We now use these bounds to bound $\mathcal{A}_2^+(t\nabla^j u_T)$ as $T\to \infty$ and $\mathcal{A}_2^+(t\nabla^j (u-u_\varepsilon))$ as $\varepsilon\to 0^+$.

First, by definition of $\mathcal{A}_n^\ell$ and $\mathcal{A}_f^\ell$,
\begin{equation*}\doublebar{\mathcal{A}_2^+(t\nabla^j u_T)}_{L^p(\R^n)}
\leq \doublebar{\mathcal{A}_n^{T/\sqrt{4n}}(t\nabla^j u_T)}_{L^p(\R^n)}
+ \doublebar{\mathcal{A}_f^{T/\sqrt{4n}}(t\nabla^j u_T)}_{L^p(\R^n)}.
\end{equation*}
Next, by the bounds \cref{eqn:green:proof:n,eqn:green:proof:f},
\begin{equation*}\doublebar{\mathcal{A}_2^+(t\nabla^j u_T)}_{L^p(\R^n)}
\leq C_p\doublebar{\mathcal{A}_f^{T/\sqrt{16n}}(t\nabla^j u)}_{L^p(\R^n)}
.\end{equation*}
If $\mathcal{A}_2^+ (t\nabla^j u)(x)<\infty$, then $\mathcal{A}_f^{T/\sqrt{16n}} (t\nabla^j u)(x)\to 0$ as $T\to \infty$, and so by the dominated convergence theorem, if $\mathcal{A}_2^+ (t\nabla^j u)\in L^p(\R^n)$, then $\mathcal{A}_f^{T/\sqrt{16n}} (t\nabla^j u)\to 0$ in $L^p(\R^n)$ as $T\to \infty$. Thus, $\doublebar{\mathcal{A}_2^+(t\nabla^j u_T)}_{L^p(\R^n)}\to 0$ as $T\to \infty$, as desired.

We now turn to $u-u_\varepsilon$. By definition of $\mathcal{A}_n^\ell$ and $\mathcal{A}_f^\ell$,
\begin{align*}\doublebar{\mathcal{A}_2^+(t\nabla^j (u-u_\varepsilon))}_{L^p(\R^n)}
&\leq \doublebar{\mathcal{A}_n^{\varepsilon/\sqrt{4n}}(t\nabla^j u)}_{L^p(\R^n)}
+  \doublebar{\mathcal{A}_n^{\varepsilon/\sqrt{4n}}(t\nabla^j u_\varepsilon)}_{L^p(\R^n)}
\\&\qquad+ \doublebar{\mathcal{A}_f^{\varepsilon/\sqrt{4n}}(t\nabla^j (u-u_\varepsilon))}_{L^p(\R^n)}.
\end{align*}
By the bounds \cref{eqn:green:proof:n,eqn:green:proof:f},
\begin{equation*}\doublebar{\mathcal{A}_2^+(t\nabla^j (u-u_\varepsilon))}_{L^p(\R^n)}
\leq 
C\doublebar{\mathcal{A}_n^{2\varepsilon}(t\nabla^j u)}_{L^p(\R^n)}
+ C_p\doublebar{\mathcal{A}_f^{\varepsilon/\sqrt{16n}} (\varepsilon\nabla^j u)}_{L^p(\R^n)}.
\end{equation*}
Both terms converge to zero by the dominated convergence theorem and the proof is complete.
\end{proof}

Combining Lemma~\ref{lem:lusin:translate} with Theorem~\ref{thm:BHM17pB:regular:dirichlet} (or, for more notational convenience, \cite[Theorem~6.2]{BarHM19B}) yields the following corollary.

\begin{cor}\label{cor:neumann:limit} 
Let $L$ be an operator of the form~\eqref{eqn:weak} of order~$2m$ associated to bounded $t$-independent coefficients~$\mat A$ that satisfy the ellipticity condition~\eqref{eqn:elliptic}.

Suppose that $w\in \dot W^{m,2}_{loc}(\R^\dmn_+)$ satisfies $Lw=0$ in $\R^\dmn_+$, that $\mathcal{A}_2^+(t\nabla^m\partial_t w)\in L^p(\R^n)$ for some $p$ with $1<p\leq 2$, and that $\nabla^m w(\,\cdot\,,t)\in L^p(\R^n)$ for some $t>0$.

Let $w_\varepsilon(x,t)=w(x,t+\varepsilon)$. Then
\begin{equation*}\lim_{T\to \infty} \doublebar{\M_{\mat A}^+ w_T}_{(\dot W\!A^{0,p'}_{m-1}(\R^n))^*} = 0,
\quad
\lim_{\varepsilon \to 0^+} \doublebar{\M_{\mat A}^+ (w-w_\varepsilon)}_{(\dot W\!A^{0,p'}_{m-1}(\R^n))^*} = 0.\end{equation*}
\end{cor}

We are now in a position to prove the Green's formula.

\begin{thm}\label{thm:green}
Let $L$ be an operator of the form~\eqref{eqn:weak} of order~$2m$ associated to bounded $t$-independent coefficients~$\mat A$ that satisfy the ellipticity condition~\eqref{eqn:elliptic}.

Let $p$ satisfy $p_{1,L^*}^-<p\leq 2$, where $p_{1,L^*}^-$ is as in formula~\eqref{eqn:Meyers:p-}.
Suppose that $w\in \dot W^{m,2}_{loc}(\R^\dmn_+)$ satisfies $Lw=0$ in $\R^\dmn$,
$\mathcal{A}_2^+(t\nabla^m\partial_t w)\in L^p(\R^n)$, and $\nabla^m w(\,\cdot\,,t)\in L^p(\R^n)$ for some $t>0$. 

Then we have the Green's formula
\begin{equation*}\1_+\nabla^m w=-\nabla^m\D^{\mat A}(\Tr_{m-1}^+ w)+\nabla^m\s^L(\M_{\mat A}^+ w).\end{equation*}
\end{thm}

\begin{proof}

Let $w_\varepsilon(x,t)=w(x,t+\varepsilon)$, and let $w_{\varepsilon,T}=w_\varepsilon-w_T$. If $\mat A$ is $t$-independent then $Lw_{\varepsilon,T}=0$ in $\R^\dmn_+$ for any $T>\varepsilon>0$. 
By Lemma~\ref{lem:lebesgue:lusin} or \cite[Remark~5.3]{BarHM19B}, if $T>\varepsilon>0$ then $w_{\varepsilon,T}\in \dot W^{m,2}(\R^\dmn_+)$.

Recall that formula~\eqref{eqn:green} is valid for all solutions in $\dot W^{m,2}(\R^\dmn_+)$. Thus, we have that
\begin{equation*}\1_+\nabla^mw_{\varepsilon,T} = -\nabla^m \D^{\mat A}(\Tr_{m-1}^+ w_{\varepsilon,T}) + \nabla^m \s^L(\M_{\mat A}^+ w_{\varepsilon,T}).\end{equation*}

Let $B=B((x_0,t_0),\abs{t_0}/2)$ be a Whitney ball in $\R^\dmn_\pm$. By Theorem~\ref{thm:BHM17pB:regular:dirichlet}, we have that $\Tr_{m-1}^+ w_{\varepsilon,T} \to \Tr_{m-1}^+ w$ in $\dot W\!A^{1,p}_{m-1}(\R^n)$ as $\varepsilon\to 0^+$ and $T\to \infty$, and by Corollary~\ref{cor:neumann:limit}, $\M_{\mat A}^+ w_{\varepsilon,T} \to \M_{\mat A}^+ w$ in $(\dot W\!A^{0,p'}_{m-1}(\R^n))^*$ as $\varepsilon\to 0^+$ and $T\to \infty$. By the bounds~\cref{eqn:D:N:intro,eqn:S:N:intro} established in Section~\ref{sec:D:S:regular}, we have that
\begin{equation*}-\nabla^m \D^{\mat A}(\Tr_{m-1}^+ w_{\varepsilon,T}) + \nabla^m \s^L(\M_{\mat A}^+ w_{\varepsilon,T})
\to 
-\nabla^m \D^{\mat A}(\Tr_{m-1}^+ w) + \nabla^m \s^L(\M_{\mat A}^+ w)\end{equation*}
in $L^2(B)$ as $T\to\infty$ and $\varepsilon\to 0^+$. 

Because $\mathcal{A}_2^+(t\nabla^m\partial_t w)\in L^p(\R^n)$, we have that $w_\varepsilon\to w$ as $\varepsilon\to 0^+$ in $\dot W^{m,2}(B)$. 
By Theorem~\ref{thm:BHM17pB:regular:dirichlet}, $\nabla^m w_T(\,\cdot\,,t)\to 0$ in $L^p(\R^n)$ for any fixed $t>0$, uniformly for $t$ in $(-3\abs{t_0}/2,3\abs{t_0}/2)$. Therefore, $w_T\to 0$  in $\dot W^{m,p}(B)$; by the bound~\eqref{eqn:Meyers}, $w_T\to 0$  in $\dot W^{m,2}(B)$ as $T\to \infty$. 

Thus, taking appropriate limits yields the Green's formula, as desired.
\end{proof}

\section{The Neumann problem}
\label{sec:neumann}

In this section we will prove Theorem~\ref{thm:neumann:selfadjoint}, that is, will establish well posedness of the Neumann problem with boundary data in $L^p(\R^n)$ for operators with bounded elliptic $t$-independent self-adjoint coefficients.

Our proof of Theorem~\ref{thm:neumann:selfadjoint} will be based on a duality argument. That is, we will show that well posedness of the Neumann problem with boundary data in $\dot W^{-1,p'}(\R^n)$ implies well posedness with boundary data in $L^p(\R^n)$ for adjoint coefficients; as well posedness of the subregular Neumann problem was established in \cite{Bar19p}, this implies well posedness of the $L^p$ Neumann problem.

We begin by precisely stating the well posedness result of \cite{Bar19p}.

\begin{thm}[\cite{Bar19p}]\label{thm:neumann:subregular}
Let $L$ and $\mat A$ satisfy the conditions given in  Theorem~\ref{thm:neumann:selfadjoint}.

Then there is some $\varepsilon_1>0$, depending only on the standard parameters $\dmnMinusOne$,~$m$, $\lambda$, and~$\doublebar{\mat A}_{L^\infty(\R^n)}$, with the following significance. If 
\begin{equation}\label{eqn:p:subregular}
\max\biggl(0,\frac{1}{2}-\frac{1}{\dmnMinusOne}-\varepsilon_1\biggr) < \frac{1}{p'}\leq \frac{1}{2},\end{equation}
then for every $\arr h$ in $\dot W^{-1,2}(\R^n)\cap \dot W^{-1,p'}(\R^n)$, there is a solution~$v$, unique up to adding polynomials of degree~$m-2$, to the subregular Neumann problem 
\begin{equation}
\label{eqn:neumann:rough:2q}
\left\{\begin{gathered}\begin{aligned}
L^*v&=0 \text{ in }\R^\dmn_+
,\\
\M_{\mat A^*}^+ v &\owns \arr h,
\end{aligned}\\
\doublebar{\mathcal{A}_2^+(t\nabla^m v)}_{L^2(\R^n)} + \doublebar{\widetilde N_+(\nabla^{m-1}v)}_{L^2(\R^n)}
\leq C_2\doublebar{\arr h}_{\dot W^{-1,2}(\R^\dmnMinusOne)}
,\\
\doublebar{\mathcal{A}_2^+(t\nabla^m v)}_{L^{p'}(\R^n)} + \doublebar{\widetilde N_+(\nabla^{m-1}v)}_{L^{p'}(\R^n)}
\leq C_{p'}\doublebar{\arr h}_{\dot W^{-1,p'}(\R^\dmnMinusOne)}
.\end{gathered}\right.\end{equation}
The numbers $C_2$ and $C_{p'}$ depend only on $\dmnMinusOne$, $m$, $\lambda$, $\doublebar{\mat A}_{L^\infty(\R^n)}$, and~$p'$.
\end{thm}
We remark that the ${p'}=2$ case, like the $L^2$ Neumann problem~\eqref{eqn:neumann:regular:2}, is from \cite{BarHM18,BarHM18p}. Here $\M_{\mat A^*}^+ v$ is as given in \cite[Section~2.3.2]{BarHM19B}.

If $\mat A$ is self-adjoint, then $\mat A=\mat A^*$ and $L=L^*$; however, we have phrased the problem~\eqref{eqn:neumann:rough:2q} in terms of $\mat A^*$ and $L^*$ for ease of notation for duality arguments. 
We now state our duality theorem; Theorem~\ref{thm:neumann:selfadjoint} will follow easily from Theorem~\ref{thm:neumann:p:regular}.

\begin{thm}\label{thm:neumann:p:regular}
Suppose that $L$ is an elliptic operator of the 
form~\eqref{eqn:weak} 
of order~$2m$ associated with coefficients $\mat A$ that are bounded, $t$-independent in the sense of formula~\eqref{eqn:t-independent}, and satisfy the ellipticity condition~\eqref{eqn:elliptic}.

Let $p$ and $p'$ satisfy $p_{1,L^*}^-<p<2$ and $1/p+1/p'=1$, where $p_{1,L^*}^-$ is as in formulas \cref{eqn:Meyers:p-,eqn:Meyers}. 
Suppose that for every $\arr h\in \dot W^{-1,2}(\R^n)\cap \dot W^{-1,p'}(\R^n)$ there is a unique solution $v$ to the Neumann problem~\eqref{eqn:neumann:rough:2q} for $L^*$.

Then for every $\arr g\in L^p(\R^n)$, there is a solution $w$, unique up to adding polynomials of degree at most $m-1$, to the $L^p$-Neumann problem
\begin{equation}
\label{eqn:neumann:regular:p}
\left\{\begin{gathered}\begin{aligned}
Lv&=0 \text{ in }\R^\dmn_+
,\\
\M_{\mat A}^+ v &\owns \arr g,
\end{aligned}\\
\doublebar{\mathcal{A}_2^+(t\nabla^m \partial_t w)}_{L^p(\R^n)} + \doublebar{\widetilde{N}_+(\nabla^{m}w)}_{L^p(\R^n)}
\leq C_p\doublebar{\arr g}_{L^p(\R^\dmnMinusOne)}
\end{gathered}\right.\end{equation}
where $C_p$ depends only on $p$, $n$, $m$, $\lambda$, $\doublebar{\mat A}_{L^\infty(\R^n)}$, the number $c(1,L^*,p',2)$ in formula~\eqref{eqn:Meyers}, and the constants $C_2$ and $C_{p'}$ in the problem~\eqref{eqn:neumann:rough:2q}.

\end{thm}

\begin{proof}
Fix some such $p$ and~$p'$. We will use the method of layer potentials of \cite{Ver84,BarM13,BarM16A}, specifically as formulated in \cite{Bar17}.

We let $\XX^\pm_p$ and $\widetilde \XX^\pm_{p'}$ be the space of all equivalence classes of functions such that the appropriate norm
\begin{align*}
\doublebar{w}_{\XX^\pm_p} &= \doublebar{\widetilde N_\pm(\nabla^{m} w)}_{L^p(\R^n)}  +\doublebar{\mathcal{A}_2^\pm(t\nabla^m \partial_t w)}_{L^p(\R^n)} 
,\\
\doublebar{v}_{\widetilde \XX^\pm_{p'}} &= \doublebar{\widetilde N_\pm(\nabla^{m-1} v)}_{L^{p'}(\R^n)}  +\doublebar{\mathcal{A}_2^\pm(t\nabla^m v)}_{L^{p'}(\R^n)} 
\end{align*}
is finite.

We define the following function spaces. 
\begin{align*}
\YY^\pm&=\{w_p+w_2:w_p\in \XX^\pm_p,\> w_2\in \XX^\pm_2,\>Lw_p=Lw_2=0\text{ in }\R^\dmn_\pm\}
,\\
\widetilde\YY^\pm&=\{v\in \widetilde\XX^\pm_{p'}\cap \widetilde\XX^\pm_2:L^*v=0\text{ in }\R^\dmn_\pm\}
,\\
\DD&= \dot W\!A^{1,p}_{m-1}(\R^n)+ \dot W\!A^{1,2}_{m-1}(\R^n)
,\\
\widetilde \DD&= \dot W\!A^{0,{p'}}_{m-1}(\R^n)\cap \dot W\!A^{0,2}_{m-1}(\R^n)
,\\
\NN &= (\dot W\!A^{0,{p'}}_{m-1}(\R^n))^*+(\dot W\!A^{0,2}_{m-1}(\R^n))^*
,\\
\widetilde \NN &= (\dot W\!A^{1,p}_{m-1}(\R^n))^* \cap (\dot W\!A^{1,2}_{m-1}(\R^n))^*
.\end{align*}
We will be interested in a family of norms on these function spaces. 
For each number $\delta>0$, let
\begin{align*}
\doublebar{w}_{ \YY_\delta^\pm} &= \inf\Bigl\{\doublebar{w_p}_{ \XX_p^\pm} + \frac{1}{\delta} \doublebar{w_2}_{ \XX_2^\pm}:w=w_{p}+w_2, \> Lw_{p}=Lw_2=0\Bigr\}
,\\
\doublebar{\arr \varphi}_{\DD_\delta} 
&= 
	\inf\Bigl\{\doublebar{\arr \varphi_p}_{\dot W\!A^{1,p}_{m-1}(\R^n)} + \frac{1}{\delta}\doublebar{\arr \varphi_2}_{\dot W\!A^{1,2}_{m-1}(\R^n)}
	:\arr \varphi=\arr \varphi_p+\arr \varphi_2\Bigr\}
,\\
\doublebar{\arr G}_{\NN_\delta}
&=
	\inf \Bigl\{\doublebar {\arr G_p}_{(\dot W\!A^{0,{p'}}_{m-1}(\R^n))^*} + \frac{1}{\delta} \doublebar {\arr G_2}_{(\dot W\!A^{0,2}_{m-1}(\R^n))^*}:\arr G=\arr G_p+\arr G_2 \Bigr\}
,\\
\doublebar{v}_{\widetilde\YY^\pm_\delta} &= \doublebar{v}_{\widetilde\XX^\pm_{p'}}+\delta \doublebar{v}_{\widetilde\XX^\pm_2}
,\\
\doublebar{\arr f}_{\widetilde \DD_\delta} 
&= 
	\doublebar{\arr f}_{\dot W\!A^{0,{p'}}_{m-1}(\R^n)} + \delta\doublebar{\arr f}_{\dot W\!A^{0,2}_{m-1}(\R^n)}
,\\
\doublebar{\arr H}_{\widetilde\NN_\delta}
&=
	\doublebar {\arr H}_{(\dot W\!A^{1,p}_{m-1}(\R^n))^*} + \delta\doublebar {\arr H}_{(\dot W\!A^{1,2}_{m-1}(\R^n))^*}
.\end{align*}
Then $\NN_\delta=(\widetilde \DD_\delta)^*$ and $\widetilde \NN_\delta=(\DD_\delta)^*$. See \cite[formula~(1.3) and Theorem~1.7]{LiuV69}.

By Theorems~\ref{thm:BHM17pB:rough:neumann}, \ref{thm:BHM17pB:regular:neumann}, \ref{thm:BHM17pB:rough:dirichlet} and~\ref{thm:BHM17pB:regular:dirichlet}, the operators
\begin{align*}
\Tr_{m-1}^\pm&:\YY^\pm_\delta\to \DD_\delta
,&
\Tr_{m-1}^\pm&:\widetilde \YY^\pm_\delta\to \widetilde \DD_\delta
,&
\M_{\mat A}^\pm&:\YY^\pm_\delta\to \NN_\delta
,&
\M_{\mat A^*}^\pm&:\widetilde \YY^\pm_\delta\to \widetilde \NN_\delta
\end{align*}
are bounded with bounds depending only on $p$ and the standard parameters, and in particular, not on~$\delta$ provided $\delta>0$. By Theorem~\ref{thm:potentials} and the bounds~\cref{eqn:D:N:rough:+,eqn:S:N:rough:+,eqn:D:lusin:rough:+,eqn:S:lusin:rough:+} we have that the double and single layer potentials are bounded
\begin{align*}
\D^{\mat A}&:\DD_\delta\to \YY^\pm_\delta
,&
\D^{\mat A^*}&:\widetilde\DD_\delta\to \widetilde\YY^\pm_\delta
,&
\s^L&:\NN_\delta\to \YY^\pm_\delta
,&
\s^{L^*}&:\widetilde\NN_\delta\to \widetilde\YY^\pm_\delta
\end{align*}
with bounds independent of~$\delta$.

By \cite[Theorem~4.3]{BarHM18}, and by Theorem~\ref{thm:green} and Section~\ref{sec:lower}, we have that if $v\in \widetilde \YY^\pm_\delta$ and $w\in\YY^\pm_\delta$, then the Green's formulas
\begin{align*}
\1_\pm\nabla^m v&=\mp \nabla^m \D^{\mat A^*}(\Tr_{m-1}^\pm v)+\nabla^m \s^{L^*}(\M_{\mat A^*}^\pm v)
,\\
\1_\pm\nabla^m w&=\mp \nabla^m \D^{\mat A^*}(\Tr_{m-1}^\pm w)+\nabla^m \s^{L^*}(\M_{\mat A^*}^\pm w)
\end{align*}
are valid.

Finally, the jump relations 
\begin{align}
\label{eqn:jump:dirichlet}
\Tr_{m-1}^+\D^{\mat A}\arr f-\Tr_{m-1}^-\D^{\mat A} \arr f&=-\arr f
,&
\Tr_{m-1}^+\s^L\arr g-\Tr_{m-1}^-\s^L \arr g&=\arr 0
,\\
\label{eqn:jump:neumann}
\M_{\mat A}^+\D^{\mat A}\arr f+\M_{\mat A}^-\D^{\mat A} \arr f&\owns\arr 0
,&
\M_{\mat A}^+\s^L\arr g+\M_{\mat A}^-\s^L \arr g&= \arr g
\end{align}
of \cite[Conditions 6.19--6.22]{Bar17} are valid for all $\arr f\in \dot W\!A^{1/2,2}_{m-1}(\R^n)$ and all $\arr g \in (\dot W\!A^{1/2,2}_{m-1}(\R^n))^*$; see \cite[Lemma~5.4]{Bar17}. 
By density, the relations~\cref{eqn:jump:dirichlet,eqn:jump:neumann} are true for all $\arr f$ in $\DD_\delta$ or $\widetilde \DD_\delta$ and all $\arr g$ in $\NN_\delta$ or $\widetilde \NN_\delta$.

Thus, \cite[Conditions 6.14--6.22]{Bar17} are valid for the spaces $\widetilde \YY^\pm_\delta$, $\widetilde \DD_\delta$ and $\widetilde \NN_\delta$, and so by \cite[Theorems 6.23 and~6.24]{Bar17} and well posedness of the Neumann problem~\eqref{eqn:neumann:rough:2q}, we have that $\M_{\mat A^*}^+ \D^{\mat A^*}$ is invertible $\widetilde \DD_\delta\to \widetilde \NN_\delta$, and that $\|(\M_{\mat A^*}^+\D^{\mat A^*})^{-1}\|$ is independent of~$\delta$. By duality (see \cite[Lemma~5.3]{Bar17}), $\M_{\mat A}^+\D^{\mat A}$ is invertible $\DD_\delta \to  \NN_\delta$. Furthermore, the norm is independent of~$\delta$, and the value of $(\M_{\mat A}^+\D^{\mat A})^{-1}\arr g$ is independent of~$\delta$.

Let $w=\D^{\mat A}((\M_{\mat A}^+\D^{\mat A})^{-1}\arr G)$, $\arr G\in L^p(\R^n)\subset  \NN_\delta$.

Then $w\in \YY^+_\delta$ and so $w=w_p^\delta+w_2^\delta$ for some $w_p^\delta\in \XX_p^+$, $w_2^\delta\in\XX_2^+$ with $Lw_p^\delta=Lw_2^\delta=0$ in $\R^\dmn_+$ and with 
\begin{equation*}
\doublebar{ w_p^\delta}_{\XX^+_p} 
	+ \frac{1}{\delta}\doublebar{w_2^\delta}_{\XX^+_2}
\leq C \doublebar{\arr G}_{\widetilde \NN_\delta}
\leq C\doublebar{\arr G}_{L^p(\R^n)}.\end{equation*}
Taking the limit as $\delta\to 0^+$, we see that $w_2^\delta\to 0$ in $\dot W^{m,2}_{loc}(\R^\dmn_+)$. Thus $w=\lim_{\delta \to 0^+} w_p^\delta$ and so 
\begin{equation*}\doublebar{\widetilde N_\pm(\nabla^{m} w) +\mathcal{A}_2^\pm(t\nabla^m \partial_t w)}_{L^{p'}(\R^n)} 
\leq C\doublebar{\arr G}_{L^p(\R^n)}\end{equation*}
as desired.

Thus, solutions to the Neumann problem~\eqref{eqn:neumann:regular:p} exist. We have seen that $\M_{\mat A}^+\D^{\mat A}$ is one-to-one on $\DD=\dot W\!A^{1,p}_{m-1}(\R^n)+\dot W\!A^{1,2}_{m-1}(\R^n)$, and so it is also one-to-one on the subspace $\dot W\!A^{1,p}_{m-1}(\R^n)$. The Green's formula of Theorem~\ref{thm:green} allows us to apply \cite[Theorem~6.13]{Bar17} to see that solutions to the problem~\eqref{eqn:neumann:regular:p} are unique, as desired.
\end{proof}

We conclude the paper by proving Theorem~\ref{thm:neumann:selfadjoint}.
\begin{proof}[Proof of Theorem~\ref{thm:neumann:selfadjoint}]
The ellipticity condition~\eqref{eqn:elliptic:slices} in Theorem~\ref{thm:neumann:selfadjoint} implies that the condition~\eqref{eqn:elliptic} in Theorem~\ref{thm:neumann:p:regular} is valid. Thus, if $L$ and $\mat A$ satisfy the conditions of Theorem~\ref{thm:neumann:selfadjoint}, then they satisfy the conditions of Theorems~\ref{thm:neumann:subregular} and~\ref{thm:neumann:p:regular}.

There is a $\varepsilon>0$, depending only on $\dmnMinusOne$ and the number $\varepsilon_1$ in formula~\eqref{eqn:p:subregular}, such that if $p$ satisfies the bound~\eqref{eqn:p:regular}, then $p'$ satisfies the bound~\eqref{eqn:p:subregular}. Thus, if $\varepsilon>0$ is small enough and the conditions of Theorem~\ref{thm:neumann:selfadjoint} are satisfied, then the subregular Neumann problem~\eqref{eqn:neumann:rough:2q} is well posed.

Recall from formula~\eqref{eqn:Meyers:bound} that there is some $\widetilde\varepsilon>0$ depending on the standard parameters such that 
\begin{equation*}
p_{1,L^*}^-\leq \max\biggl(1,\frac{2n}{n+2}-\widetilde\varepsilon\biggr)
.\end{equation*}
By Remark~\ref{rmk:Meyers}, if $\max\bigl(1,\frac{2n}{n+2}-\widetilde\varepsilon\bigr)<p<2$ then $c(1,L^*,p,2)$ depends only on $p$ and the standard parameters.

Thus, if $\varepsilon$ is small enough and $p$ satisfies the condition~\eqref{eqn:p:regular} of Theorem~\ref{thm:neumann:selfadjoint}, then $p$ and $L$ also satisfy the conditions of Theorem~\ref{thm:neumann:p:regular}. Thus, the Neumann problem~\eqref{eqn:neumann:regular:p} (or~\eqref{eqn:neumann:regular:p:selfadjoint}) is well posed.
\end{proof}


\newcommand{\etalchar}[1]{$^{#1}$}
\providecommand{\bysame}{\leavevmode\hbox to3em{\hrulefill}\thinspace}
\providecommand{\MR}{\relax\ifhmode\unskip\space\fi MR }
\providecommand{\MRhref}[2]{%
  \href{http://www.ams.org/mathscinet-getitem?mr=#1}{#2}
}
\providecommand{\href}[2]{#2}

\end{document}